\documentclass[10pt,twoside,a4paper,reqno]{amsart}

\usepackage{amsfonts,amsthm,latexsym,mathtools,amsmath,amssymb,enumitem,appendix,color,hyperref}
\usepackage[capitalize]{cleveref}
\usepackage{float}

\usepackage{graphicx}
\DeclareGraphicsExtensions{.pdf,.png}
\graphicspath{{./figures/}{./}}
\makeatletter
\@namedef{subjclassname@2020}{\textup{2020} Mathematics Subject Classification}
\makeatother

\theoremstyle{plain}
\newtheorem{theo+}{Theorem}
\numberwithin{theo+}{section}
\newtheorem{prop+}[theo+]{Proposition}
\newtheorem{coro+}[theo+]{Corollary}
\newtheorem{lemm+}[theo+]{Lemma}
\newtheorem{conjecture}[theo+]{Conjecture}

\theoremstyle{definition}
\newtheorem{defi+}[theo+]{Definition}
\newtheorem{problem}[theo+]{Problem}

\newtheorem*{pb-prime}{Problem 1$\mathbb{'}$}

\newtheorem{not+}[theo+]{Notation}

\theoremstyle{remark}
\newtheorem{rema+}[theo+]{Remark}
\newtheorem{example}[theo+]{Example}

\newenvironment{theorem}{\begin{theo+}}{\end{theo+}}
\newenvironment{proposition}{\begin{prop+}}{\end{prop+}}
\newenvironment{corollary}{\begin{coro+}}{\end{coro+}}
\newenvironment{lemma}{\begin{lemm+}}{\end{lemm+}}
\newenvironment{remark}{\begin{rema+}}{\end{rema+}}
\newenvironment{definition}{\begin{defi+}}{\end{defi+}}
\newenvironment{notation}{\begin{not+}}{\end{not+}}

\newcommand{\defin}[1]{%
\relax\ifmmode

\textcolor{blue}{#1}%
\else \textcolor{blue}{\emph{#1}}%
\fi%
}

\newcommand{\setC}{\mathbb{C}}
\newcommand{\setR}{\mathbb{R}}
\newcommand{\setN}{\mathbb{N}}

\newcommand{\setRS}{\bar{\mathbb{C}}} 

\newcommand{\bC}{\setC} 
\newcommand{\bR}{\setR}

\newcommand{\flex}{\mathcal{I}} 

\newcommand{\RR}{R} 
\newcommand{\zvec}{\mathbf{z}}

\renewcommand{\Re}{\operatorname{Re}}
\renewcommand{\Im}{\operatorname{Im}}

\newcommand{\infl}{\mathfrak{I}}

\newcommand{\diffz}{{\frac{d}{dz}}}

\newcommand{\minvset}[1]{{{\mathrm{M}}_{#1}^T}}

\newcommand{\julia}{\mathcal{J}}
\newcommand{\TTT}{\mathfrak{T}}
\newcommand{\ttt}{\mathfrak{tr}}
\newcommand{\zeros}{\mathcal{Z}}
\newcommand{\al}{\alpha}
\newcommand{\be}{\beta}

\numberwithin{equation}{section}

\begin{document}

\title[First order differential operators and Hutchinson invariant sets]
{Linear first order differential operators  and their Hutchinson invariant sets}

\author[P.~Alexandersson]{Per Alexandersson}
\address{Department of Mathematics, Stockholm University, SE-106 91 Stockholm,
      Sweden}
\email{per.w.alexandersson@gmail.com}

\author[N.~Hemmingsson]{Nils Hemmingsson}
\address{Department of Mathematics, Stockholm University, SE-106 91 Stockholm,
      Sweden}
\email{nils.hemmingsson@math.su.se}

\author[D.~Novikov]{Dmitry Novikov}
\address{Faculty of Mathematics and Computer Science, Weizmann Institute of Science, Rehovot, 7610001 Israel}
\email{dmitry.novikov@weizmann.ac.il}

\author[B.~Shapiro]{Boris Shapiro}
\address{Department of Mathematics, Stockholm University, SE-106 91 Stockholm,
      Sweden}
\email{shapiro@math.su.se}

\author[G.~Tahar]{Guillaume Tahar}
\address{Beijing Institute of Mathematical Sciences and Applications, Huairou District, Beijing, China}
\email{guillaume.tahar@bimsa.cn}

\date{\today}


\begin{abstract} 
In this paper, we initiate the study of a new interrelation between linear ordinary 
differential operators and  complex  dynamics which we discuss in detail in the simplest 
case of  operators of order $1$. Namely, assuming that such an operator $T$  
has polynomial coefficients, we interpret it as a continuous family of 
Hutchinson operators acting on the space of positive powers of linear forms. Using this interpretation of $T$, we introduce
its continuously  Hutchinson invariant  subsets of the complex plane and investigate a variety of 
their properties. In particular, we  prove that for any  $T$ with non-constant 
coefficients, there exists a unique minimal under inclusion invariant 
set $\minvset{CH}$ and find explicitly what operators $T$ have the property that $\minvset{CH}=\bC$. 
\end{abstract}

\maketitle


\tableofcontents

\section{Introduction}\label{sec:intro} 

\subsection{Motivating notions from complex dynamics}
In this paper, we study certain families of sets associated with linear first order ordinary differential operators.
These sets are closely related to the attractors of Hutchinson operators and to the Julia sets of rational functions.

\smallskip
 Namely, in complex dynamics one often considers a map $\mathcal{F}: 2^{\setRS} \to 2^{\setRS}$
 from the set of subsets of the Riemann sphere to itself and tries to find the  ``fixed points'' of this map, i.e., non-empty subsets  
 $S\subseteq \setRS$ such that  $\mathcal{F}(S) = S$.
For example, the Julia set of a rational map $f(z)$ of degree $2$, is the minimal closed "fixed point" of
\begin{equation}\label{eq:juliaIntro}
\mathcal{F}(S) = \{ f(z) : z \in S \}\bigcup_{u \in S} \{ z:  f(z) = u \}
\end{equation} 
containing at least three points.

\smallskip
Another well-studied instance of such situation occurs in the case of 
Hutchinson operators, i.e. operators $\mathcal{F}: 2^{\bC} \to 2^{\bC}$ of the form 
\begin{equation}\label{eq:hutchinsonfinite}
\mathcal{F}(S) =  \bigcup_{z \in S}  \left\{ \phi_1(z),\dotsc,\phi_\ell(z) \right\}, 
\end{equation}
where $\phi_1,\dotsc,\phi_\ell : \bC \to \bC$ is a finite collection of contracting maps.
Such collections of contractions  are usually referred to as \defin{iterated function systems} (IFS for short).
Then it is classically known that the equation $\mathcal{F}(S) = S$
has a unique non-trivial closed solution, namely, the \defin{attractor} of $\mathcal{F}$, see \cite{Hutchinson1981}.
Examples of such attractors are e.g. \emph{the Sierpinski triangle}, \emph{Koch's snowflake}
and \emph{Barnsley's fern}.

\smallskip
In most of the cases discussed in the existing literature, a non-trivial closed solution to $\mathcal{F}(S) = S$
is  unique due to the fact that $\mathcal{F}$ under consideration
is a contraction in a suitable topology. A good introductory text to this subject is \cite{Barnsley2013DevelopmentsIF}. Significant research has been devoted to the case when the contracting maps $\phi_1,\phi_2 : \bC \to \bC$ are linear (and $\ell=2$). In particular, in \cite{BARNSLEY1985421}, the notion of the "Mandelbrot set" $\mathcal{M}$ for a pair of linear contractions $(\lambda z-1,\lambda z+1):=(f_1,f_2)$, with $\lambda<1$ was introduced. The set $\mathcal M$ is the set of all points in the unit disk for which the unique attractor $S\subset \bC$ of the IFS defined by $(f_1,f_2)$ is connected. The set $\mathcal M$ is contained in the closed annulus of radii $1/2$ and $1$. Among other results it is proved in \cite{BARNSLEY1985421}  that the attractor $S$ is connected if and only if $f_1(S)\cap f_2(S)\neq \emptyset$ and is otherwise totally disconnected. Furthermore,  $\mathcal M$ is different from the closure of its interior. Moreover, if $S$ is connected, it is also locally connected.  In \cite{Bousch1992SurQP} it is proved that $\mathcal M$ is connected and locally connected. (Recall that in classical holomorphic dynamics, local connectedness of the classical Mandelbrot set is one of, if not the, main open question.) In \cite{ChristophBandt2002} the author gives a computer assisted proof that $\{z\in \bC: |z|< 1\}\setminus \mathcal M$ has more than one component. 
 
 \smallskip
The situation which we consider below is quite similar to that of Hutchinson operators.
However in our case,  the map $\mathcal{F}$ is not a contraction and we therefore  do not have unique ``fixed point".
Hence, instead of an attractor we have families of what we call \defin{invariant sets}.

\smallskip
Finally, in our situation the set of maps appearing in \cref{eq:hutchinsonfinite} is  not finite,
but rather parameterized by $t\ge 0$.
Iterated function systems with uncountable many maps 
are not very common in the mathematical literature, but they do appear, see e.g.~\cite{Strobin2021}.
However, such systems are abundant in the fractal art community, whose goal is to  generate and color
interesting attractors.
We refer to the seminal paper by S.~Draves \cite{dravesflame} for background on this topic
and the computer software Apophysis.

\subsection{Invariant sets for linear first order differential operators}
 Let us now describe our basic set-up in detail.  
 Given  two polynomials $P,Q$ non-vanishing  identically, consider the first order linear 
differential operator   
\begin{equation}\label{eq:1st}
T=Q(z)\frac{d}{dz}+ P(z).
\end{equation}
We say that a closed subset $S \subset \setC$ is \defin{Hutchinson $T$-invariant} 
(or \defin{$T_H$-invariant} for short) if for any $u\in S$ and $n \in \setN$,
we have that the polynomial $T[(z-u)^n]$ is either identically zero, or has all its roots in $S$.
In other words, a closed subset $S\subseteq \setC$ is Hutchinson $T$-invariant 
if and only if it is a (closed) ``fixed point" of the operator 
\begin{align*}
  \mathcal{F}(S) &= \bigcup_{u \in S} \bigcup_{n \in \setN} 
  \left\{ z : n Q(z) (z-u)^{n-1} + P(z) (z-u)^n =0 \right\} \\
   &= \bigcup_{u \in S} \bigcup_{n \in \setN} 
  \left\{ z : n Q(z) + P(z) (z-u) =0 \right\}. 
\end{align*}
Unfortunately, it seems  rather difficult to study Hutchinson $T$-invariant sets $S$ for a 
somewhat general operator \eqref{eq:1st}, but we hope to return to this topic in the future. In \cite{hemmingsson2023equidistribution} the second author studied a similar problem to the one above, focusing on sets that are \defin{Hutchinson invariant in degree $n$}. Given $n$ and $T$ (not only of order one but of arbitrary finite degree), these are the closed subsets $S \subset \setC$ such that for any $u\in S$ we have that the polynomial $T[(z-u)^n]$ is either identically zero or has all its zeros in $S$. This problem is directly related to iterations of holomorphic correspondences and it was proved that for a large class of $T$ and sufficiently large $n$, the minimal Hutchinson invariant in degree $n$ exists and is the support of a measure that describes the (equi)distribution of forward iterations of the holomorphic correspondence mapping $u$ to $z$ defined by $T[(z-u)^n]=0$.

Below we concentrate on another related case that is easier to
handle than the Hutchinson $T$-invariant sets, where one replaces a non-negative integer power $n$ with a non-negative real parameter $t$. 
Therefore our main definition  is as follows. 

\begin{definition}
A non-empty closed set $S\subset \setC$ is \defin{continuously Hutchinson invariant}
(or \defin{$T_{CH}$-invariant} for short) if for any $u\in S$ and  an \emph{arbitrary non-negative number} $t$, the 
image $T(f)$ of the function $f(z)=(z-u)^t$ 
has all roots in $S$ or vanishes identically. 
\end{definition}
\begin{remark}
Since for $T$ given by \eqref{eq:1st}, one has $T[(z-u)^t]=(z-u)^{t-1}( t \, Q(z) + (z-u) P(z))$, the polynomial equation
\begin{equation}\label{eq:main}
 t \, Q(z) + (z-u) P(z) =0
\end{equation}
is our main object of interest; note that for any 
$T_{CH}$-invariant set $S\subset \setC$,  every $u \in S$, and any $t \geq 0$, 
the roots of \eqref{eq:main} must also belong to $S$.
\end{remark}

We now observe that a closed set $S$ is $T_{CH}$-invariant if it is a
fixed-point of 
\begin{align}
\mathcal{F}_T(S) &= \overline{ \bigcup_{u \in S} \bigcup_{t\geq 0} \left\{ z : t Q(z) + (z-u)P(z) =0 \right\} } \notag \\
               &= \overline{ \bigcup_{u \in S} \bigcup_{t\geq 0} \left\{ z : z + t\frac{ Q(z)}{P(z)} = u \right\} }, \label{eq:juliaConnection}
\end{align}
where $\overline{\cdots}$ denotes the closure. 
By looking at \eqref{eq:juliaConnection} and comparing it with \eqref{eq:juliaIntro},
it is fairly straightforward to show that any $T_{CH}$-invariant
set must contain the Julia set associated with the map $z \mapsto z + t\frac{ Q(z)}{P(z)}$, for any fixed $t>0$, see \cref{prop:juliaSetIsSubset}. 
This connection is illustrated in \cref{fig:cochleoid}.

\begin{example}
The operator $T = z^2\diffz + (z-1)$ has a unique minimal (under inclusion) $T_{CH}$-invariant set,
whose boundary in polar coordinates can be parameterized  as 
$r(\theta) = \frac{\sin{\theta}}{\theta}$, see the leftmost picture in \cref{fig:cochleoid}.
This boundary curve is called the \defin{cocheloid}. 
The central picture (constructed via a type of Monte-Carlo simulation, hence the artificial void in the center) illustrates 
the minimal invariant set 
associated with iterations of
\begin{equation}\label{eq:tgeq1iterations}
 \mathcal{F}_T(S) = \overline{ \bigcup_{u \in S} \bigcup_{t\geq 1} \left\{ z : z + t\frac{ Q(z)}{P(z)} = u \right\} }.
\end{equation}
Note that having $t\geq 1$ seem to give a fractal boundary.
Finally, the rightmost picture shows the union of several Julia sets associated with iterations of 
$z + t\frac{z^2}{z-1}$.
\begin{figure}[!ht]
\centering
\includegraphics[width=0.29\textwidth]{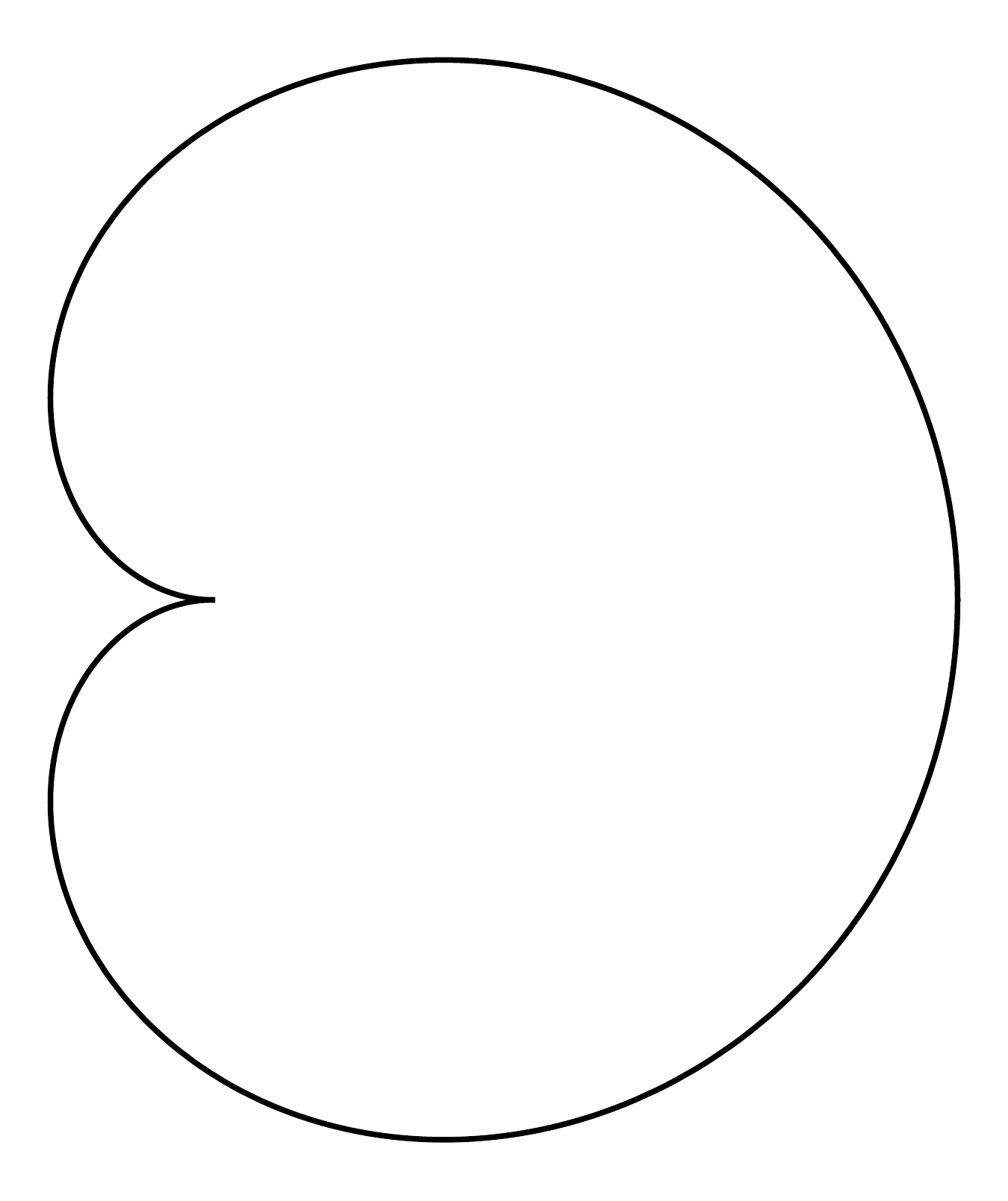}\hspace{0.25cm}
\includegraphics[width=0.32\textwidth]{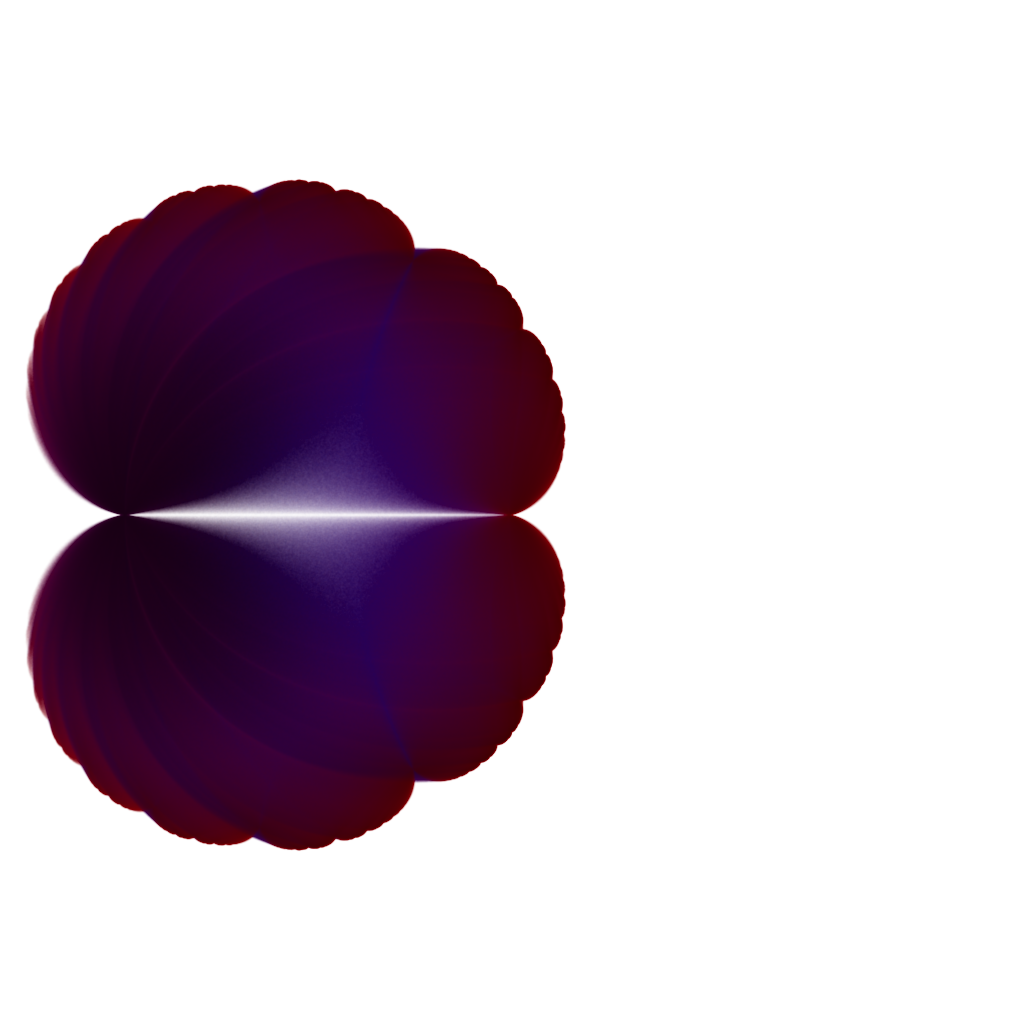}\hspace{-0.75cm}
\includegraphics[width=0.32\textwidth]{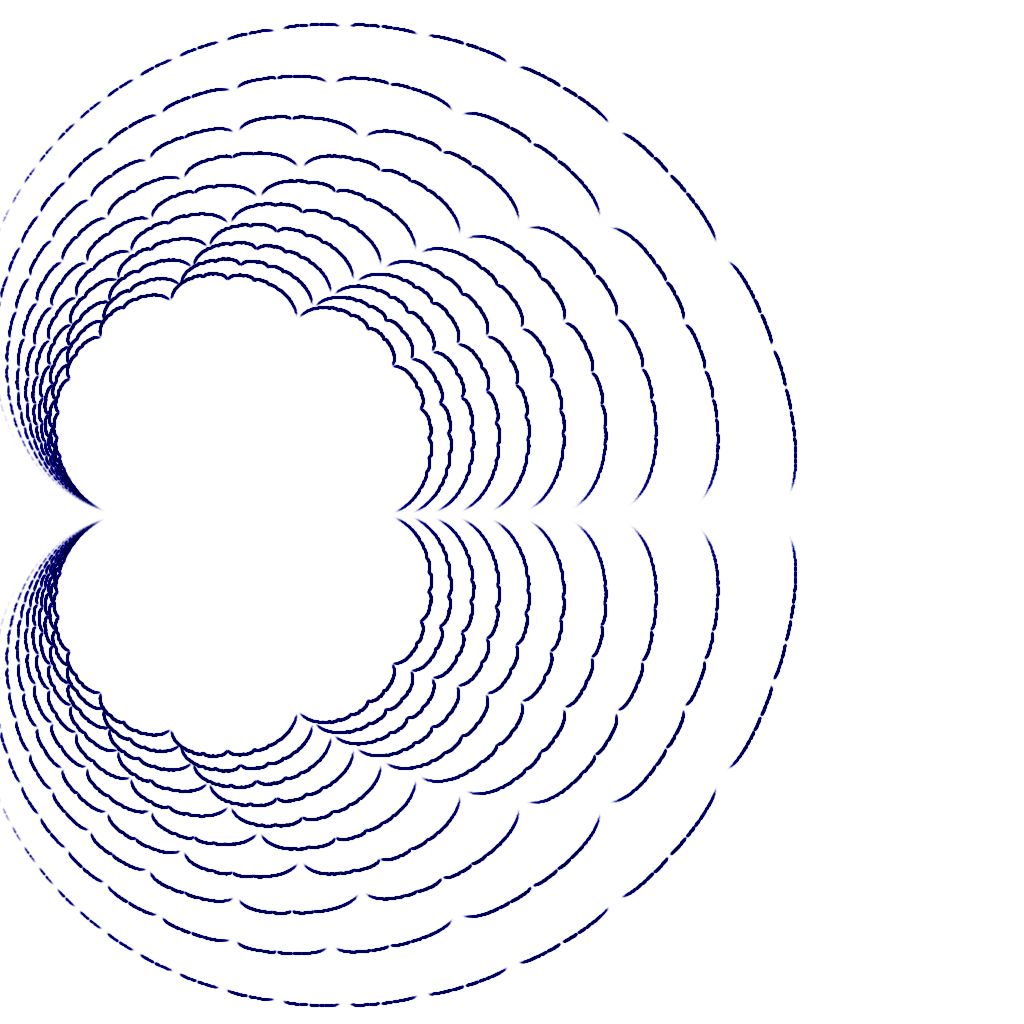}
\caption{The leftmost image shows the boundary of the minimal $T_{CH}$-invariant set for
the operator $T = z^2\diffz + (z-1)$. 
The center image shows (a numerical approximation of) the minimal invariant set, which is a fixed-point of \eqref{eq:tgeq1iterations}.
The rightmost figure shows the union of the Julia sets of the maps $z \mapsto z + t\frac{z^2}{z-1}$, for
$t \in \{0.2, 0.4, 0.6,\dotsc, 1.8, 2.0\}$.
}
\label{fig:cochleoid}
\end{figure}
\end{example}

\medskip 

Below we, in particular,  consider the following questions.

\begin{problem}\label{prob:exist}
For which linear differential operators $T$ given by \eqref{eq:1st}, does
 there exist a unique minimal under inclusion $T_{CH}$-invariant set?
\end{problem}

Below we denote this minimal set (if it exists) by \defin{$\minvset{CH}$}. 

\begin{problem}\label{prob:charac}
Find possible alternative characterizations of $T_{CH}$-invariant sets in terms of the polynomials $P$ and $Q$. 
Describe topological and geometric properties of $T_{CH}$-invariant sets and their boundaries.  
\end{problem}

\begin{problem}\label{prob:compact}
Assuming that for a linear differential operator $T$, the minimal set $\minvset{CH}$ exists,
under which additional conditions is it compact in $\bC$?
\end{problem}

Observe that $\minvset{CH}$ is closed by definition, 
so its compactness is equivalent to its boundedness. 

\smallskip
In the sequel \cite{AHNST} we will concentrate on the following question. 

\begin{problem}\label{prob:boundaryMIN}
Describe properties of the boundary of $\minvset{CH}$ and 
describe explicitly $\minvset{CH}$ for some classes of operators $T$.
\end{problem}

\medskip
To finish the introduction, we present here a  sample of our results related to the above questions.
%
\begin{proposition}\label{prop:basic}
For any linear differential operator $T$ given by \eqref{eq:1st} 
such that at least one of $Q(z)$ and $P(z)$ are non-constant, the next statements hold:

\smallskip
\noindent
{\rm (i)} any $T_{CH}$-invariant set $S\subset \setC$ contains  all roots of both $P(z)$ and $Q(z)$;

\noindent
{\rm (ii)} there exists a unique  minimal under inclusion $T_{CH}$-invariant set \defin{$\minvset{CH}$}.

\end{proposition} 

\smallskip
We call $\minvset{CH}$ \defin{trivial} if it (and therefore every $T_{CH}$-invariant set) coincides with the whole $\bC$.

\begin{notation}  For an operator $T$ given by \eqref{eq:1st}, throughout the whole paper  a very important role will be played by the rational vector field $\dot z= R(z)$ where $R(z)=\frac {Q(z)}{P(z)}$. To distinguish the latter vector field from the rational function $R(z)$, we will denote it by $R(z)\partial_z$. We will also  frequently use the vector field $-R(z)\partial_z$. 
\end{notation}

\begin{notation}\label{not:pq}
We define $p_{\infty},q_{\infty} \in \mathbb{C}^{\ast}$, and $p,q \in \mathbb{N}$ such that $P(z)=p_{\infty}z^{p}+o(z^{p})$ and $Q(z)=q_{\infty}z^{q}+o(z^{q})$.\newline
Then, we have $\lambda = \frac{q_{\infty}}{p_{\infty}} \in \mathbb{C}^{\ast}$ and $\phi_{\infty}=\arg(\lambda)$.\newline
Similarly, for any point $\alpha \in \mathbb{C}$, we have $R(z)=r_{\alpha}(z-\alpha)^{m_{\alpha}}+o(|z-\alpha|^{m_{\alpha}})$ with $r_{\alpha} \neq 0$ and $m_{\alpha} \in \mathbb{Z}$.
\end{notation}

\subsection{Main results}

\begin{theorem}\label{th:main1} 
For any linear differential operator $T=Q(z)\frac{d}{dz}+P(z)$
where $P,Q$ do not vanish identically, there is a non-trivial
$T_{CH}$-invariant set in $\mathbb{C}$
if and only if one of the following statements holds:
\begin{itemize}
\item $\deg Q -\deg P=-1$;
\item $\deg Q -\deg P=0$;
\item $\deg Q -\deg P= 1$, $\deg P \geq 1$ and $Re(\lambda) \geq 0$;
\item $\deg Q=1$, $\deg P=0$ and $\lambda \notin \mathbb{R}^{-}$.
\end{itemize}
\end{theorem}
\begin{theorem}\label{th:main2} 
For any linear differential operator $T=Q(z)\frac{d}{dz}+P(z)$
where $P,Q$ do not vanish identically, there is a compact $T_{CH}$-invariant set
in $\mathbb{C}$ if and only if $\deg Q -\deg P= 1$
and one of the following two statements holds:
\begin{itemize}
\item $\deg P \geq 1$ and $\Re(\lambda) \geq 0$;
\item $\deg Q=1$, $\deg P=0$ and $\lambda \notin \mathbb{R}{<0}$.
\end{itemize}
\end{theorem}

Therefore non-trivial (i.e. different from the whole $\bC$) minimal invariant set $\minvset{CH}$ can only exist when $\deg Q -\deg P\in\{-1, 0, 1\}$. 
As we will show below,  $\minvset{CH}$  is necessarily unbounded and  nontrivial when $\deg Q -\deg P\in\{-1, 0\}$. 

\begin{remark} We do not know of any other situation related to analytic vector fields in which zeros of orders exactly $1, 2$ and $3$ are more special than zeros or poles of other orders.  
\end{remark} 

To move further, we need to introduce the following natural topological property. 

\begin{definition}\label{def:regularSets}
Let $S\subset \bC$ and denote by $S^\circ$ the set of all its interior points.  
Points in $\overline{S^\circ}$ are called \defin{regular} and  points in $S\setminus \overline{S^\circ}$ are called \defin{irregular} points of $S$.  We also call $S\setminus \overline{S^\circ}$ the \defin{irregular locus} of $S$.
We say that  $S\subset \bC$ is (topologically) \defin{regular} if all its points are regular, i.e., $S=\overline{S^\circ}$ 
(which implies that $\partial{S}=\partial{S}^\circ$). 
Otherwise, the set $S$ is called \defin{irregular}. We say that an irregular set $S$ is \defin{fully irregular} if  $S^\circ$ is empty and \defin{partially irregular} otherwise. (Obviously every point of an irregular $S$ is irregular if and only if $S$ is fully irregular). 
\end{definition}

\smallskip
Surprisingly, irregular $T_{CH}$-invariant sets can only occur for operators $T$ enjoying a  very specific property. Namely, up to multiplication by a complex constant and an affine change of the variable $z$, such operator $T$ must  have real-valued coefficients $Q$ and $P$ subject to the restriction $|\deg Q -\deg P|\le 1$.  Existence of fully irregular sets require additional restrictions which can nevertheless be made explicit.  More exact statement is as follows. 

\begin{theorem}\label{thm:ClassReg}
In connection  with (ir)regularity of $T_{CH}$-invariant sets, any linear differential operator $T=Q(z)\frac{d}{dz}+P(z)$
with  $P$ nor $Q$  not identically vanishing belongs to one of the three classes presented in Table~\ref{tab:ClassReg}.
\end{theorem}

\begin{table}
\begin{tabular}{|p{2.5cm}||p{4.5cm}|p{7cm}|} 
\hline 
\centering \textbf{Class of operators}  & \centering \textbf{Shape of $T_{CH}$-invariant sets} & \centering \textbf{Characterization}\tabularnewline 
\hline
\centering Class Ia & 
At least one $T_{CH}$-invariant set is fully irregular.

&
Up to an affine change of variables and the multiplication of $P,Q$ by a common non-zero constant, $P$ and $Q$ are real on $\mathbb{R}$.\newline

Roots of $P$ and $Q$ are real and interlacing (see Section~\ref{sub:realirregular}).\newline

At each root pole of $R$, we have $r_{\alpha}<0$.\newline
\tabularnewline  
\hline

\centering Class Ib & At least one $T_{CH}$-invariant set is fully irregular.

&

$R(z)=\lambda$ with $\lambda \in \mathbb{C}^{\ast}$
\tabularnewline  
\hline

\centering Class Ic & At least one $T_{CH}$-invariant set is fully irregular.

&

$R(z)=\lambda(z-\alpha)$ with $\lambda \notin \mathbb{R}_{<0}$ and $\alpha \in \mathbb{C}$ and $\deg Q=1, \deg P=0$.
\tabularnewline  
\hline

\centering Class Id & At least one $T_{CH}$-invariant set is fully irregular.

&

$R(z)=\lambda(z-\alpha)$ with $\lambda \in \mathbb{R}_{>0}$
\tabularnewline  
\hline

\centering Class II & 
At least one $T_{CH}$-invariant set is irregular,
but no $T_{CH}$-invariant set is fully irregular.

&
Up to an affine change of variables and the multiplication of $P,Q$ by a non-vanishing constant, $P$ and $Q$ are real on $\mathbb{R}$.\newline

$|\deg Q- \deg P| \leq 1$.\newline

If $\deg Q - \deg P = \pm 1$, then $\lambda \in \mathbb{R}_{>0}$.\newline

Does not belong to class I.
\tabularnewline  
\hline  
\centering Class III & 
Every $T_{CH}$-invariant set is regular.
&
Does not belong to classes I or II.
\tabularnewline  
\hline  

\end{tabular}
    \vskip1ex
    \caption{Characterization of linear operators w.r.t. existence of irregular $T_{CH}$-invariant sets.}
    \label{tab:ClassReg}
\end{table}

\medskip
The structure of the paper is as follows. 
In \S~\ref{sec:basic}  we present a general sufficient condition for the existence of  $\minvset{CH}$  
and its implicit description in terms of complex dynamics. 

In \S~\ref{sec:roottrails} we introduce the \emph{root trails/trajectories/$t$-traces} which are solutions of \eqref{eq:main} 
for a fixed initial $u$ and $t\in [0,+\infty)$ and discuss their properties.  
We also introduce a time-dependent vector which describes their dynamics. 
 Additionally we  give a general characterization of $T_{CH}$-invariant sets in terms of the family of \emph{associated rays} of $T$ which is the family of half-lines in $\bC$ spanned by the rational vector field  $R(z){\partial_z}$.  
 
 In \S~\ref{sec:asympt} we introduce a number of notions which will allow us to describe the cases in which all $T_{CH}$-invariant sets of a given operator $T$ are trivial, i.e. coincide with $\bC$.   
 In \S~\ref {sec:topol}  we discuss two basic topological properties of $T_{CH}$-invariant sets, namely their connectedness and compactness.  
 In \S~\ref{sec:triviality} we give necessary and sufficient conditions describing operators $T$ all $T_{CH}$-invariant sets of which are trivial. 
 
  In \S~\ref{sec:irreg}--\S~\ref{sec:partirreg} we discuss irregular  $T_{CH}$-invariant sets and present necessary and sufficient conditions for their existence.  In particular, we completely describe the case when there exist fully irregular $T_{CH}$-invariant sets.  
    In \S~\ref{sec:outlook} we describe many open questions related to our set-up. Finally, Appendix~\ref{sec:ratfields} contains relevant classical material on rational vector fields on $\bC P^1$, special properties of separatrices,  and description of curves of inflections for trajectories of analytic vector fields. 
\medskip

\begin{remark}
Although this paper is an outgrowth of our still unpublished studies \cite{ABS1,ABS2}, it is self-contained and independent of the latter manuscripts.
All the necessary notions and results are presented below. In the sequel \cite{AHNST}  we study in substantially more detail minimal continuous Hutchinson invariant sets $\minvset{CH}$ and their boundaries. 
\end{remark}

\medskip\noindent
\emph{Acknowledgments:}
The fourth author wants to acknowledge the financial support of his
research provided by the Swedish Research Council grants 2016-04416 and 2021-04900. 
The third and the fifth authors were supported by the ISRAEL SCIENCE FOUNDATION (grant No. 1167/17),
	by funding received from the MINERVA Stiftung with the funds
	from the BMBF of the Federal Republic of Germany and
	by funding from the European Research Council (ERC) under the
	European Union's Horizon 2020 research and innovation programme
	(grant agreement No 802107).

\section{Initial facts about $T_{CH}$-invariant sets and $\minvset{CH}$} \label{sec:basic} 

\subsection{Existence}\label{sub:exist} 

\smallskip
For an operator $T$ given by \eqref{eq:1st} and a set $\Omega \subset \setC$,
we call by the \emph{$T_{CH}$-extension} $\TTT(\Omega)$ of $\Omega$ 
the set obtained by the following iterative procedure.
Set $\Omega_0\coloneqq \Omega$ and for any positive integer $j=1,2,\dotsc ,$ define
\[
\Omega_j\coloneqq \bigcup_{u\in \Omega_{j-1}}\ttt_u,
\]
where $\ttt_u$ stands for the set of all solutions of \eqref{eq:main} for a given fixed $u$ and all $t\ge 0$, see details in \cref{sec:roottrails}. (In what follows $\ttt_u$ will be called the \emph{root trail} of $u$.)  
The closure of $\cup_{j=0}^\infty \Omega_j$ is denoted by $\TTT(\Omega)$. If $\Omega=\{\omega\}$ is a singleton we use the notation $\TTT(\Omega)=\TTT(\omega)$.

\smallskip
We start with the following simple claims.

\begin{lemma}\label{le:zerosPQ}
For any linear differential operator $T$ given by \eqref{eq:1st} 
such that at least one of $Q(z)$ and $P(z)$ are non-constant, any $T_{CH}$-invariant set $S\subset \setC$ contains  all roots of both $P(z)$ and $Q(z)$;
\end{lemma}
\begin{proof}
Take a non-empty $T_{CH}$-invariant set $S\subseteq \setC$ and recall that for $t\geq 0$ and $u\in S$, 
\[
T[(z-u)^t ]= (z-u)^{t-1} (t Q(z) + (z-u) P(z)) .
\]
By our assumption, for any $t\ge 0$, all roots of the r.h.s. of the latter expression must belong to $S$. 
  For $t=0$ and $u\in S$,  the roots of
\[
tQ(z)+(z-u)P(z)=0
\]
 are  all the zeros of $P$  together with $z=u$.

For  $t>0$,  dividing both sides of the equation by $t (z-u)^{t-1}$ we get 
\begin{equation}\label{eq:factor}
\frac{T[(z-u)^t]}{t (z-u)^{t-1}} = Q(z) + \frac{1}{t} P(z)(z-u).
\end{equation}
When $t\to \infty$ the second term in the r.h.s. tends 
to $0$. Therefore if $\deg P< \deg Q$ then for $t\to \infty$, all  roots of the right-hand side with respect to $z$  will necessarily tend to those of $Q(z)$. If $\deg P\ge \deg Q$ then $\deg P-\deg Q+1$ roots will tend to infinity and $\deg Q$ roots will tend to those of $Q(z)$. 
Conclusion follows.  

\end{proof}

\begin{lemma}\label{lm:simple} For any linear differential operator $T$ given by \eqref{eq:1st},  the next statements hold:  

\noindent
\rm {(i)} For any $\Omega\subset \setC$,
its $T_{CH}$-extension $\TTT(\Omega)$ is $T_{CH}$-invariant; 

\smallskip\smallskip
\noindent

\noindent
{\rm (ii)} Let $(S_i)_{i\in I}\subset \setC$ be $T_{CH}$-invariant sets.  If $S=\cap_{i\in I}S_i$ is non-empty, then $S$ is invariant.

\noindent
{\rm (iii)} If $S_1$ and $S_2$ are  $T_{CH}$-invariant, then $S_1\cup S_2$ is $T_{CH}$-invariant.

\end{lemma} 
\begin{proof} 
Let $u_0\in \TTT(\Omega)$. We need to show that the closure of the root-trail of $u_0$ is in $\TTT(\Omega)$. Suppose not. We note that when $t\to 0$, the zeros of \eqref{eq:main} tend to zeros of $P$ together with  $u$ and some zeros might  tend to $\infty$ and when $t\to \infty$, the zeros tend to the zeros of $Q$ and  some zeros might tend to $\infty$. Hence, we can deduce that there is a solution $z_0$ to \eqref{eq:main} for $u=u_0$ with $t=t_0\in (0,\infty)$ such that $z_0\notin \TTT(\Omega)$. Now, there is a sequence $(u_k)_{k=1}^\infty$ converging to $u_0$ such that $u_k\in \bigcup_{j=0}^\infty \Omega_j$. Consider the sequence of sets $(U_k)_{k=1}^\infty$ of solutions to \eqref{eq:main} satisfying $u=u_k, t=t_0$. We have that $U_k\subset \bigcup_{j=0}^\infty \Omega_j$. Since the set of zeros of a polynomial is continuous in the coefficients, it follows that there are $z_k\in U_k$ such that $z_k\to z_0$. Hence, $z_0\in \overline{ \bigcup_{j=0}^\infty \Omega_j}=\TTT(\Omega)$. Contradiction. We conclude (i).

For item (ii), suppose the conclusion is false. Then there is a point $u\in S$ such that $T[(z-u)^t]$ has a zero $z_0$ not in $S$. However, each $S_i$ is invariant and contains $u$, so $z_0\in S_i$ for each $i$. This is a contradiction and the statement follows.
 To settle (iii)
recall that a subset of $\bC$ is $T_{CH}$-invariant, if it coincides with its $T_{CH}$-extension,
so $\TTT(S_1) = S_1$ and $\TTT(S_2) = S_2$.
The definition of $T_{CH}$-extension implies directly that for any sets $\Omega_1$, $\Omega_2$,
we have 
\[
 \TTT(\Omega_1 \cup \Omega_2) = \TTT(\Omega_1) \cup \TTT(\Omega_2).
\]
Thus,
\[
 \TTT(S_1 \cup S_2) = \TTT(S_1) \cup \TTT(S_2) = S_1 \cup S_2
\]
which implies that  $S_1 \cup S_2$ is $T_{CH}$-invariant.
\end{proof}

\begin{proof}[Proof of Proposition~\ref{prop:basic}] 
Item (i) is \cref{le:zerosPQ}.

Item (ii)  follows from item (i) together with item (ii) from \cref{lm:simple}. Indeed, by definition any $T_{CH}$-invariant set is closed and  we  conclude  the existence of a unique minimal $T_{CH}$-invariant set $\minvset{CH}$ 
obtained as the intersection of the complete family of $T_{CH}$-invariant sets.   Notice that depending on $T$ the minimal set $\minvset{CH}$  might or might not be bounded which will be discussed in details below. 
\end{proof}

Observe that the assumption that $T$ has non-constant coefficients is essential for the 
existence of $\minvset{CH}$, see \cref{ssec:constant}.

\begin{remark}\label{re:affineChange}
When using an affine change of variable, we typically apply it to  the vector 
field $R(z)\partial_z=\frac{Q(z)}{P(z)}\partial_z$ and not the 
rational function $R(z)$. 
That is, for $z \mapsto aw+b$,  we get a new vector field 
$\hat{R}(w)\partial_w \coloneqq \frac{Q(aw+b)}{a(P(aw+b))} \partial_w$.
This is equivalent to considering the change of variables in the 
operator $T=Q(z)\diffz + P(z)$.
Indeed, the change of variables $z \mapsto aw+b$ yields 
\[
T[(aw+b-u)^t]=tQ(aw+b)(aw+b-u)^{t-1}+P(aw+b)(aw+b-u)^t.
\] 
Dividing with $\frac{(aw+b-u)^{t-1}}{a}$ gives 
\[
\frac{t}{a}Q(aw+b)+P(aw+b)\left(w-\frac{u-b}{a}\right).
\]
Since $w=\frac{z-b}{a}$, we see that a set $S$ is $T_{CH}$-invariant for $T$
if and only if $\hat{S}=\{\frac{z-b}{a}:z\in S\}$ is $\hat{T}_{CH}$-invariant where
\[
\hat{T} = \frac{1}{a}Q(aw+b)\frac{d}{dw}+P(aw+b).
\] 
As stated, here we obtain $\hat{R}(w)=\frac{Q(aw+b)}{a(P(aw+b))}.$ 
\end{remark}

\subsection{Implicit characterization of $\minvset{CH}$ and $1$-point generated  sets}
  For any operator $T$ with non-constant coefficients, let us now provide 
a general implicit description of the minimal $T_{CH}$-invariant set $\minvset{CH}$ in spirit of complex dynamics.

\begin{lemma}\label{lm:trivial}

For any operator $T$ given by \eqref{eq:1st} with non-constant $Q(z)$ or non-constant $P(z)$,
its minimal $T_{CH}$-invariant set $\minvset{CH}$ is the $T_{CH}$-extension of any  point belonging to $\minvset{CH}$.  
\end{lemma}

\begin{proof}

Indeed if there is a point $u\in \minvset{CH}$
whose $T_{CH}$-extension $\TTT(u)$, which is $T_{CH}$-invariant by (i) of \cref{lm:simple}, is strictly contained in $\minvset{CH}$
then would not be the minimal under inclusion $T_{CH}$-invariant set.  
\end{proof}

Using the above definitions we obtain the following. 

\begin{corollary}\label{cor:simple}
For any operator $T$
given by \eqref{eq:1st} such that  $\deg Q(z)\ge 1$,
the minimal $T_{CH}$-invariant set $\minvset{CH}$ is the $T_{CH}$-extension of
an arbitrary root of $Q(z)$. Similarly, if $\deg P(z)\ge 1$ 
then $\minvset{CH}$ is the $T_{CH}$-extension of any root of $P(z)$.
\end{corollary}
\begin{proof}
The statement follows immediately from \cref{lm:trivial} and item (i) of \cref{prop:basic}. 
\end{proof}

\cref{lm:trivial} and Corollary~\ref{cor:simple} show that $\minvset{CH}$ is an example a of  
\defin{$1$-point generated $T_{CH}$-invariant set}  which, by definition, is the $T_{CH}$-extension of a single point $u\in \bC$. 
Moreover $\minvset{CH}$ is generated by any of its points which is not true for more general $1$-point generated $T_{CH}$-invariant sets. 

Nevertheless these sets are ``building blocks" of arbitrary $T_{CH}$-invariant sets.  
Namely,  by item (iii) of \cref{lm:simple} any $T_{CH}$-invariant set is (the closure of) 
a union of some collection of $1$-point generated sets. 

Thus they play a special role in our set-up and unless $\minvset{CH}=\bC$ there exist $1$-point
generated sets different from  $\minvset{CH}$. 
Another natural application of $1$-point generated sets is related to the following claim. 

\begin{theorem}\label{th:reduc} Assume that polynomials $Q(z)$ and $P(z)$ have a common factor  $U(z)$ with roots $z_1,\dots, z_\ell$ so that $Q(z)=\widetilde Q(z) U(z)$ and $P(z)=\widetilde P(z) U(z)$. Also assume that neither $Q(z)$ nor $P(z)$ vanish identically and at least one of them is non-constant. Then the minimal invariant set $\minvset{CH}$ equals the union of $1$-point generated $\widetilde T_{CH}$-invariant sets $S_1, S_2,\dots, S_\ell$ generated by $z_1, z_2,\dots, z_\ell$ respectively, where  $\widetilde T=\widetilde Q(z)\frac{d}{dz}+ \widetilde P(z)$.
\end{theorem}

\begin{proof} 
Indeed, any $T_{CH}$-invariant set must contain all roots of $Q(z)$ and $P(z)$ and, in particular, all roots of $U(z)$. Equation~\ref{eq:main} factorizes as 
\[
U(z)(t \tilde Q(z) +(z-u) \tilde P(z))=0
\]
which means that any $T_{CH}$-invariant set  is in fact a $\widetilde T_{CH}$-invariant containing all roots of $U(z)$ and vice versa. Thus  $\minvset{CH}=S_1\cup S_2\cup \dots \cup S_\ell$. 
\end{proof} 

\begin{remark} Notice that Theorem~\ref{th:reduc} describes  $\minvset{CH}$ not as a $T_{CH}$-extension but as a $\tilde T_{CH}$-extension and that Lemma~\ref{lm:trivial}  holds both in the case when $Q$ and $P$ are coprime and when they have a common factor. 
\end{remark} 

\subsection{Trivial special cases}\label{ssec:trivial}

For the sake of completeness we discuss here some very degenerate and  trivial cases which will be exceptional from the of view of  our  general framework and results described later.  

\subsubsection{Case of operators with constant coefficients}\label{ssec:constant}\label{subsublinear}  
In \cref{sub:exist} we have shown the existence of $\minvset{CH}$ for any operator $T$ given by \eqref{eq:1st} with at least one non-constant coefficient. Let us separately cover the case $T=\al \frac{d}{dz}+\be$ where $\al$ and $\be$ are non-vanishing complex numbers. 

\begin{definition} Given a complex number $\xi\neq 0$, we say that a set $S\subset \setC$ is \emph{closed in direction $\xi$} if for any point $p\in S$, the ray $p+t \xi$ where $t$ runs over non-negative numbers belongs to $S$. 
\end{definition}

\begin{lemma} \label{lm:const} Given an operator $T=\al \frac{d}{dz}+\be$ where $\al$ and $\be$ are non-vanishing complex numbers, a closed set $S\subset \setC$ is $T_{CH}$-invariant if and only if 
$S$ is closed in direction $\xi=-\frac{\al}{\be}$. 
\end{lemma} 
\begin{proof} For any $T$ as above, equation~\eqref{eq:main} takes the form $\al t +\be (z-u)=0$. Therefore if $S$ is $T_{CH}$-invariant and $u\in S$, then  $z(t)=u-\frac{\al}{\be}t$ belongs to $S$ for all non-negative $t$. The latter condition coincides with the requirement that $S$ is closed in direction $\xi=-\frac{\al}{\be}$. 
\end{proof} 

In particular, in the above notation the $T_{CH}$-invariant set generated by a point $u\in \bC$ coincides with the closed ray starting at $u$ and having its direction vector equal to  $\xi$.

\subsubsection{Case when either $P$ or $Q$ vanish identically}
Note that for constant coefficients, using the notation of \S~\ref{subsublinear},  if at least one of $\al$ and $\be$ are equal to zero, then any closed subset of $\setC$ is $T_{CH}$-invariant. If however $P\equiv 0$ and $\deg Q\ge 1$, then a set is $T_{CH}$-invariant if and only if it contains the roots $Q$. In the same vein, if $Q\equiv 0$ and $\deg P\ge 1$, then a set is $T_{CH}$-invariant if and only if it contains the roots of $P$.

\subsubsection{Case when $Q=\alpha(z-z_0),P=-\delta\alpha$ and $\delta>0$}\label{ssec:special}
In this very special situation, set $u=z_0$ and consider the roots of 
\[t\alpha(z-z_0)-\delta\alpha(z-u)=0 \iff  (t-\delta)(z-z_0)=0.\]
When $t=\delta$, it is not immediately clear what the roots of this equation are supposed to be. In order to make this case consistent with the others and comply with \cref{th:charact} below, we define the roots of the equation to be the whole of $\bC$ so that $\minvset{CH}=\bC$.

\medskip
\emph{In what follows (unless explicitly mentioned) we shall  assume that $Q$ and $P$ have no common roots and that $T$ is not as in \cref{ssec:constant}--\ref{ssec:special}.}

\section{Root trails, associated rays and explicit criterion of $T_{CH}$-invariance} \label{sec:roottrails} 

In this section we introduce a number of technical tools to be used throughout the paper.   
\begin{definition}
For an operator $T$ given by \eqref{eq:1st}, a complex number $u$ and $t> 0$,
we call by the  \defin{root divisor} $\ttt_u(t)$ of the pair $(u,t)$ the set of all solutions of  \eqref{eq:main} considered in $\setRS=\bC P^1$ and by the  \defin{root trail} $\ttt_u$ of $u$ the
closure  in $\setRS$ of the union $\cup_{t> 0}\ttt_u(t)$. 
\end{definition} 

Observe that for any complex number $u$, the root divisor $\ttt_u(0)$ which we call the \defin{initial divisor} contains $u$ and 
 the zero locus of $P(z)$ (together with  $\infty\in \setRS$ if $\deg Q> \deg P+1$). Further,   the root divisor $\ttt_u(\infty)=\lim_{t\to +\infty}\ttt_u(t)$ which we call the \defin{final divisor}  contains $Q(z)$ (together with  $\infty\in \setRS$ if $\deg P\ge \deg Q$).

\begin{definition}
We say that a  number $u\in \bC$ is \defin{$T$-generic}  if for all $t>0$, 
the root divisors $\ttt_u(t)\subset \setRS$ are simple, i.e., have no multiple roots. 
By definition, for a generic $u$, its \defin{open trail} $\ttt_u^\circ=\cup_{t>0}\ttt_u(t)$ splits into $\max(\deg Q, \deg P+1)\coloneqq N$ smooth non-intersecting connected components which we call \defin{$t$-trajectories} of $u$. (Notice that by our assumption $Q$ and $P$ are coprime.) 
\end{definition} 

\begin{definition}
For any fixed $u$, we call by a \defin{$t$-trace} a continuous  function $\gamma_u(t): [0, +\infty)\to \setRS$ such that $\gamma_u(t)$ solves \eqref{eq:main} for each $t\geq 0$. The notation $\gamma(t)$ will also be used and we say that a $t$-trace $\gamma(t)$ corresponds to  $u\in \bC$ if it solves \eqref{eq:main}.
\end{definition}

\begin{figure}[!ht]
\begin{center}
 \includegraphics[width=0.7\textwidth]{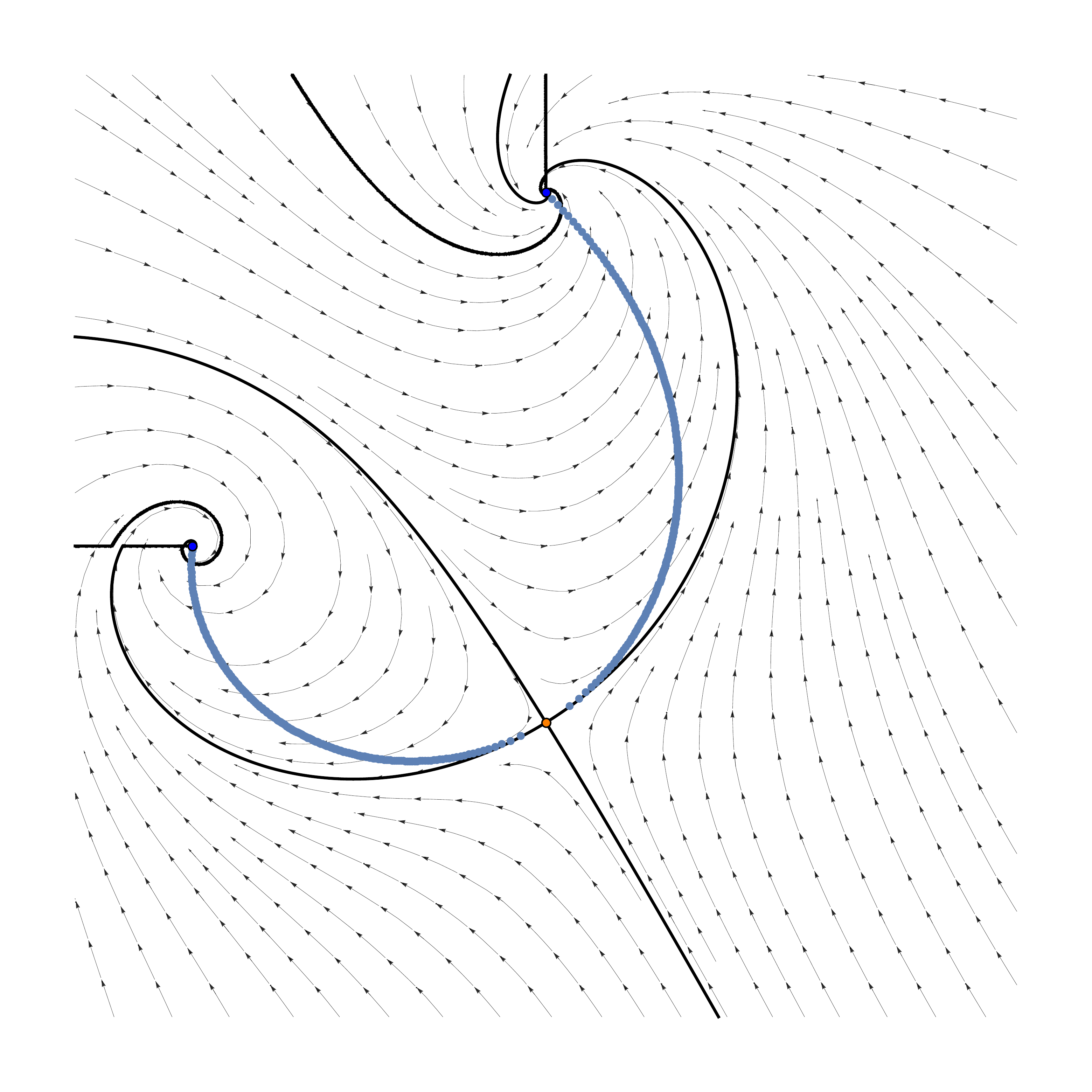}
\end{center}
\caption{For $Q(z) = (z + 1) (z - i)$, $P(z)=  2 z + i$, 
we show a zero $z_0 = -i/2$ of $P(z)$, 
two separatrices (see \cref{def:septrix} for a definition) emerging from $z_0$ (black, thin),
and $t$-trajectories originating at $z_0$ (blue, thick).
}
\end{figure}

Below we will describe the set of $T$-non-generic complex numbers $u$ and the subset of $\bC$ where such $\ttt_u$ are non-generic. It is somewhat surprising that this set 
is a part of the so-called \emph {curve  $\infl_R$ of inflection points} of the vector field $R{ \partial_z}$ where $R(z)=\frac{Q(z)}{P(z)}$,  see \S~\ref{ssec:inflections}.

\subsection{Non-generic root trails}

Interpreting $t$ as time we will show that $t$-traces
form a time-dependent flow in $\setC$ and we will find explicitly the singular time-dependent vector field in the space $\setC\times \bR_{\ge 0}$ which generates  this flow. The time-dependent singularities of this field will be closely related to non-generic root trails, i.e., those which do not be split into separate $t$-trajectories over the half line $t>0$. 

For the next result and throughout the text we will use the notation $\zeros(f)$ for the set of zeros of the function $f$, with the convention that if $f\equiv 0$, $\zeros(f)=\emptyset$. (To fully understand the result below our readers need to take a brief look at \S~\ref{ssec:inflections} which introduces and describes  the plane curve consisting of inflection points of trajectories of the  vector field given by an arbitrary analytic function in $\bC$.) 

\begin{lemma}\label{le:genericu}
{\rm (i)} For any operator $T$ given by \eqref{eq:1st},   a  starting
point $u\in \bC $ is generic, i.e.,  its  root divisors $\ttt_u(t)$ are simple for all positive $t$, if and only if the rational function $\Psi_u(z)=\frac{(u-z)P(z)}{Q(z)}$ has no positive  critical values.

\noindent 
{\rm (ii)} The set  $\Theta_T\subset \bC$  consisting of points $z_0$ such that  there exists $u\in \bC$ for which $z_0$ is a critical point of $\Psi_u$ having a positive critical value coincides with the negative part of the curve of inflections, i.e. with $\infl_R^-\subset \bC$. 

\noindent 
{\rm (iii)}  The set $\theta_T \subset \bC$ consisting of all $T$-non-generic $u$ is the image of $\infl_R^-\subset \bC$ under the mapping $z\mapsto u$ given by $u=z+\frac{PQ}{P^\prime Q-Q^\prime P}$. 
\end{lemma}

\begin{proof} To settle (i), 
notice that for a fixed starting point $u$, its root divisor $\ttt_u(t)$ is given by the equation \eqref{eq:main} which is 
 equivalent to 
\begin{equation}\label{eq:t}
t=\frac{(u-z)P(z)}{Q(z)}.
\end{equation}
Thus for $u$ fixed, the roots of the latter equation (w.r.t the variable $z$)
for distinct values of $t$ belong to different
level sets of the rational function $\frac{(u-z)P(z)}{Q(z)}$ and therefore
are necessarily disjoint. It might however happen that for a fixed $t_0$, 
some of the roots of \eqref{eq:t} are multiple which corresponds to the
case when $t_0$ is a critical value of $\frac{(u-z)P(z)}{Q(z)}$.
This occurs, for example, for $t=0$ if $u$ is chosen to be a root of $P(z)$.
If, for a given $u$, $t>0$ is never a critical value of 
$\frac{(u-z)P(z)}{Q(z)}$ then $\ttt_u^\circ$ is not self-intersecting implying that $u$ is generic. Item (i) follows.

To settle (ii), notice that for $u$ fixed, 
\[
\diffz\Psi_u(z)=-\frac{P}{Q}+(u-z)\left( \frac{P}{Q} \right)^\prime, 
\]
where the symbol $^\prime$ means that the derivative is taken w.r.t. $z$. The critical points of $\Psi_u(z)$ satisfy the equation 
\[
u\rho^\prime=z\rho^\prime+\rho \iff u=z+\frac{\rho}{\rho^\prime}=z+\frac{PQ}{P^\prime Q- Q^\prime P},
\]
where $\rho(z)=\frac{P(z)}{Q(z)}=\frac{1}{R(z)}$. Thus if we  fix $z_0$ and want to find $u_0$ such that $z_0$ is 
a critical point of $\Psi_{u_0}(z)$ then we should take $u_0=z_0+\frac{\rho(z_0)}{\rho^\prime(z_0)}$.
Let us now calculate the critical value of  $\Psi_{u_0}(z)$  at $z_0$. We get 
\[
\Psi_{u_0}(z_0)=(u_0-z_0)\rho(z_0)=\frac{\rho(z_0)}{\rho^\prime(z_0)}\rho(z_0)=
\frac{P^2(z_0)}{P^\prime (z_0) Q(z_0)- Q^\prime (z_0)P(z_0)}=- \frac{1}{R^\prime(z_0)}.
\]
The requirement that  $\Psi_{u_0}(z_0)$ is positive is equivalent 
to the requirement that $R^\prime(z_0)$ is negative, which by \cref{lm:infl} defines the negative part of the curve of inflections $\infl_R^-$.

\smallskip
Finally, to settle (iii), we already noted that for a given $z_0$, to make it  a critical point of $\Psi_{u_0}(z)$ one 
should take $u_0=z_0+\frac{\rho(z_0)}{\rho^\prime(z_0)}$. 
Thus the set $\theta_T$ of all non-generic $u$ is obtained from  the set $\Theta_T$ under the latter mapping.
\end{proof} 

\begin{remark} 
Notice that if $Q(z)$ has a multiple root then for any choice of $u$,
$\ttt_u(\infty)$ has a multiple root. 
If all roots of $P(z)$ are simple then for any $u$ which is not a root of $P(z)$,
$\ttt_u(0)$ has no multiple roots. 
(If $\deg Q>\deg P+1$ then $\deg Q - \deg P -1$ roots of $\ttt_u$ will be coming from $\infty\in \setRS$.)
\end{remark}

Observe that for any pair $(z_0,t_0)$,  where $z_0$ is an arbitrary complex number different 
from a root of $P(z)$ and $t_0> 0$, there exists a unique $u_0$ such that  $\gamma_{u_0}(,t_0)=z_0$. Indeed if $z_0 $ is not a root of $P(z)$, then
the starting point $u_0$ is given by  $z_0+t_0R(z_0)$. (This circumstance will be very important later when we introduce the notion of associated rays). 
In other words, equation~\eqref{eq:main} can be solved for $u$ in case $t_0>0$ and
 $z_0$ different from any root of $P(z)$ and it can not be solved for $u$ 
in case $t>0$ and $z_0$ being a root of $P(z)$. 
For $t=0$, equation \eqref{eq:main} can be solved for $u$ for all $z_0$ including the roots of $P(z)$. 
\begin{remark}
Notice that the root divisors $\ttt_{u_1}(t), \ttt_{u_2}(t)$ of two starting points $u_1\neq u_2$ can be such that $\bC\supset \left(\ttt_{u_1}(t_1)\cap\ttt_{u_2}(t_2)\right)\neq \emptyset$ with either $t_1>0$ or $t_2>0$, but only if $t_1\neq t_2$.
\end{remark} 
The following result holds.

\begin{proposition} 
Given $u\in \bC$, its root trail $\ttt_{u}$ is (the closure in $\setRS$ of) the real semi-algebraic curve given by
\begin{equation}
  \Im\left( P(z) \overline{Q(z)} (z-u) \right) =0, \qquad 
  \Re\left( P(z) \overline{Q(z)} (z-u) \right) \leq 0.
\end{equation}
\end{proposition}
\begin{proof}
Suppose $t Q(z) + (z-u)P(z) =0$. 
Multiplying both sides by $\overline{Q(z)}$ and solving for $t$ gives
\[
 t   = -\frac{P(z)\overline{Q(z)}(z-u)}{ |Q(z)|^2 }.
\]
This expression must be real, so we get the first condition. Moreover,
if we want $t \geq 0$, we need the second condition.
\end{proof}

\subsection{Time-dependent vector field} 

\begin{theorem}\label{th:addit}
For any operator $T$ given by \eqref{eq:1st}, the following claims are valid.
\smallskip
\noindent
{\rm (i)} The $t$-traces are trajectories of the time-dependent vector field $V(z,t)\partial_z$, where
\begin{equation}\label{eq:field}
	V(z, t)=-\frac{R(z)}{t R^\prime(z)+1}.
\end{equation}

\noindent{\rm (ii)} If $P(z)$ and $Q(z)$ are coprime, then for $t>0$  the roots of $P(z)$  are repelling zeros of the vector field \eqref{eq:field} while the multiple zeroes of $Q(z)$ are  zeroes of \eqref{eq:field} of the same multiplicity. Moreover,  simple zeros of $Q$ are attracting simple zeros of  \eqref{eq:field} for $t\gg 1$. 

\noindent{\rm (iii)} In addition to the above time-independent zeros,   the vector field \eqref{eq:field} has moving poles given by the equation $R^\prime(z)=-\frac{1}{t}$.  When $t\to 0^+$, for any root $z_0$ of $P$ of multiplicity $n$ there are $n+1$  moving poles tending  to $z_0$  giving at the limit $t=0$, together with the simple zero of $V$ at $z_0$, a pole of order $n$. 
\end{theorem} 

\begin{remark} 
Observe that the moving poles traverse the negative part $\infl_R^-$ of the curve of inflections, see \S~\ref{ssec:inflections}  which is exactly the set where different $t$-trajectories with the same starting point $u$ can ``merge".
\end{remark} 
\begin{proof}
To settle (i) we invoke \eqref{eq:main}. 
Fixing $t\ge 0$, we want to find the derivative $\gamma^\prime(t)$ which will provide the formula for the vector field in question.
The starting position $u$ for the trajectory
under consideration is given by
\begin{equation}\label{eq:begin}
u=\gamma(t)+tR(\gamma(t)).
\end{equation}
Taking the  derivative with respect to $t$, we get
\begin{equation*}
	0=\gamma'(t)+R(\gamma(t))+tR'(\gamma(t))\cdot \gamma'(t),
\end{equation*}
which is \eqref{eq:field}.

\smallskip
To settle (ii),  for any $z_0\in\mathcal{Z}(PQ)$ let $n=n(z_0)$ be the order of zero of $R$ at $z_0$, i.e. $R(z)=\alpha (z-z_0)^{n} + o\left((z-z_0)^n\right)$ as $z\to z_0$. Then 
\begin{equation*}
	V(z,t)=-\alpha (z-z_0)^n + o\left((z-z_0)^n\right),\quad\text{for } n>1,
\end{equation*}
\begin{equation*}
	V(z,t)=-\frac{\alpha}{1+t\alpha} (z-z_0) + o\left((z-z_0)\right),\quad\text{for } n=1,
\end{equation*}
and
\begin{equation*}
	V(z,t)=-\frac{1}{tn} (z-z_0) + o\left((z-z_0)\right),\quad\text{for } n<0.
\end{equation*}

\smallskip
Finally, to settle (iii), note that
\[
t R^\prime(z)+1=0 \quad \iff \quad  R^\prime(z) = -\frac{1}{t}<0,
\]
and, as $\operatorname{ord}_{z_0}R^\prime=-n-1$, there are $n+1$ simple roots of this equation converging to $z_0$ as $t\to 0+$.
%
%
\end{proof}

\subsection{Dependence of \texorpdfstring{$t$}{t}-trajectories on \texorpdfstring{$u$}{u}}

Let $\gamma_u(t)$ be a $t$-trace corresponding to a starting point $u$.  It depends locally on a complex parameter $u$ and a real/positive parameter $t$. In Theorem~\ref{th:addit} we have calculated the derivative of a $t$-trace with respect to the time parameter $t$. We can also find the partial derivative of  $\gamma_u(t)$ with respect to the starting point $u$, when it exists.  

\begin{lemma}\label{prop:deru} 
Let $u_0\in \bC$, $\gamma_u(t)$ be a $t$-trace such that $P(\gamma_{u_0}(t_0))\neq 0$ and $\gamma_{u_0}(t_0)\neq\infty$ and $\gamma_{u_0}(t_0)$ is not a multiple root of \eqref{eq:main} with $t=t_0, u=u_0$. Then
\[
\frac{\partial\gamma_u(t_0)}{\partial u}\Big|_{u=u_0}=\frac{1}{t_0R'(\gamma_{u_0}(t_0))+1}.
\]

\end{lemma}

\begin{proof}
Since $\gamma_{u_0}(t_0)$ is not a multiple root of \eqref{eq:main} for $t=t_0$, $u=u_0$, $\gamma_u(t_0)$ is a function in $u$ on some neighborhood of $u$.
	Then $\gamma_u(t_0)$ satisfies the equation	
\[
t_0R(\gamma_u(t_0))+\gamma_u(t_0)-u=0.
\]
Differentiating with respect to $u$ at $u=u_0$ we obtain 
\[
(tR'+1)\cdot\frac{\partial\gamma_u(t_0)}{\partial u}\Big|_{u=u_0}-1=0,
\]
which proves the Lemma.
\end{proof}

\begin{remark}
If $u$ is generic and $t>0$ then $\frac{\partial\gamma_u(t)}{\partial u}$ exists and is non-zero.
Furthermore, for a given point $z\notin \zeros(P)\cup\{\infty\}$, there is a family of $t$-traces defined by the condition  $\gamma_u(t)=z$, $u=z+tR(z)$. For all but at most one pair $(u,t)$ such that $u=z+tR(z)$, the derivative $\frac{dz}{du}\coloneqq \frac{\partial \gamma_u(t)}{\partial u}$ exists and is non-zero. 
\end{remark}

\subsection{Behavior of \texorpdfstring{$t$}{t}-trajectories near roots of \texorpdfstring{$P$}{P}}

We now show how the $t$-traces behave when $t\to 0$.

\begin{proposition}\label{prop:rootTrajectoryDirections}
Suppose that $P$ and $Q$ have no common zeros and that $P(z)=(z-z_0)^m G(z)$ 
where $G(z_0) \neq 0$ and $m \geq 1$. 
Consider the Laurent series expansion
\[
  R(z) = \sum_{j\geq -m} b_j (z-z_0)^j \text{ where } b_{-m} = \frac{Q(z_0)}{G(z_0)}.
\]

\smallskip
\noindent
The following facts hold:

\noindent
{\rm (i)}
If $\gamma(t)$ solves the equation $t Q(z)+(z-u)P(z)=0$, where $\gamma(0) = z_0\neq u$ and $\eta(t)\coloneqq\gamma(t^m)$     
then
\[
  (\dot{\eta}(0))^m  = -\frac{Q(z_0)}{(z_0-u)G(z_0)}.
\]
 In particular, if $m=1$, we have 
\[
  \dot{\eta}(0) = -\frac{Q(z_0)}{(z_0-u)P'(z_0)}.
\]

\noindent
{\rm (ii)} If $\gamma(t)$ solves $t Q(z)+(z-u)P(z)=0$,  where $\gamma(0)=z_0=u$, and  $\eta(t)\coloneqq \gamma(t^{m+1})$  
then
\[
  (\dot{\eta}(0))^{m+1} = -\frac{Q(z_0)}{G(z_0)}.
\]

\end{proposition}
\begin{proof}
In case (i), note that $\eta(t)$ solves $t^{m}R(z) + (z-u)=0$.
	Equivalently, 
	\[
	t^m=-\frac{\eta(t)-u}{R(\eta(t))}.
	\]
Therefore 
\[
t^m=(\eta(t)z-z_0)^m U(\eta(t)), \quad\text{where } U(z)=-\frac {z_0-u} {b_{-m}}+O(z-z_0).
\]
Taking an $m$-th root and applying an inverse function theorem, we get 
\[
\eta(t)-z_0=\xi\left(-\frac {b_{-m}} {z_0-u}\right)^{1/m}t+o(t), \quad \xi^m=1,
\]
which implies (i).
Similarly, if $u=z_0$ then 
\[
t^{m+1}=(\eta(t)z-z_0)^{m+1} U(\eta(t)), \quad\text{where } U(z)=-\frac {1} {b_{-m}}+O(z-z_0)
\]
and 
\[
\eta(t)-z_0=\xi\left(-{b_{-m}}\right)^{\frac{1}{m+1}}t+o(t), \quad \xi^{m+1}=1,
\]
which implies (ii).
\end{proof}

Item (ii) above shows that the directions of the
$t$-trajectories originating at $u=z_*\in \zeros(P)$
agree with the directions of the root-starting separatrices of the vector field $-R(z){\partial_z}$ (c.f. \cref{prop:dirSeptrix}).

\begin{example}\label{ex:rootTrajectoryDirections}
Set $Q(z)=z (z + \frac{1 + i}{3}) (z^2 + \frac14)$ and $P(z) = (z - 1)^2 (z - i)$.
In \cref{fig:tTrajectoriesExample}, we show the vector field $-R(z){\partial_z}$
together with the $t$-trajectories of the zeros of $P$; 
the $t$-trajectories associated with the double root $u=1$ are shown in  black,
while the ones with $u=i$ are shown in blue.
\begin{figure}[H]
 \includegraphics[width=0.4\textwidth]{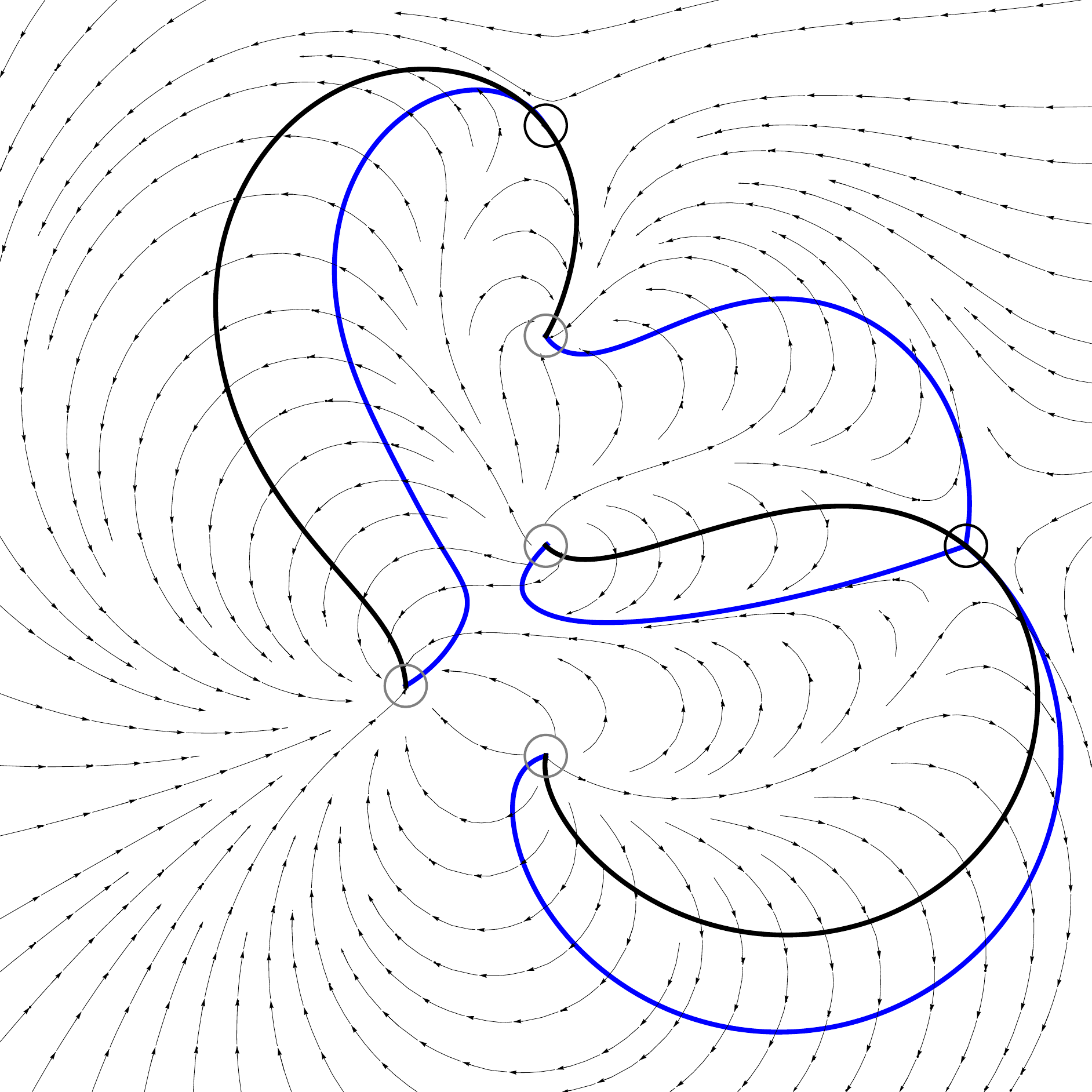}
 \includegraphics[width=0.4\textwidth]{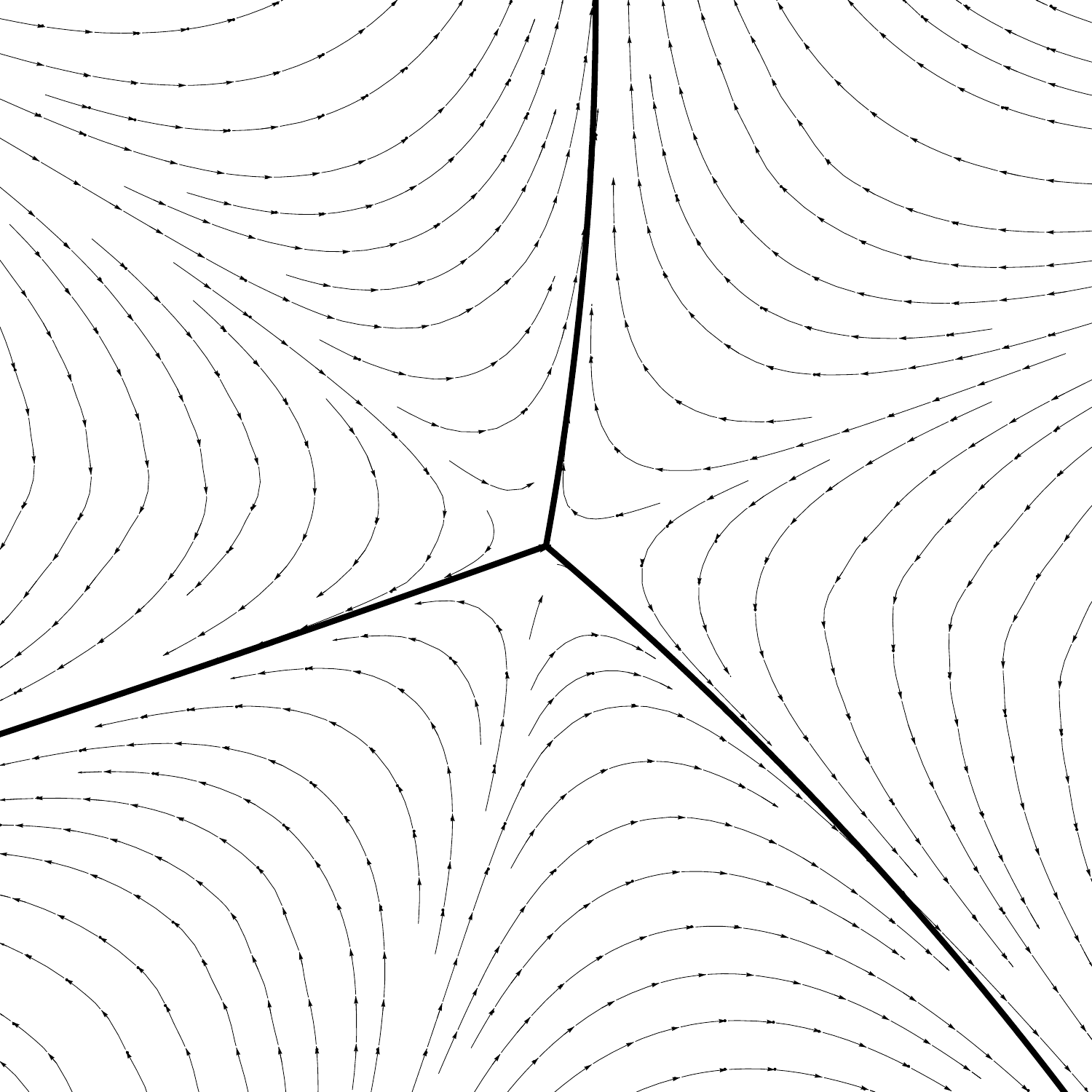}
 \caption{The vector field $-R(z){\partial_z}$ and $t$-trajectories for the above $P$ and $Q$ whose zeros are encircled. The figure on the right is a closeup near $z = 1$.
 Note how close the (black) curves are   to being separatrices of $-R(z){\partial_z}$ near this point.
 }\label{fig:tTrajectoriesExample}
\end{figure}

\end{example}

\subsection{Unbounded $t$-traces}

\begin{lemma}\label{le:rootsStartingInInfty}
Assume that $L=\deg Q-\deg P-1>0$. Then for all $u$ and all sufficiently small $t_0$ there exist $L$ components of $\ttt_u((0,t_0))$ that start at $\infty\in \setRS$. The arguments of the $t$-traces that tend to $\infty$ as $t\to 0^+$ tend to the arguments of the solutions of
\[
z^L+\frac{p_\infty}{q_\infty}=0. 
\]
\end{lemma}
\begin{proof}
The fact that there are $L$ components of $\ttt_u((0,t_0))$ emanating from $\infty$ is evident from the fact that \cref{eq:main} have $\deg P+1$ solutions in $\bC$ for $t=0$  and  $\deg Q$ solutions for $t>0$. Now, suppose that $\gamma(t)$ is a $t$-trace such that $t\to 0$ yields $\gamma(t)\to \infty$. We have that 
\[
t=-\frac{(\gamma(t)-u)P(\gamma(t))}{Q(\gamma(t))}.
\]
As $\gamma(t)\to \infty$
this yields
\[
t\gamma(t)^L+\frac{p_\infty}{q_\infty}=O(\gamma(t)^{-1})\quad \text{as $\gamma(t)\to \infty$},
\]
which has $L$ solutions $\gamma_j(t)=a_jt^{-\frac 1 L}+o(t^{-\frac 1 L})$, $j=1, \dots, L$, with $a_j^L=-\frac{p_\infty}{q_\infty}$.
Taking $t\to 0^+$, we get the required result.

\end{proof}
\begin{lemma}\label{le:rootsEndingInInfty}
As above, set $\deg Q-\deg P-1=L$ and assume that $L<0$.  Then for all $u$ and all sufficiently large $t_0$ there exist $-L$ components of $\ttt_u((t_0,\infty))$ which end at $\infty$. The arguments of the $t$-traces that tend to $\infty$ as $t\to \infty$ tend to the arguments of the solutions to
\[
z^{-L}+\frac{q_\infty}{p_\infty}=0.
\]
\end{lemma}
\begin{proof}

The fact that there are $-L$ components of $\ttt_u((t_0,\infty))$  going to $\infty$ is evident from \cref{eq:factor} since for $t\in \mathbb R$, the RHS has $\deg P+1$ solutions and as $t\to \infty$, $\deg Q$ roots will tend to the roots of $Q$ and the others will tend to $\infty$. Now, suppose that $\gamma(t)$ is a $t$-trace such that  $\gamma(t)\to \infty$ when $t\to \infty$. We have that 
\[
t=-\frac{(\gamma(t)-u)P(\gamma(t))}{Q(\gamma(t))}.
\]
For large $t$, this becomes 
\[
t\gamma(t)^{-L}+\frac{p_\infty}{q_\infty }=O(t^{-1/L})\quad \text{as $\gamma(t)\to \infty$},
\]
which has $L$ solutions $\gamma_j(t)=a_jt^{\frac 1 L}+o(t^{\frac 1 L})...$, $j=1, \dots, L$, with $a_j^L=-\frac{p_\infty}{q_\infty}$. Taking $t\to \infty$, we get the required result.
\end{proof}


\subsection {Associated rays and explicit criterion  of  $T_{CH}$-invariance} 

\begin{definition}
 In the above notation, for a point $p\in \setC$,
 define its \defin{associated ray} $r_p\coloneqq \{p+t R(p)\}$,
 where $t\in [0,\infty)$. Observe that $r_p$ is well-defined
 unless $p$ is a pole of $R(z)$. Additionally, if $p\in \zeros(R)$ then $r_p$ is merely a point. 
 \end{definition}
 
\begin{lemma}\label{lm:ray} 
Any point $p\in \setC$ which is not a pole of $R(z)$, lies on the root 
trail $\ttt_u$ if and only if $u\in r_p$. 
\end{lemma}
\begin{proof}
 Indeed, under our assumptions $p$ solves the equation $tQ(p)+(p-u)P(p)=0$ for some non-negative $t$. 
 If $P(p)\neq 0$, i.e., $p$ is not a pole of 
 $R(z)$ then 
 \[
 u-p=tR(p) \iff  u=p+tR(p)
 \]
 which means that $u\in r_p$.

\end{proof}
Note again that  unless $p$ is a root of $R(z)$ the ray $r_p$ does not degenerate to a point.

The following result will be intensively used throughout the rest of the paper. 

\begin{theorem}\label{th:charact}  For an operator $T$ given by \eqref{eq:1st} with  $P(z)$ and $Q(z)$ not  identically vanishing, 
   $S\subset \bC$ is $T_{CH}$-invariant set if and only if 
  
 {\rm (i)} $S$ is closed; 
  
{\rm   (ii)} $S$ contains $\zeros(PQ)$;  
  
 {\rm (iii)} the associated rays of all points in $S^c:=\bC\setminus S$ are contained in $S^c$.
\end{theorem}

Theorem~\ref{th:charact}  applies even  when $P(z)$ and $Q(z)$ 
are non-vanishing constants in which case item (ii) is empty. 
\begin{proof} 
Items (i) and (ii) are trivial. Indeed, a $T_{CH}$-invariant set must be closed by definition. 
Moreover, in \cref{prop:basic}, for an operator $T$ given 
by \eqref{eq:1st} with non-constant and not identically vanishing $Q(z)$
and $P(z)$, we have established  existence
and uniqueness of its minimal $T_{CH}$-invariant set  $\minvset{CH}$.
We have also shown that $\minvset{CH}$ must contain all the
zeros of $Q(z)$ and $P(z)$. 

To settle the necessity of item (iii) for the 
$T_{CH}$-invariance of $S$, notice the following. If $S$ is $T_{CH}$-invariant 
then no point in $S^c$ lies on  the root trail $\ttt_u$ of a point $u\in S$. 
Observe that $S^c$ is open in $\setC$.  Since $S$ contains both all roots and 
all poles of $R(z)$ we get that by \cref{lm:ray} the ray $r_p$ of any point $p\in S^c$
must completely lie in $S^c$. 
To prove the sufficiency of items (i) -- (iii) for the $T_{CH}$-invariance of $S$, 
we argue by contradiction. Assume that $S$ is not $T_{CH}$-invariant although (i) -- (iii) hold. 
This means that there exists a point $u_0\in S$ such that its root trail $\ttt_{u_0}$ leaves $S$. 
In other words, there is a point $p\in S^c$, $u_0\in S$ and some $t\geq 0$ such that solves
\[tQ(p)+(p-u_0)P(p)=0.\]
by (ii), $p\notin \zeros(P)$ so $p+tR(p)=u_0$.
This contradicts item (iii). 
\end{proof} 

Using the notation in \cref{re:affineChange}, note that
under the affine change of variables $z \mapsto aw +b$,
the associated ray $\{z_0+tR(z_0):t\geq 0\}$ 
is mapped to 
\[
\left\{\frac{z_0-b}{a}+\frac{t}{a}R(z_0):t\geq 0 \right\}=\{w_0+ t\hat{R}(w_0):t\geq 0\},
\]
where $z_0 = aw_0+b$. The family of affine maps is precisely the family
which sends straight lines to straight lines.

We also have this following alternative formulation of \cref{th:charact}.
\begin{corollary}
A set $S$ is $T_{CH}$-invariant if and only if its complement $S^c$ is open, 
does not contain $\zeros(P)$, and is forward-invariant under the family of maps
\[
\{z \mapsto z+tR(z), \; t\geq 0\}.
\]
If neither $Q(z)$ nor $P(z)$ are identically zero and at 
least one of them has positive degree, there exists a maximal open subset $S^c \subset \bC$ with the above  
properties and whose complement coincides with $\minvset{CH}$.
\end{corollary}

Under some natural additional assumptions conditions of item (iii) of Theorem~\ref{th:charact} can be verified only on the boundary of a set $S$ and not in the whole complement $S^c$.   Namely, we say that a closed set $M\subset \bC$  is \defin{regular} if it coincides with the closure of the set $M^o$ of its interior points.  

\begin{proposition}\label{prop:raysGiveCHinv}
Suppose that $S\subset \bC$ is a regular set such that
   
    \rm{(i)} $\zeros(PQ)$ lies in $ S^\circ$; 
    
    \rm{(ii)} 
    for every point $p\in \partial S$, 
the associated ray $r_p$ lies in $\overline{S^c}$.

Then $S$ is $T_{CH}$-invariant.
\end{proposition}
\begin{proof}
Suppose that $S$ is not $T_{CH}$-invariant.
Then there exist a $u_0\in S$ and a corresponding $t$-trace $\gamma(t)$
such that $\gamma(t_0)\notin S$ for some $t_0>0$.

Let $\epsilon_0>0$ be the distance between 
$\gamma(t_0)$ and $S$.
Note that  the roots of 
\[
tQ(z)+(z-u)P(z)=0
\]
depend continuously on the coefficients.
Moreover, since $u_0 \in S$ and $S$ is a regular domain
so for any neighborhood $U $ of $ u_0$, 
there is a point $u_1 \neq u_0$ in $S^\circ \cap U$. 

Hence, if $u_0 \notin S^\circ$, we can choose $u_1$ close enough to $u_0$
 to guarantee that there is a solution $z_1$ of
\[
t_0 Q(z)+(z-u_1)P(z)=0
\]
whose distance to $\gamma(t_0)$
is at most $\epsilon_0/2$.
In particular, $z_1 \notin S$.

But then there is a $t$-trace $\eta(t)$ such that
\[
tQ(\eta(t))+(\eta(t)-u_1)P(\eta(t))=0.
\]
Moreover, since all zeros of $PQ$ lie in  $S^\circ$ and any $t$-trace has a startpoint or an endpoint 
among the zeros of $PQ$, it follows that one or 
both of $\lim_{t\to 0}\eta(t)$ and $\lim_{t\to \infty} \eta(t)$ belong to $ S^\circ$.

Furthermore, there is at most 
one $\tau\in [0,\infty]$ such that $\lim_{t\to \tau}\eta(t)=\infty$ (this is true for any $t$-trace). 
Hence, by continuity, there is a $t_1\in (0,\infty)$
such that $\eta(t_1) \in \partial S$.
But then 
\[
u_1 = \eta(t_1) + t_1R(\eta(t_1))
\]
which contradicts the assumption that the associated 
rays of points in $\partial S \setminus \zeros(P)$ belong to   $\overline{S^c}$. 
\end{proof}

\begin{proposition}
If $S\subset \bC$ is such that there is a point $p\in \partial S$ for which $r_p$ intersects the interior of $S$, then $S$ is not $T_{CH}$-invariant. 
\end{proposition}
\begin{proof}
By continuity of $R$ there is a point in the exterior of $S$ close to $p$ such that its associated ray intersects $S$. By \cref{prop:raysGiveCHinv}, $S$ is not $T_{CH}$-invariant. 
\end{proof}

\section{Asymptotic geometry of $T_{CH}$-invariant sets} \label{sec:asympt}

\subsection{Ends of an unbounded subset of $\mathbb{C}$}
Below we introduce a number of notions important for the study of the asymptotic behavior  of unbounded  $T_{CH}$-invariant sets. 

\medskip
Given  an open set $X\subseteq \bC$, consider the sequence of disks $\mathbb{B}(0,n)$ centered at the origin and having radius $n \in \mathbb{N}$.

\begin{definition}
An \textit{end} of $X$ is a non-increasing sequence $U_{1} \supseteq U_{2} \supseteq \dots$ of subsets in $\bC$ such that for any $n$, $U_{n}$ is a connected component of $X \setminus \mathbb{B}(0,n)$.
\end{definition}

A sequence $(u_{n})_{n\in \mathbb{N}}$ of points in $X$ or a ray $z_0+\bR_+z_1 \subset X$ is said \emph{ to converge to the end $\kappa$} corresponding to the sequence $U_{1} \supseteq U_{2} \supseteq \dots$ as above if for any positive integer  $k$, there exists a positive integer  $\sharp(k)$ (resp. $t_k>0$) such that for any $n \geq \sharp(k)$, we have that $u_{n} \in U_{k}$ (resp. $z_0+tz_1\in U_k$).

\begin{definition}
For any end $\kappa$ of an open set $X \subset \mathbb{C}$, we define  the subset of $I_{\kappa}\subseteq \mathbb{S}^{1}$ formed by the accumulation points of sequences $(arg(u_{n}))_{n \in \mathbb{N}}$,  where $(u_{n})_{n\in \mathbb{N}}$  is a sequence of points in $X$ converging to $\kappa$. Similarly, we define  the union  $I_{X}:=\bigcup_\kappa I_{\kappa}\subseteq \mathbb{S}^1$ where $\kappa$ runs over the set of all ends of $X$.
\end{definition}

\begin{lemma}
For any end $\kappa$ of some open unbounded set $X$, $I_{\kappa}$ is a closed interval.
\end{lemma}

\begin{proof}
By definition, $I_{\kappa}$ is a closed subset of $\mathbb{S}^{1}$. Furthermore, for any pair of directions $\theta_{0}$ and $\theta_{1}$ in $I_{\kappa}$, there exist two sequences $(y_{n})_{n \in \mathbb{N}}$ and $(z_{n})_{n \in \mathbb{N}}$ of points in $X$ such that:
\begin{itemize}
    \item $(y_{n})_{n \in \mathbb{N}}$ and $(z_{n})_{n \in \mathbb{N}}$ converge to $\kappa$;
    \item $arg(y_{n}) \longrightarrow \theta_{0}$ and $arg(z_{n}) \longrightarrow \theta_{1}$;
    \item $arg(y_{n})$ and $arg(z_{n})$ are monotone in the cyclic order on the circle $\mathbb{S}^{1}$;
    \item for any $n \in \mathbb{N}$, $x_{n},y_{n} \in X \setminus \mathbb{B}(0,n)$.
\end{itemize}
For every $n$, we can find an arc $\gamma_{n}$ connecting $y_{n}$ and $z_{n}$ inside $X \setminus \mathbb{B}(0,n)$. It follows from the hypothesis that (up to a choice of  a subsequence), there exists an ascending family of intervals $I_{1} \subset I_{2} \subset \dots$ such that:
\begin{itemize}
    \item for each $n \in \mathbb{N}$, boundary points of $I_{n}$ coincide with $arg(y_{n})$ and $arg(z_{n})$;
    \item for each $n \in \mathbb{N}$, $I_{n}$ is contained in the image of $arg(\gamma_{n})$.
\end{itemize}
Finally, there exists a closed interval $I$ whose endpoints are $\theta_{0}$ and $\theta_{1}$ such that for any $\theta$, we can pick a point $z_{n}\in \gamma_{n}$ such that the sequence $\{z_{n}\}_{n \in \mathbb{N}}$ converges to $\kappa$ and the sequence $\{\arg(z_{n})\}$ converges to $\theta$.  

Consequently, $I_{\kappa}$ contains an interval connecting $\theta_{0}$ and $\theta_{1}$. Therefore $I_{\kappa}$ is connected implying that  $I_{\kappa}$ is topologically  a closed subinterval of $\mathbb{S}^{1}$ or coincides with the whole $\mathbb{S}^{1}$. 
\end{proof}

\subsection{Special compactification of the complex plane}

In this section we introduce a certain compactification of $\bC$ which we baptise as  the \textit{extended complex plane} $\mathbb{C} \cup \mathbb{S}^{1}\supset \mathbb{C}$. (Notice that the most frequently used compactification of $\bC$ is $\setRS=\bC P^1$.) \newline

\emph{The extended complex plane} $\mathbb{C} \cup \mathbb{S}^{1}$ is set-theoretically the union of $\bC$ and $\mathbb{S}^1$ endowed with the topology defined by the following basis of neighborhoods:
\begin{itemize}
    \item for a point  $x\in \mathbb{C}$, we choose the usual open neighborhoods of $x$ in $\mathbb{C}$;
    \item for a direction $\theta \in \mathbb{S}^{1}$, we choose open neighborhoods of the form $I \cup C(z,I)$ where $I$ is an open interval of $\mathbb{S}^{1}$ containing $\theta$ and $C(z,I)$ is an open cone with apex $z \in \mathbb{C}$ whose opening (i.e. the interval of directions) is $I$.
\end{itemize}

One can easily see that the extended plane $\mathbb{C} \cup \mathbb{S}^{1}$ is compact and homeomorphic to a closed disk. In particular, usual straight  lines in $\mathbb{C}$ have compact closures in $\mathbb{C} \cup \mathbb{S}^{1}$. (Below we will make no distinction between a real line in $\mathbb{C}$ and its closure in $\mathbb{C} \cup \mathbb{S}^{1}$).   \textit{Open half-planes} in $\mathbb{C} \cup \mathbb{S}^{1}$ are, by definition, connected components of the complement to a line.\newline

Given a $T_{CH}$-invariant  set $S \subset \mathbb{C}$, we denote by $\overline{S}$ its closure  in the extended plane $\mathbb{C}\cup \mathbb{S}^{1}$.\newline

\subsection{Ends of connected components of $S^{c}$}
Recall that $\sigma(z)=\arg R(z)$ and $r(z)=\{z+tR(z): t\geq 0\}$

\begin{lemma}\label{lem:ArcCone}
Given an $T_{CH}$-invariant $S\subset \bC$,  let $\gamma:[0,1]\rightarrow \mathbb{C}$ such that:
\begin{itemize}
    \item $\forall t \in (0,1)$, $\gamma(t) \in S^{c}$;
    \item $\sigma(\gamma(0)) \neq \sigma(\gamma(1))$;
    \item $\sigma(\gamma)$ is homotopic to the positive arc from $\sigma(\gamma(0))$ to $\sigma(\gamma(1))$.
\end{itemize}
If  $\Gamma$ denotes the connected component containing $(\sigma(\gamma(0)),\sigma(\gamma(1)))$ in the complement of $r(\gamma(0)) \cup \gamma \cup r(\gamma(1))$, then $\Gamma \subset S^{c}$.
\end{lemma}

\begin{proof}
Given a point $y \in \Gamma$, we will prove that there exists $t \in (0,1)$ such that $y \in r(\gamma(t))$. 
First, we introduce a closed loop $\alpha$ in the extended complex plane $\mathbb{C} \cup \mathbb{S}^{1}$. This loop $\alpha$ is formed by (in the given order):
\begin{itemize}
    \item path $\gamma$ from $\gamma(0)$ to $\gamma(1)$;
    \item associated ray $r(\gamma(1))$ from $\gamma(1)$ to $\sigma(\gamma(1))$;
    \item boundary arc $[\sigma  (\gamma(0)),\sigma (\gamma(1))]$ from $\sigma(\gamma(1))$ to $\sigma(\gamma(0))$ in positive direction;
    \item associated ray $r(\gamma(0))$ from $\sigma(\gamma(0))$ to
    $\gamma(0)$. 
\end{itemize}
Now we consider two continuous fields of directions on loop $\alpha$. The first one is given by $\arg(y-\alpha)$. Since $y \in \Gamma$, the index of this field is $-1$.\newline
The second field is defined in the following way:
\begin{itemize}
    \item on $\gamma$, the field coincides with $\sigma$ (direction of the associated ray);
    \item on $r(\gamma(1))$, the field is a constant equal to $\sigma (\gamma(1))$;
    \item at any point $\theta$ of boundary arc $[\sigma (\gamma(0)),\sigma (\gamma(1))]$, the field is given by $\theta$;
    \item on $r(\gamma(0))$, the field is a constant equal to $\sigma (\gamma(0))$.
\end{itemize}
Since $\sigma (\gamma)$ is defined to be hopotopic to the positive path from $\sigma (\gamma(0))$ to $\sigma (\gamma(1))$, this second field has index $0$ along the loop $\alpha$. \newline
The two continuous fields have distinct indices so there is exists a point on the loop $\alpha$ where the two fields coincide. The two fields are opposite on boundary arc $[\sigma  (\gamma(0)),\sigma (\gamma(1))]$. Besides, $y$ does not belong to $r(\gamma(0))$ or $r(\gamma(1))$ so a point of $\alpha$ where the two fields coincide automatically belongs to $\gamma$. There is $t \in ]0,1[$ such that $\sigma(\gamma(t))=\arg(y-\gamma(t))$. Thus, $y \in r(\gamma(t))$ and $y \in S^{c}$. Finally, we obtain that $\Gamma \subset S^{c}$.\newline
\end{proof}

\section{Topology of $T_{CH}$-invariant sets}\label{sec:topol} 

\subsection{Winding number}

Below we will need the following topological notion.

Consider a point $z_{0} \in \mathbb{C}$ and a closed loop  $\gamma$ in $\mathbb{C} \setminus \lbrace{ z_{0} \rbrace}$. We define the \emph{topological index} $\iota_{\gamma}(z_{0})$ as the degree of the map $\mathcal{C}_\gamma: \mathbb{S}^{1}\to \mathbb{S}^{1}$ given by $t \mapsto \arg(\gamma(t)-z_{0})$, $t\in \mathbb{S}^{1}$. (Notice that $\iota_{\gamma}(z_{0})$  is independent of the parametrization of $\gamma$.) 

\begin{definition}\label{def:assind} 
Let $\gamma$ be a closed loop in $\mathbb{C} \setminus \mathcal{Z}(PQ)$. The \textit{associated index} $\mathcal{A}_{\gamma}$ is, by definition,  the degree of the map of $\mathcal{C}_{\sigma,\gamma}: \mathbb{S}^{1}\to \mathbb{S}^{1}$ defined as $t \mapsto \sigma \circ \gamma(t)$.  (Recall that $\sigma(z)=\arg R(z)$).  
\end{definition}

\begin{lemma}\label{lem:Index}
Given  a closed loop  $\gamma$  in $\mathbb{C} \setminus \mathcal{Z}(PQ)$,  let $z_{0}$ be a point of $\mathbb{C} \setminus \gamma$. Then at least one of the following statements holds:
\begin{itemize}
    \item[\rm(i)] $\exists t \in \mathbb{S}^{1}$ such that $z_{0}$ belongs to the associated ray $r(\gamma(t))$ of some point $\gamma(t)$;
    \item[\rm(ii)] $\mathcal{A}_{\gamma}=\iota_{\gamma}(z_{0})$.
\end{itemize}
\end{lemma}

\begin{proof}
Indeed if $\mathcal{A}_{\gamma} \neq \iota_{\gamma}(z_{0})$, then there exists $t \in \mathbb{S}^{1}$ such that $R\circ \gamma(t)$ has the same argument as $z_{0}-\gamma(t)$. Consequently, the half-line starting at $\gamma(t)$ in the direction of $z_{0}-\gamma(t)$ coincides with the associated ray $r(\gamma(t))$. In other words, $z_{0}$ belongs to  $r(\gamma(t))$.
\end{proof}

\begin{lemma}[Argument principle]\label{lem:DegreeDisk}
For an oriented closed loop $\gamma:\mathbb{S}^{1}\longrightarrow \mathbb{C}\setminus\mathcal{Z}(PQ)$ bounding a topological disk $\mathcal{D}$, the topological degree of $\arg(R \circ \gamma)$ coincides with the sum of the degrees of zeros and poles  of $R(z)$ contained in $\mathcal{D}$.  
\end{lemma}


We start our topological discussions with the most basic  property of $T_{CH}$-invariant sets. 
\subsection{On (dis)connectedness of $T_{CH}$-invariant sets}
Observe that for any closed subset $S\subset \bC$,  its complement  $S^{c}$ is  open in $\mathbb{C}$ and its connected components coincide with its path-connected components.  

\begin{lemma}\label{lem:contractible}
Every connected component of a $T_{CH}$-invariant set  $S\subset \bC$ is contractible.
\end{lemma}

\begin{proof}
Assume that a connected component $S_{0}\subset S$ is non-contractible, then there exists a simple loop $\gamma \subset S_{0}$ such that a connected component $X$ of $S^{c}$ is contained in the bounded part of $\mathbb{C} \setminus \gamma$. However, for any $x\in X$, the associated ray $r(x)$ lies in $X$.  Therefore $X$ cannot be bounded.  
\end{proof}

\begin{proposition}\label{prop:compact}
For a $T_{CH}$-invariant set $S\subset \bC$, if there exists a connected component $X$ of $S^{c}$ which is not simply connected, then $S$ is contractible and compact. In the latter case, for the operator $T$,  one has $\deg Q-\deg P=1$.
\end{proposition}

\begin{proof}
We consider a non-contractible positively oriented simple loop $\gamma$ in $X$ encompassing 
a (compact) connected component $S_{1}$ of $S$.
For a point $z_{0} \in S_{1}$, the associated ray $r(\gamma(t))$ of any point $\gamma(t)$ does not contain $z_{0}$.
Then, following Lemma~\ref{lem:Index}, we have $\mathcal{A}_{\gamma}=\iota_{\gamma}(z_{0})=1$. However for any $z$ belonging to the unbounded part $\mathcal{U}$ of $\mathbb{C} \setminus \gamma$, we have $\iota_{\gamma}(z)=0$.
 Therefore Lemma~\ref{lem:Index} implies that $z$ belongs to some associated ray $r(\gamma(t))$.
Consequently, $S \cap \mathcal{U} = \emptyset$. This reasoning applies to every non-contractible loop of $S^{c}$. Thus, $S$ is connected and contractible by Lemma~\ref{lem:contractible}.
Using Lemma \ref{lem:DegreeDisk}, we conclude that $\mathcal{A}_{\gamma}$ coincides with the sum of multiplicities of zeros and poles of $R(z)$ inside $\mathcal{D}$.
Since every root of $PQ$ belongs to $S$, we obtain $\deg Q-\deg P=\mathcal{A}_{\gamma}=1$.
\end{proof}

\smallskip
\subsection{Discussion of connectedness for non-compact $T_{CH}$-invariant sets} 

Observe  that without the assumption of compactness,  connectedness of $T_{CH}$-invariant sets might fail.  For example, consider an operator $T=Q(z)\frac{d}{dz}+P(z)$ with  $Q(z)=\lambda P(z)$ where $\lambda \in \bC^\ast$, i.e. $R(z)\equiv \lambda\neq 0 $ is a constant function.  Then a closed subset $S\subset \bC$ containing roots of $PQ$ is $T_{CH}$-invariant if and only if it is closed under positive translations in the $\theta$-direction where $\theta$ is the argument of $\lambda$. In particular, in this case there exist disconnected $T_{CH}$-invariant sets.
 
Furthermore, it is not hard to find operators $T$  with $R(z)\not \equiv const$ for which there still exist disconnected $T_{CH}$-invariant sets. In particular,  one can take  $Q=P+\epsilon$ with  sufficiently small $\epsilon\in \bC^\ast$.

 To provide an explicit example, take \[Q=z^2+1,\; P=Q+\frac{1}{10}=z^2+\frac{11}{10}.\] Let $A$ be the intersection of the two sets: \[S_1=\left\{z\in \bC \vert \; \Im z\in \left[\frac{1}{2},2-\Re z\right]\right\}\]  and
  \[S_2=\left\{z\in \bC\vert\; \Im z\geq \frac{1}{2}+\Re z\right\}.\]  Let $B$ be the reflection of $A$ in the real axis. We shall prove that $M=A\cup B$ is $T_{CH}$-invariant. It is clear that $M$ is disconnected in $\bC$ and $T_{CH}$-invariance of $M$ will prove the disconnectedness of $\minvset{CH}$.
  
Firstly, observe that $M$ is regular and that the zeros of $P,Q$ are in the interior of $M$. We shall study the associated rays $r(x)$ of every point $x$ on the boundary of $M$ and note that they do not intersect the interior of $M$ which by Proposition~\ref{prop:raysGiveCHinv}  will imply that $M$ is indeed $T_{CH}$-invariant. 

By the symmetry of $M$,  it suffices to consider $\partial A$. Firstly, take $z(s)\; \vert\; \Im z=1/2,\; \Re z=s\in (-\infty,0]$. Then the associated ray $r(z(s))$ at $z(s)$ is given by $s+i/2+tR(s+i/2)$ which, in its turn,  equals    \[s+\frac{i}{2}+\left(1-\frac{1}{10 \left(\frac{11}{10}+\left(s+\frac{i}{2}\right)^2\right)}\right) t.\] Hence, for $s=0$,  the associated ray is parallel to the real axis and does not intersect it. If however $s<0$, then the intersection of the associated ray with the real line occurs at the point $-5 s^3-12 s-\frac{51}{16 s}$.   We can easily see that this expression is strictly greater than 1. In particular, the associated ray $r(z(s))$ does not intersect $M$ in any other point different from  its starting point $z(s)$.

Next, we consider the associated rays at the points $z(s)=i/2+s(1+i),\; s\in (0,1]$. We can check that $\arg(R(z))\in (0,\pi/4)$ and hence that the associated rays of the latter points do not intersect $M$ at any other point either. 

Finally, it is easy to see that for a point $z$ belonging to  the remaining part of the boundary one has the following. If $\Re z\geq 0$, then the associated ray at $z$ does not intersect the real axis and if $\Re z<0$, then it does so at a point with real part greater than $1$. Therefore we conclude that on all three parts of the boundary of $A$, the associated rays do not intersect $M$ at any other point than the initial. By symmetry the same is true for $B$ and we obtain that $M$ is $T_{CH}$-invariant by \cref{prop:raysGiveCHinv}.

\medskip
\begin{remark} 
Observe that every $T_{CH}$-invariant set $S$ is connected if and only if the minimal invariant set $\minvset{CH}$ is connected. 
We plan to return to the question about connectedness of $\minvset{CH}$ in  \cite{AHNST}.
\end{remark}

\subsection{Happy ends} 

\begin{lemma}\label{lem:ConstantEnd} For a $T_{CH}$-invariant set $S\subset \bC$, 
let $X$ be a connected component of $S^{c}$. Then there exists a unique topological end $\kappa$ of $X$ such that for every point $x \in X$, the ray $r(x)$ converges to $\kappa$. In particular, $X$ is unbounded.
\end{lemma}

\begin{proof}
For any $x \in X$, $R(x) \in \mathbb{C}^{\ast}$ (because every root of $PQ$ belongs to $S$). Thus, the associated ray $r(x)$ is a half-line contained in  $X$ which implies that $X$ is unbounded. We denote by $\kappa$ the end of $X$ to which the ray $r(x)$ converges. The end  $\kappa$ is invariant under small perturbations of $x$. Thus, it is a topological invariant of the whole component $X$.
\end{proof}

\begin{definition}
We will call the above special end as the \textbf{happy end} of $X$. 
\end{definition} 

\begin{lemma}\label{lem:interval}
For a non-trivial $T_{CH}$-invariant set $S\subset \bC$, let $X$ be a connected component of $S^{c}$. Then $\sigma(X)$ is an open interval of $\mathbb{S}^{1}$.
\end{lemma}

\begin{proof}
Since $R(z)$ is non-constant we get that both $R(z)$ and $\sigma$ are  open and continuous maps.
Since $S$ is closed, $X$ is open. Consequently, $\sigma(X)$ is an open connected subset of $\mathbb{S}^{1}$.
\end{proof}

\begin{lemma}\label{lem:argkappa}
For a non-trivial $T_{CH}$-invariant set $S\subset \bC$, let $X$ be a connected component of $S^{c}$ and $\kappa$ be the happy end of $X$. Then $\sigma(X) \subset I_{\kappa}$.
\end{lemma}

\begin{proof}
If $\theta \in \sigma(X)$, then there exists $z_{0} \in X$ such that for any $t >0$, one has that $z_{0}+te^{i\theta} \in X$. Following Lemma \ref{lem:ConstantEnd},  this associated ray converges to  the happy end $\kappa$ of $X$. Thus $\theta \in I_{\kappa}$.
\end{proof}

For $z \in \mathbb{C}$ and an open interval  $I$  of $\mathbb{S}^{1}$, the \textit{cone} $C(z,I)$ is the $1$-parameter family of half-lines starting at $z$ in the directions of $I$.

\begin{lemma}\label{lem:cone}
For a non-trivial $T_{CH}$-invariant set $S\subset \bC$, let $X$ be a connected component of $S^{c}$ and $\kappa$ be its happy end. If $\theta \in \sigma(X)$, then there exists $z_{0} \in X$ and an open interval $I \subset \mathbb{S}^{1}$ containing $\theta$ such that:
\begin{itemize}
    \item the cone $C(z_{0},I)$ is contained in $X$;
    \item every half-line in the cone $C(z_{0},I)$ converges to $\kappa$.
\end{itemize}
\end{lemma}

\begin{proof}
If $\theta \in \sigma(X)$, then there exists $z \in X$ such that $\sigma(z)=\theta$. Then we can find a curve $\gamma(t):I\to X$  and an open interval $I$ containing $\theta$ such that $I$ coincides with $(\sigma(\gamma(0)),\sigma(\gamma(1)))$. It follows from Lemma~\ref{lem:ArcCone} that $S^{c}$ contains some open cone $C(z_{0},I)$.

Finally, outside a compact subset, the associated ray $r(z)$ is contained in $C(z_{0},I)$. This implies that $C(z_{0},I)$ belongs to $X$ and that every half-line of the cone $C(z_{0},I)$ converges to $\kappa$.
\end{proof}

\begin{corollary}\label{cor:TourCone} For a $T_{CH}$-invariant set $S\subset \bC$, 
let $X$ be a component of $S^{c}$. If $\sigma(X)=\mathbb{S}^{1}$, then $X^{c}$ is compact and so is $S$.
\end{corollary}

\begin{proof}
For each $\theta \in \mathbb{S}^{1}$, there is a point $z_{\theta} \in X$ and an open interval $I_{\theta}$ containing $\theta$ such that the cone $C(z_{\theta},I_{\theta})$ is contained in $X$, see Lemma~\ref{lem:cone}. 

Since $\mathbb{S}^{1}$ is compact, we can extract a finite collection of such cones whose intervals of directions form a cover of $\mathbb{S}^{1}$. The complement to the union of these cones is compact. Thus, the complement $X^{c}$ to $X$ in the extended plane  is compact.
\end{proof}

\begin{lemma}\label{lem:disjoint}
For a non-trivial $T_{CH}$-invariant set $S\subset \bC$, let $X,Y$ be two (possibly identical) connected components of $S^{c}$. Let $\kappa$ be an end of $Y$  different from  the happy end of $X$. Then  $\sigma(X) \cap I_{\kappa}=\emptyset$.
\end{lemma}

\begin{proof}
We assume by contradiction that $\theta \in \sigma(X) \cap I_{\kappa}$. Following Lemma~\ref{lem:cone}, there exists $z_{0} \in X$ and an open interval $I$ containing $\theta$ such that the cone $C(z_{0},I)$ is contained in $X$ and such that each of its half-lines converges to the happy end of $X$.

Furthermore there exists a sequence $(y_{n})_{n \in \mathbb{N}}$ of points of $Y$ converging to the end $\kappa$ and such that $\arg(y_{n})$ tends to $\theta$. Then there is a bound $N>0$ such that for every $n \geq N$, $y_{n}$ belongs to the cone $C(z_{0},I)$. Consequently, as $n \to \infty$, $y_{n}$ also tends to the happy end of $X$ (just like every half-line of the cone) which thus coincides with $\kappa$. This contradicts to the assumption  and implies that $\sigma(X) \cap I_{\kappa}=\emptyset$.
\end{proof}

\begin{corollary}\label{cor:distinctarg}
For a non-trivial $T_{CH}$-invariant set $S\subset \bC$, let $X,Y$ be two different  connected components of $S^{c}$, then $\sigma(X) \cap \sigma(Y) = \emptyset$.
\end{corollary}

\begin{proof}
The claim follows from Lemmas \ref{lem:argkappa} and \ref{lem:disjoint}.
\end{proof}

\smallskip
Let us present one application of the notion of happy end.  

\begin{lemma}\label{lem:convex}
For a non-trivial $T_{CH}$-invariant set $S\subset \bC$, take $x,y \in S$, consider a connected component $X$ of $S^{c}$,  and let $\kappa$ be the happy end of $X$.  Define $X_{0}$ as  the connected component of $S^{c} \cap [x,y]^{c}$ containing  $\kappa$.  
Then for any $z \in X_{0}$, its associated ray $r(z)$ is disjoint from $[x,y]$.
\end{lemma}

\begin{proof}
Consider $z \in X_{0}$ and its associated ray $r(z)$. As $r(z)$ connects $z$ and $\kappa$, then if $r(z)$ crosses the boundary of $X_{0}$,  it should cross it an even number of times. However,  by definition, $r(z)$ is  disjoint from $S$. Moreover since two lines intersect each other (at most) once then if $r(z)$ crosses $[x,y]$,  it should cross it at most once. Consequently, $r(z)$ remains disjoint from $[x,y]$.
\end{proof}

\begin{corollary}
Let $S\subset \bC$ be a $T_{CH}$-invariant set and $X$ be a connected component of $S^{c}$. Then $X^{c}$ is $T_{CH}$-invariant. Moreover, its convex hull $Conv(X^{c})$ is also $T_{CH}$-invariant.\newline
\end{corollary}

\subsection{Directions of ends}
Recall the following definition. We define $p_{\infty},q_{\infty} \in \mathbb{C}^{\ast}$, and $p,q \in \mathbb{N}$ such that $P(z)=p_{\infty}z^{p}+o(z^{p})$ and $Q(z)=q_{\infty}z^{q}+o(z^{q})$.\newline
Then, we have $\lambda = \frac{q_{\infty}}{p_{\infty}} \in \mathbb{C}^{\ast}$ and $\phi_{\infty}=\arg(\lambda)$ (the reader may recall the definiition from \cref{not:pq}).\newline

We introduce the function $d:\mathbb{S}^{1}\longrightarrow \mathbb{S}^{1}$ such that $d(\theta)=\phi_{\infty}+(\deg Q-\deg P)\theta$.

\begin{lemma}\label{lem:SlopesCC2} For a $T_{CH}$-invariant set $S\subset \bC$, 
let $X$ be a connected component of $S^{c}$. Then the following statements hold:
\begin{itemize}
    \item if $\deg Q-\deg P \neq 0$, then $\sigma(X)$ is a  connected open subset of $\mathbb{S}^{1}$ and for any $\theta \in \sigma(X)$, $d(\theta)$ belongs to ${\sigma(X)}$;
    \item if $\deg Q-\deg P=0$, then for any $\theta \in \sigma(X)$, $d(\theta)=\phi_{\infty}$ belongs to $\overline{\sigma(X)}$.
\end{itemize}
\end{lemma}

\begin{proof}
For any $\theta \in \sigma(X)$, there is a point $z_{0} \in X$ and an open interval $I$ containing $\theta$ such that the cone $\mathcal{C}=C(z_{0},I)$ is contained in $X$ (see Lemma~\ref{lem:cone}). $R(\mathcal{C})$ is an open set of $\mathbb{C}$ and $\sigma(C)$ is an open interval of $\mathbb{S}^{1}$.

Further for any $\eta \in I$, $R(z_{0}+te^{i\eta}) \simeq \frac{q_{0}}{p_{0}}(te^{i\eta})^{\deg Q-\deg P}$ as $t$ tends to $+\infty$.  Thus, $\sigma(z_{0}+te^{i\eta})$ tends to $d(\eta)$ and $d(\eta)$ is in the closure of $\sigma(\mathcal{C})$. 

Next if $\deg Q-\deg P\neq 0$, then $d(I)$ is an open interval. The closure of $\sigma(\mathcal{C})$ is a closed interval containing the open interval $d(I)$. Consequently we obtain $d(I) \subset \sigma(\mathcal{C})$ and thus $d(\theta) \in \sigma(X)$.
\end{proof}

\begin{corollary}\label{cor:RawArg}
For a connected component  $X$ of $S^{c}$,  one of the following statements holds:
\begin{itemize}
    \item[\rm{(i)}] $\sigma(X)=\mathbb{S}^{1}$;
    \item [\rm{(ii)}] $\deg Q-\deg P =1$ and $\phi_{\infty}=0$;
    \item [\rm{(iii)}] $\deg Q-\deg P =0$ and $\phi_{\infty}$ belongs to $\overline{\sigma(X)}$;
    \item[\rm{(iv)}] $\deg Q-\deg P = -1$ and either $\frac{\phi_{\infty}}{2}$ or $\frac{\phi_{\infty}}{2}+\pi$ is the bisector of open interval $\sigma(X)$. Besides, the length of $\sigma(X)$ is at most $\pi$.
\end{itemize}
\end{corollary}

\begin{proof}
We first assume $\deg Q-\deg P \neq 0$. Following Lemma~\ref{lem:interval}, then $\sigma(X)$ is an open interval of $\mathbb{S}^{1}$. Following Lemma~\ref{lem:SlopesCC2}, $\sigma(X)$ is invariant under the action of $d$. Consequently, if $|\deg Q-\deg P| \geq 2$, then $\sigma(X)$ coincides with $\mathbb{S}^{1}$.\newline
If $\deg Q-\deg P=1$, then $\sigma(X)$ is an open interval preserved by  rotation by the angle $\phi_{\infty}$. Thus, either $\sigma(X)$ coincides with the whole circle of directions or the rotation is trivial.\newline
If $\deg Q-\deg P=0$, Lemma~\ref{lem:interval} proves that $\phi_{\infty}$ belongs to $\overline{\sigma(X)}$.\newline
If $\deg Q-\deg P=-1$, then $\sigma(X)$ is an open interval preserved by the involution $\theta \mapsto \phi_{\infty}-\theta$. This implies that the line with the slope $\frac{\phi_{\infty}}{2}$
is the symmetry axis of the interval $\sigma(X)$. If the length of $\sigma(X)$ is strictly bigger than $\pi$, then we consider the family 
of open cones of directions $\sigma(X)$ contained in $X$. Lemma~\ref{lem:ArcCone} proves that such cones exist. We consider a maximal cone $\Gamma$ of this family ordered by inclusion. If $\Gamma \neq \mathbb{C}$, the same lemma proves that associated rays of $\Gamma$ cover a cone strictly bigger than $\Gamma$. This is impossible since $S^{c}$ is nonempty. It follows that the length of $\sigma(X)$ is at most $\pi$ for any connected component $X$ of $S^{c}$.
\end{proof}

\begin{lemma}\label{lem:multipleEnds}
Let $X$ be a connected component of $S^{c}$. We assume that $X$ has at least two distinct topological ends. Let $\kappa$ be its happy end and $\kappa'$ be some other end. Then one of the following statement holds:
\begin{itemize}
    \item $\deg Q-\deg P=0$ and $I_{\kappa'}=\lbrace{ \phi_{\infty}+\pi \rbrace}$;
    \item $\deg Q-\deg P=-1$ and $I_{\kappa'}$ is a singleton $\lbrace{ \theta \rbrace}$ where $\theta \equiv \frac{\phi_{\infty}+\pi}{2}~[\pi]$. Besides, $\sigma(X)$ coincides with either $(\frac{\phi_{\infty}-\pi}{2},\frac{\phi_{\infty}+\pi}{2})$ or $(\frac{\phi_{\infty}+\pi}{2},\frac{\phi_{\infty}+3\pi}{2})$.
\end{itemize}
In both cases, the interior of the closure of $X$ in $\mathbb{C} \cup \mathbb{S}^{1}$ has  connected complement (in $\mathbb{C} \cup \mathbb{S}^{1}$).  
\end{lemma}

\begin{proof}
Let $(z_{n})_{n \in \mathbb{N}}$ be a sequence of points of $X$ converging to $\kappa'$ and such that $\arg(z_{n})$ converges to $\theta \in I_{\kappa'}$. Then there exists a compact $K$ such that $\kappa$ and $\kappa'$ belong to the closures of distinct components of $X \cap K^{c}$. Thus, the associated rays $r(z_{n})$ have to cross $K$ to converge to $\kappa$. 

As $n\to \infty$, $\arg(z_{n})$ converges to $\theta$ and thus the slopes of $r(z_{n})$ converge to $\theta + \pi$. Finally, we obtain that $\phi_{\infty}+(\deg Q-\deg P)\theta=\theta+\pi$.     
Since $X$ has several topological ends, the complement of $X$ cannot be compact. Therefore Corollary~\ref{cor:TourCone} implies that $\sigma(X) \neq \mathbb{S}^{1}$. Corollary~\ref{cor:RawArg} then gives restrictions on $\deg P-\deg Q$ and $\phi_{\infty}$. 

We treat separately the three cases depending on the value of $\deg Q-\deg P$.\newline
If $\deg Q-\deg P=1$, $\phi_{\infty}+(\deg Q-\deg P)\theta=\theta+\pi$ implies that $\phi_{\infty}=\pi$. This contradicts the conclusion of Corollary~\ref{cor:RawArg}.

If $\deg Q-\deg P=0$, then we obtain $\phi_{\infty}=\theta+\pi$. Thus, $I_{\kappa'}$ coincides with the singleton $\lbrace{ \phi_{\infty}+\pi \rbrace}$.

If $\deg Q-\deg P=-1$, then we obtain $\phi_{\infty}=2\theta+\pi$. Thus, $\theta \equiv \frac{\phi_{\infty}+\pi}{2}~[\pi]$ and $I_{\kappa'}$ is automatically a singleton. Besides, the following facts are true:
\begin{itemize}
    \item $I_{\kappa'}$ is disjoint from $\sigma(X)$, see Lemma~\ref{lem:disjoint};
    \item $\sigma(X)$ is an open interval of $\mathbb{S}^{1}$ invariant under the map $d$, see Lemma~\ref{lem:SlopesCC2};
    \item $d(\theta)$ belongs to the closure of $\sigma(X)$;
    \item $d$ is an involution ($\deg Q-\deg P=-1$).
\end{itemize}
It follows that $\theta+\pi$ does not belong to $\sigma(X)$ either. Consequently $\sigma(X)$ coincides with either $(\frac{\phi_{\infty}-\pi}{2},\frac{\phi_{\infty}+\pi}{2})$ or with $(\frac{\phi_{\infty}+\pi}{2},\frac{\phi_{\infty}+3\pi}{2})$. 

In order to prove that the interior of the closure of $X$ in the extended plane has a connected complement it suffices to observe that this interior has a connected intersection $\sigma(X)$ with $\mathbb{S}^{1}$. Besides, it is contractible because otherwise some connected complement of $S$ would be compact which by Proposition~\ref{prop:compact} proves can only happen  if $\deg Q-\deg P=1$.
\end{proof}

\begin{corollary}\label{cor:Raw1}
For a $T_{CH}$-invariant set $S\subseteq \bC$,  one of the following statements holds:
\begin{itemize}
    \item[\rm{(i)}] $S=\mathbb{C}$, i.e. $S$ is trivial;
    \item[\rm{(ii)}] $\deg Q-\deg P = 1$ and $S$ is compact and contractible;
    \item[\rm{(iii)}] $\deg Q-\deg P = 1$, $S$ is non-trivial (i.e. different from $\bC$) and non-compact and $\phi_{\infty}=0$; 
    \item[\rm{(iv)}] $\deg Q-\deg P =0$ and $S$ is non-trivial  and non-compact;
    \item[\rm{(v)}] $\deg Q-\deg P =-1$ and $S$ is non-trivial and  non-compact.
\end{itemize}
\end{corollary}

\begin{proof}
We assume that $S$ is non-trivial, i.e. does not coincide with $\mathbb{C}$. Let $X$ be a connected component of $S^{c}$. It follows from Corollary \ref{cor:TourCone} that if $\sigma(X)$ coincides with $\mathbb{S}^{1}$, then $S$ is compact and thus $S^{c}$ is not simply connected. Proposition \ref{prop:compact} then proves that $S$ is compact and contractible. Besides we have $\deg Q-\deg P=1$.

Next let us assume that items (i) and (ii) do not hold. This implies in particular that $\sigma(X)$ does not coincide with $\mathbb{S}^{1}$. Corollary~\ref{cor:RawArg} then proves that $|\deg Q-\deg P|\leq 1$. In each case, there are constraints on $\phi_{\infty}$. In particular, if $\deg Q-\deg P=1$, the only possible case implies $\phi_{\infty}=0$.
\end{proof}

\begin{remark} In \S~\ref{sec:triviality} we provide more details on the realizability of the cases mentioned in Corollary~\ref{cor:Raw1}. 
\end{remark}

As we have shown above non-trivial $T_{CH}$-invariant sets can only occur if $\deg Q-\deg P=-1, 0, 1$. Let us provide more details in each of  these three cases below. 

\begin{proposition}
For $\deg Q-\deg P=1$ and  $S$ being a nontrivial $T_{CH}$-invariant set (i.e. different from $\mathbb{C}$), one of the following statements holds:
\begin{itemize}
    \item[{\rm (i)}] $S$ is compact and contractible;
    \item[{\rm (ii)}] $S$ is non-compact and $\phi_{\infty}=0$.
\end{itemize}
\end{proposition}

\begin{proof}
Both claims follow from Corollary \ref{cor:Raw1}.
\end{proof}

In particular, if we consider the case $Q(z)=z$ and $P(z)=\lambda$ where $\lambda \in \mathbb{R}_{+}^{\ast}$, any radially invariant set is $T_{CH}$-invariant. Therefore, $S^{c}$ can have an arbitrarily large number of connected components.\newline

\begin{proposition} For $\deg Q-\deg P=0$, assume that $S$ is a non-trivial $T_{CH}$-invariant set. For any connected component $X$ of $S^{c}$, $\phi_{\infty}$ belongs to $\overline{\sigma(X)}$. Moreover,  $S^c$ has at most two connected components and if $\kappa$ is an end of $X$ 
that is not the happy end of $X$, then $\kappa$ satisfies the condition: 
\[I_{\kappa}= \lbrace{ \phi_{\infty}+\pi \rbrace}.\]
\end{proposition}

\begin{proof}
 Corollary \ref{cor:Raw1} implies that $S$ is non-compact. Corollary~\ref{cor:TourCone} shows that for any connected component $X$ of $S^{c}$, $\sigma(X) \neq \mathbb{S}^{1}$. Corollary~\ref{cor:RawArg} then proves that $\phi_{\infty}$ belongs to $\overline{\sigma(X)}$.
 
 Further following Corollary~\ref{cor:distinctarg}, if there are two distinct connected components $X$ and $Y$ of $S^{c}$, then $\sigma(X) \cap \sigma(Y) = \emptyset$ while $\phi_{\infty}$ belongs $\overline{\sigma(X)}\cap\overline{\sigma(Y)}$. Consequently, $S^{c}$ has at most two connected components. Lemma~\ref{lem:contractible} proves that each of them is contractible.
 
By Lemma~\ref{lem:argkappa}, if an end $\kappa$ is  different from the happy end of  some connected component of $S^{c}$, then $I_{\kappa}= \lbrace{ \phi_{\infty}+\pi \rbrace}$.
\end{proof}

\begin{remark} 
Just like in the (trivial) case when $R(z)\equiv const$ we expect that $S^{c}$ can have an arbitrarily large number of ends.
\end{remark} 

\begin{proposition}
For $\deg Q-\deg P=-1,$ let $S$ be a non-trivial $T_{CH}$-invariant set.  Then $S^{c}$ has at most two connected components each of which is contractible.

For each connected component $X$ of $S^{c}$, one of the following statements holds:
\begin{itemize}
    \item[\rm{(i)}] $X$ has exactly one end;
    \item[\rm{(ii)}] $\sigma(X)$ is an open interval $(\theta_{0}-\frac{\pi}{2},\theta_{0}+\frac{\pi}{2})$ where $\theta_{0} \equiv \frac{\phi_{\infty}}{2}~[\pi]$. For any end $\kappa$ of $X$ distinct from the happy end, $I_{\kappa}$ coincides with either $\lbrace{ \theta_{0}-\frac{\pi}{2} \rbrace}$ or $\lbrace{ \theta_{0}+\frac{\pi}{2} \rbrace}$.
\end{itemize}

\end{proposition}

\begin{proof}
 Corollary \ref{cor:Raw1} implies that $S$ is non-compact. Corollary~\ref{cor:TourCone} proves that for any connected component $X$ of $S^{c}$, $\sigma(X) \neq \mathbb{S}^{1}$. Corollary~\ref{cor:RawArg} then implies that the diameter of the slope $\frac{\phi_{\infty}}{2}$ is the symmetry axis of the interval $\sigma(X)$. Since there are at most two  such disjoint symmetric intervals (see Corollary~\ref{cor:distinctarg}), then $S^{c}$ has at most two connected components. Lemma~\ref{lem:contractible} then provides that each of them is contractible.
 
Lemma~\ref{lem:multipleEnds} settles the case of connected components of $S^{c}$ having several topological ends.\newline
\end{proof}

\begin{theorem}\label{thm:connectedclosure}
For any $T_{CH}$-invariant set $S\subset \bC$, its closure $\overline{S}$ in the extended plane $\mathbb{C}\cup \mathbb{S}^{1}$ is connected and contractible. 
\end{theorem}

\begin{proof}
The statement is trivially true if $S=\mathbb{C}$. If $S$ has a compact connected component in $\mathbb{C}$, then Proposition~\ref{prop:compact} proves that $S$ coincides with this component.\newline
In the remaining cases, every connected component of $S$ is unbounded. Let $S_{0}$ be such a connected component. If $\overline{S}$ is not connected in the extended plane $\mathbb{C}\cup \mathbb{S}^{1}$, then there exist two connected components $\overline{S_{1}}$ and $\overline{S_{2}}$ of $S$ such that each of them belong to a distinct connected component of the complement of $\overline{X}$. Here, $\overline{X}$ is the interior of the closure in $\mathbb{C}\cup \mathbb{S}^{1}$ of a connected component $X$ of $S^{c}$. Lemma~\ref{lem:multipleEnds} rules out this case.   

It remains to show that $\overline{S}$ is contractible. If this were not true, then we can find a connected component $X$ of $S^{c}$ such that the interior of its closure in $\mathbb{C}\cup \mathbb{S}^{1}$ has no intersection with $\mathbb{S}^{1}$. However, this intersection always contains $\sigma(X)$ which is an open interval. This claim finishes the proof.
\end{proof}

\section{Criterion for the existence of non-trivial $T_{CH}$-invariant subsets}\label{sec:triviality}

In the previous section we have shown that any $T_{CH}$-invariant set for any operator $T$ with $|\deg Q- \deg P|>1$ is trivial, i.e. equal to $\bC$. In this section we will show that in two of the remaining cases $\deg Q -\deg P=-1$ and $\deg Q -\deg P=0$ non-trivial $T_{CH}$-invariant sets always exist and, in particular, $\minvset{CH}$ is non-trivial. For the existence of non-trivial $T_{CH}$-invariant subsets in the last remaining case $\deg Q-\deg P=1$, the natural additional restriction $\Re \lambda  \ge 0$ is required. This section provides the proofs of Theorems~\ref{th:main1} and \ref{th:main2}
 
\subsection{Case $\deg Q -\deg P=-1$}

\begin{lemma}\label{firstcase}
For any operator $T$ as above with $\deg Q -\deg P=-1$, there exist non-trivial $T_{CH}$-invariant sets. 
\end{lemma} 

\begin{proof}
Notice that in this case $R(z)=\frac{\lambda}{z}+o(\frac{1}{z})$ as  $z\to \infty$ where $\lambda \in \bC^\ast$. The curve of inflection $\mathfrak{I}_{R}$ defined by $Im(R')=0$ has four asymptotic branches whose asymptotic directions are $\frac{\phi_{\infty}}{2}+\frac{\pi}{2}\mathbb{Z}$ where $\phi_{\infty}=\arg(\lambda)$, see Appendix A.2.
We consider an open cone $\Gamma$ satisfying the following conditions:
\begin{itemize}
\item the apex $z_{0}$ of the cone belongs to $\mathfrak{I}_{R}$;
\item the opening of $\Gamma$ is $I=(\frac{\phi_{\infty}}{2}-\epsilon,\frac{\phi_{\infty}}{2}+\epsilon)$ for some small $\epsilon>0$ in such a way that $\sigma(z_{0}) \in I$;
\item no root of $PQ$ belongs to $\Gamma$ or to its closure;
\item the branch of the curve $\mathfrak{I}_{R}$ starting at $z_{0}$ and having the asymptotic direction equal to $\frac{\phi_{\infty}}{2}$ is contained in $\Gamma$. No other point of  $\mathfrak{I}_{R}$ belongs to $\Gamma$ or its closure.
\end{itemize}
We want to prove that for any point $x \in \Gamma$, we have $\sigma(x) \in \bar{I}$, which would imply that the complement to the interior of $\Gamma$ is a nontrivial $T_{CH}$-invariant set.\newline

Since $R$ is an open map and
\[\lim_{t\to +\infty} \arg R \left(z_0+te^{i\left(\frac{\arg \phi_\infty}{2}+\theta\right)}\right)=\frac{\arg \phi_\infty}{2}-\theta,\]xi
the extremal values of $\arg R$ on the (compact) closure of $\Gamma$ in $\mathbb{C} \cup \mathbb{S}^{1}$ are achieved on  the limit half-lines $z_0+\bR_{\ge 0}e^{i\left(\frac{\arg \phi_\infty}{2}\pm\epsilon \right)}.$ Besides, such half-lines meet $\mathfrak{I}_{R}$ only in $z_{0}$ and therefore $\arg R$ is a \emph{monotone} function on each of them. On each of these half-lines, the range of the function  $\arg R$ is an interval whose endpoints are $\arg R(z_{0})$ for the first one and one slope of the form $\frac{\phi_{\infty}}{2} \pm \epsilon$ for the second one. It follows then that for any point $x$ in the closure of $\Gamma$, $\arg R(x) \in \bar{I}$. 
\end{proof}

We can also deduce a more precise description of minimal invariant sets in this case.

\begin{corollary}\label{cor:minpq-1}
For any operator $T$ as above with $\deg Q -\deg P=-1$, the complement of $\minvset{CH}$ in $\mathbb{C}$ has exactly two connected components $X_{1},X_{2}$.\newline
For each $X_{i}$, $\sigma(X_{i})$ are respectively $(\frac{\phi_{\infty}-\pi}{2},\frac{\phi_{\infty}+\pi}{2})$ and $(\frac{\phi_{\infty}+\pi}{2},\frac{\phi_{\infty}+3\pi}{2})$.
\end{corollary}

\begin{proof}
We already know that $\minvset{CH}$ is different from $\mathbb{C}$ (Lemma~\ref{firstcase}) and noncompact (Corollary~\ref{cor:Raw1}). The construction of Lemma~\ref{firstcase} can be carried out for both directions $\frac{\phi_{\infty}}{2}$ and $\frac{\phi_{\infty}}{2}+\pi$. Besides, the opening of the cone can be made arbitrarily close to $\pi$. It follows that if the complement $X$ of $\minvset{CH}$ in $\mathbb{C}$ is connected, then $\sigma(X)$ is a connected open subset of $\mathbb{S}^{1}$ containing both $\frac{\phi_{\infty}}{2}$ and $\frac{\phi_{\infty}}{2}+\pi$. Since $\sigma(X)$ is invariant by the map $\theta \mapsto \phi_{\infty} - \theta$, it follows that $\sigma(X)=\mathbb{S}^{1}$. This contradicts Corollary~\ref{cor:TourCone}.\newline
Since the complement of $\minvset{CH}$ cannot be connected, it follows from Corollary~\ref{cor:RawArg} that it is formed by exactly two connected components $X_{1},X_{2}$, such that $\frac{\phi_{\infty}}{2}$ and $\frac{\phi_{\infty}}{2}+\pi$ are respectively the bisectors of $\sigma(X_{1})$ and $\sigma(X_{2})$. These two intervals are of length $\pi$.
\end{proof}

\subsection{Case $\deg Q -\deg P=0$}

\begin{lemma}\label{secondcase} For any operator $T$ with $\deg Q -\deg P=0$, there exist non-trivial  $T_{CH}$-invariant sets. 
\end{lemma} 
\begin{proof} 
Under our assumptions, $\vert R(z)-\lambda \vert=\vert \frac{Q(z)}{P(z)}-\lambda\vert$ can be made arbitrary small  when $|z|$ is sufficiently large. Consider a $1$-parameter family of the halfplane $H_\ell,\theta: \{\Re \left(\frac{ze^{i\theta}}{\lambda}\right) \leq \ell\}$ for $\theta \in (-\frac{\pi}{2},\frac{\pi}{2})$. We will show that for $\ell$ sufficiently large (depending on $\theta$), $H_\ell,\theta$ is $T_{CH}$-invariant.  Indeed, for any $\epsilon >0$, there exists $\ell_0$ such that for any $\ell\ge \ell_0$ and any point $z$ in the complement of $ H_\ell,\theta$, we have that $\arg z\in (\phi_\infty-\epsilon, \phi_\infty +\epsilon)$. Therefore the associated ray of such $z$ remains in the complement of $H_\ell,\theta$ which finishes the proof by Theorem~\ref{th:charact}. 
\end{proof}

A more precise description of minimal invariant sets follows.

\begin{corollary}\label{cor:minpq0}
For any operator $T$ as above with $\deg Q -\deg P=0$ and any $\epsilon>0$, there is an open cone $\Gamma$ whose interval of directions is $(\phi_{\infty}+\pi-\epsilon,\phi_{\infty}+\pi+\epsilon)$ such that $\minvset{CH} \subset \Gamma$.
\end{corollary}

\begin{proof}
The construction in Lemma~\ref{secondcase} can be carried out to provide a pair of half-planes whose boundary lines have directions arbitrarily close to $\phi_{\infty}+\pi$. Hence, we can find two invariant planes that are $T_{CH}$-invariant, for which the intersection, which is again $T_{CH}$-invariant, is contained in a cone $\Gamma$ whose interval of directions is $(\phi_{\infty}+\pi-\epsilon,\phi_{\infty}+\pi+\epsilon)$.
\end{proof}

\subsection{Case $\deg Q -\deg P=1$}

This situation requires more work and an additional restriction. The existence of non-trivial $T_{CH}$-invariant sets in this case follows from the existence of compact $T_{CH}$-invariant sets for any such linear differential operator. As above, we will use the following expression 
\begin{equation}\label{eq:expN}
R(z){\partial_z}= \frac{Q(z)}{P(z)}{\partial_z}=\left(\lambda z + O(1)\right){\partial_z}. 
\end{equation}

\medskip\noindent
Notice that introducing $z=1/w$ yields
\begin{equation*}
R(w){\partial_w}= -w^2\frac{Q(w)}{P(w)}{\partial_w}=\left(-\lambda w + O(w^2)\right){\partial_w},
\end{equation*}
as $w\to 0$ and so $R(z){\partial_z}$ has a zero of order 1 at $\infty$. 
If $\Re(\lambda)>0$, $\infty$ is a sink, if $\Re(\lambda)=0$ it is a center, and if $\Re(\lambda)<0$ it is a source of the vector field $R(z){\partial_z}$. 

The case where $\deg P=0$ and $\deg Q=1$ is almost trivial. It has been proved in Section~\ref{ssec:special} that if $\lambda \in \mathbb{R}_{<0}$ any $T_{CH}$-invariant set is trivial. Conversely, for any other value of $\lambda$, the unique root of $Q$ forms a (compact) $T_{CH}$-invariant set.

In the following, we will assume $\deg P \geq 1$.


\begin{proposition}\label{th:negativeResidueAtInfty}
If $\Re \lambda<0$, then $\minvset{CH}=\bC$.
\end{proposition}
\begin{proof}
Integral curves of $-R(z) {\partial_z}$ coincide with real negative trajectories of the translation structure induced by meromorphic $1$-form $\frac{dz}{R(z)}$ on $\mathbb{C} \cup \lbrace{ \infty \rbrace}$. Here, $\infty$ is a simple pole of the differential and its residue $\lambda$ has a negative real part. Assuming that $\deg P \geq 1$, there always exists a trajectory from $\infty$ to a zero of $\frac{dz}{R(z)}$ (see Section 5 of \cite{Ta} and in particular Proposition 5.5). Such a zero is automatically a root of $P$.\newline
We found an integral curve $\varphi(t)$ of $-R(z) {\partial_z}$ that is $P$-starting separatrix and hence is contained in any $T_{CH}$-invariant set $S$, see \cref{pro:separatrixIsInInv}. Corollary~\ref{cor:Raw1} then proves that unless $\lambda \in \mathbb{R}^{+}$, any such noncompact $T_{CH}$-invariant set coincides with $\mathbb{C}$. 
\end{proof}

\smallskip
We now concentrate on the case $\deg Q -\deg P=1$ and $\Re \lambda\ge 0$ and show that this condition guarantees the existence of compact $T_{CH}$-invariant sets. 

\begin{lemma}\label{third case} For any operator $T$ with $\deg Q -\deg P=1$ and $\Re(\lambda)>0$, there exist compact $T_{CH}$-invariant sets. 
\end{lemma} 
\begin{proof} 
For $|z|$  sufficiently large, the expression $\vert \arg R(z)-\arg(\lambda z) \vert$ can be made arbitrary small. Consider a $1$-parameter family of the disks $D_\ell: \{|z| \leq \ell\}$. We will show that for $\ell$ sufficiently large, $D_\ell$ is 
$T_{CH}$-invariant. Indeed, set $\arg \lambda \in (-\frac{\pi}{2},\frac{\pi}{2} )$. Then there exists $\ell_0$ such that for any $\ell\ge \ell_0$ and any point $z \notin D_\ell$, the associated ray of such $z$ remains in the half-plane bounded by  the line $z+iz\mathbb{R}$ that does not contain disk $D_{\ell_{0}}$. It follows that $D_{\ell_{0}}$ is a $T_{CH}$-invariant set.
\end{proof}

Finally, let us consider the limiting case. 

\begin{proposition}
If $\Re \lambda =0$,  then there exist compact $T_{CH}$-invariant sets.
\end{proposition}
\begin{proof}
In case $\Re \lambda=0, \; \lambda\neq 0$ we get that $\infty\in \bar \bC$ is a center of $R(z)\partial_z$, i.e. its trajectories  near $\infty$ are closed curves. Moreover sufficiently close to $\infty$ these trajectories (which look approximately like circles of very large radius) are convex: by \cref{lem:blueCurveCompact} they are  locally convex, hence convex. Let $\Gamma$ be one such closed convex trajectory. Since the vector $R(z)\partial_z$ is tangent to $\Gamma$ at any  point $z\in \Gamma$ and it points ``towards $\infty$", the associated ray at this point never crosses $\Gamma$ for $t>0$. Thus by Proposition~\ref{prop:raysGiveCHinv} the compact domain bounded by $\Gamma$ is $T_{CH}$-invariant. 
\end{proof}

\section{(Ir)regularity of $T_{CH}$-invariant sets: general properties} \label{sec:irreg}


In this and the two following sections we  discuss irregularity of $T_{CH}$-invariant sets, see Definition~\ref{def:regularSets}. Our interest in irregularity is related to the facts that  (i) it is easier to describe when a regular set is $T_{CH}$-invariant than an irregular one, (ii) $T$ under some not very strong assumptions admits irregular $T_{CH}$-invariant sets only if the vector field $R(z)\partial z$ has a line of symmetry which only happens if $R(z)$ becomes real after an appropriate affine change of the variable $z$. So morally the study of irregular sets is reduced to the study of operators $T$ with real-valued coefficients which is a natural subclass of all operators under consideration. We begin with the following definition. In this section we assume only that neither $P$ nor $Q$ is vanishing identically.

\begin{definition}
For a given rational function $R(z)$, a line is defined to be $R$-invariant if for any $z \in \Lambda$ such that $R(z)$ is defined, we have $z+R(z) \in \Lambda$.
\end{definition}

\begin{lemma}\label{lem:Dihedral}
For a given nonzero rational function $R(z)$, if two distinct lines are $R$-invariant, then one of the following statements holds:
\begin{enumerate}
    \item $R(z)=\lambda$ for some $\lambda \in \mathbb{C}^{\ast}$;
    \item $R(z)=\lambda(z-\alpha)$ for some $\lambda \in \mathbb{C}^{\ast}$ and $\alpha \in \mathbb{C}$;
    \item there are finitely $R$-invariant lines, they intersect in some point $\alpha$ and conjugations along (i.e. reflection through) them determine a finite dihedral group.
\end{enumerate}
\end{lemma}

\begin{proof}
A rational function $R(z)$ for which the real axis is a $R$-invariant line is a real rational function and is therefore invariant by conjugation. It follows that invariant lines of $R(z)$ generate a reflection group $G$ of maps commuting with $R$. In particular, any element in $G$ sends the finite set of $R(z)$ to itself.
\par
If $G$ is finite, then the classification of reflection groups of the plane proves that $G$ is a finite dihedral group. Invariant lines of $R(z)$ thus intersect in some point $\alpha$.
\par
If $G$ is infinite, then it admits an infinite subgroup formed by maps fixing each zero or pole of $R(z)$. It follows that there is at most one such singular point. If $R(z)$ has no finite poles or zeroes, then we have $R(z)=\lambda$ for some $\lambda \in \mathbb{C}^{\ast}$. Otherwise, $R(z)=\lambda(z-\alpha)^{k}$ for some $\lambda \in \mathbb{C}$, $\alpha \in \mathbb{C}$ and $k \in \mathbb{Z}$. Every $R$-invariant line contains $\alpha$ and $G$ thus contains an infinite group of rotations around $\alpha$. As every element in $G$ commutes with $R$, it follows that $R(z)=\lambda(z-\alpha)$. 
\end{proof}

In the cases $\deg Q - \deg P = \pm 1$, we obtain more precise statements.

\begin{lemma}\label{lem:OddInvariant}
We consider a rational function $R(z)=\frac{Q(z)}{P(z)}$ such that $\deg Q - \deg P =-1$. If there are two distinct $R$-invariant lines, then up to an affine change of variable, $R(z)$ is a real odd rational function on $\mathbb{R}$ (nonzero coefficients of its Laurent series have odd order). Besides, there are only two $R$-invariant lines and they coincide with the real and imaginary axis.
\end{lemma}

\begin{proof}
Following Lemma~\ref{lem:Dihedral}, conjugations along $R$-invariant lines generate a finite dihedral group. At infinity, we have $R(z)=\frac{\lambda}{z}+o(1/z)$ for some $\lambda \in \mathbb{C}^{\ast}$. From this it follows that any two $R$-invariant lines have to be perpedincular. Up to an affine change of variable, we will assume that these $R$-invariant lines coincide with $\mathbb{R}$ and $i\mathbb{R}$. Writing down the two conjugations that leave $R(z)$ invariant in terms of the coefficients of the Laurent series of $R(z)$ in $0$ provides the characterization of $R(z)$ as a real odd rational function.
\end{proof}

\begin{lemma}\label{lem:qp1LambdaInvariant}
Let $R(z)=\frac{Q(z)}{P(z)}$ be such that $\deg Q - \deg P =1$. If there is at least one $R$-invariant line, then we have $R(z)=\lambda z +o(z)$ with $\lambda \in \mathbb{R}^{\ast}$.
\end{lemma}

\begin{proof}
This is immediate from the definition of a $R$-invariant line.
\end{proof}

\smallskip
The next statement provides an important necessary condition for the existence of irregular $T_{CH}$-invariant sets.



\begin{lemma}\label{lem:IrrLine}
Assume that a linear differential operator $T = Q(z)\frac{d}{dz} + P(z)$ has an irregular $T_{CH}$-invariant set $S \subset \mathbb{C}$. Then any irregular point of $S$ is contained in some $R$-invariant line.
\end{lemma}

\begin{proof}
We consider a $t$-trace $\gamma(t)$ that solves 
\[
    tQ(z)+(z-z_0)P(z)=0,
\]
such that  $\gamma(0)=z_0$ where $z_{0}$ is an irregular point of a $T_{CH}$-invariant set $S$. We have that $\gamma'(t)$ is given by \eqref{eq:field}, which is non-singular for small $t>0$. 
Furthermore, the forward trajectories of $\gamma(t)$ are contained in $S$. 

By irregularity of $z_0$, $\gamma'(t)=\frac{-R(z)}{tR'(z)+1}=-k(t)R(z)$ for some $k(t)\in \mathbb{R}\setminus\{0\}$ and sufficiently small $t>0$. Since $\gamma(0)=\gamma(t)+tR(\gamma(t))$, it follows that $\gamma$ is locally a segment of some line $\Lambda$. Besides, for any $t$, the direction of $R(\gamma(t))\partial_z$ is parallel to one of the directions of $\Lambda$.

Consequently, up to an affine change of variables, we can assume that $\Lambda = \mathbb{R} \subset \mathbb{C}$. It follows that $R(z)$ is a real rational function. In other words, for any $z \in \Lambda$, $R(z)$ has the direction of $\Lambda$.
\end{proof}

Next, we have the following.
\begin{lemma}\label{le:zerosofPQinIrregular}
Assume that a linear differential operator $T = Q(z)\frac{d}{dz} + P(z)$ has an irregular $T_{CH}$-invariant set $S \subset \mathbb{C}$. Any irregular point $z_{0}$ of $S$ that is a zero or a pole of $R(z)$ is a simple zero or a simple pole.
\end{lemma}
\begin{proof}
Suppose first that $z_0$ is a pole of $R(z)$ order $k\geq 2$. Using \cref{prop:rootTrajectoryDirections} twice, first with $u=z_0$ and then with $u$ belonging to the root root trail of $z_0$ sufficiently close to $z_0$, we deduce that $z$ is not irregular. Hence, suppose $z_0$ is a zero of $R(z)$
with multiplity $k$, and consider
the root trail of any point $u\neq z_0$, which exists since $P(z)$ is assumed to not vanish identically. There are $k$ solutions of $tQ(z)+(z-z_0)P(z)=0$ that tend to
$z_0$ as $t\to \infty$. Moreover, $tQ(z)+(z-z_0)P(z)=0$ has only simple zeros for all but a finite set
of $t$, since $\frac{(z_0-z)P(z)}{Q(z)}$ have only a finite number of critical values (see \cref{le:genericu}). We have that $R(z)\approx a (z-z_0)^k$ close to $z_0$ and after a change of variables we may assume that $a=1, z_0=0$. Further, $tR(z)+z=u$ for points $z$ in the root trail of $u$. By the assumption that $z_0$ is an irregular point $U\cap S$ contains only irregular points for sufficiently small neighborhoods of $z_0$. It follows that $k\leq 2$, $u\in \mathbb R$, $\mathfrak{tr}_u\cap U\subset \mathbb R$ for sufficiently small open neighborhoods $U$ of $z_0$ and that one of the zeros in $\mathfrak{tr}_u$ tends to $0$ from the positive side and another from the negative side. However, there are then points close to 0 with negative real part and non-zero imaginary part whose associated rays intersect $\mathfrak{tr}_u$, and these points does not belong to any $R$-invariant line. We thereby deduce from the assumption that $z_0$ is irregular that k=1.
\end{proof}
In the proposition and its corollary, we give a more precise description of the irregular locus of $T_{CH}$-invariant sets. The reader may recall the definition of the \emph{irregular locus} from \cref{def:regularSets}.

\begin{proposition}\label{prop:IrrTail}
For any linear differential operator $T=Q(z)\frac{d}{dz}+P(z)$,  suppose that $z_{0} \notin \mathcal{Z}(PQ)$ is an irregular point of some $T_{CH}$-invariant set $S$, then one of the following statements holds:
\begin{enumerate}
    \item\label{item:1} $R(z)=\lambda$ for $\lambda \in \mathbb{C}^{\ast}$;
    \item\label{item:2} $R(z)=\lambda(z-\alpha)$ for $\lambda \in \mathbb{R}_{>0}$ and $\alpha \in \mathbb{C}$;
    \item\label{item:3} $S$ is contained in the line $z_{0}+R(z_{0})\mathbb{R}$;
    \item\label{item:4} half-line $L=z_{0}+R(z_{0})\mathbb{R}_{>0}$ is contained in an open cone $\Gamma$ with vertex at   $z_{0}$ such that $\Gamma \cap S \subset L$. Besides, for any $z \in L$, $R(z) \in R(z_{0})\mathbb{R}_{>0}$.
\end{enumerate}
\end{proposition}

\begin{proof}
We assume we are not in the cases of items (1), (2), or (3) and let us conclude that  item (4) holds. \cref{lem:Dihedral}   together with \cref{lem:IrrLine} proves the existence of an $R$-invariant line $\Lambda$ and a cone $\Gamma$ such that $\Lambda$ that contains the intersection $S\cap \Gamma$.

\end{proof}
We call a set of the form $\Gamma\cap S$ as in  \cref{item:4} above a \defin{tail}.

\begin{corollary}
If $\minvset{CH}$ is not fully irregular, then the irregular locus is formed by finitely many straight segments $(\alpha_{1},\beta_{1}],\dots,(\alpha_{k},\beta_{k}]$ where for any $i$:
\begin{itemize}
    \item the segment $(\alpha_{i},\beta_{i}]$ belongs to an $R$-invariant line;
    \item for any $z \in (\alpha_{i},\beta_{i}]$, $\frac{\beta_{i}-\alpha_{i}}{R(z)} \in \mathbb{R}_{>0}$;
    \item $\alpha_{i}$ belongs to the regular locus of $\minvset{CH}$;
    \item $\beta_{i} \in \mathcal{Z}(P)\cap \mathcal{Z}(Q)$;
    \item $\beta_{i}$ is a root of the same multiplicity for $P$ and $Q$.
\end{itemize}
In particular, if no point of $\mathcal{Z}(PQ)$ satisfies the latter condition, then $\minvset{CH}$ is either regular or fully irregular.
\end{corollary}

\begin{proof}
Assuming $\minvset{CH}$ is not fully irregular, we consider a point $z_{0}\minvset{CH}$ that does not belong to $\mathcal{Z}(PQ)$ nor the regular locus of $\minvset{CH}$. Following \cref{thm:connectedclosure}, $z_{0}$ is not an isolated point. We apply \cref{prop:IrrTail} to $z_{0}$.
\par
Since $\minvset{CH}$ is not fully irregular, \cref{item:3} in \cref{prop:IrrTail} does not apply. If $R(z)=\lambda$ for some $\lambda \in \mathbb{C}^{\ast}$, then the minimal set clearly coincides with a finite union of parallel half-lines starting at points of $\mathcal{Z}(PQ)$. In this case, $\minvset{CH}$ is also fully irregular so \cref{item:1} does not apply. If $R(z) = \lambda (z-z_{1})$ for some $\lambda \in \mathbb{C}^{\ast}$ and $z_{1} \in \mathbb{C}$, then $\minvset{CH}$ is regular unless $\lambda \in \mathbb{R}_{>0}$ (see \cref{th:negativeResidueAtInfty} and Lemma~\ref{lem:qp1LambdaInvariant}), provided there is $z_0\notin \zeros(PQ)$ belonging to $\minvset{CH}$. In the case when $\lambda \in \mathbb{R}_{>0}$, $\minvset{CH}$ is clearly a finite union of segments between $z_{0}$ and points of $\mathcal{Z}(PQ)$. Therefore $\minvset{CH}$ is fully irregular and \cref{item:2} does not apply either.
\par
Therefore, \cref{item:4} applies and $z_{0}$ belongs to a tail. Such a tail can be removed harmlessly from any $T_{CH}$-invariant set and still render it $T_{CH}$-invariant set unless its final vertex (for the orientation defined by $R(z)\partial_{z}$) belongs to $\mathcal{Z}(PQ)$. If the initial point $\alpha_{i}$ of the tail does not belong to the regular locus of $\minvset{CH}$, then it has to be a simple zero or a simple pole of $R(z)$ (see \cref{le:zerosofPQinIrregular}) and applying again \cref{prop:IrrTail} we deduce easily that in this situation, $\minvset{CH}$ is contained in a line.
\end{proof}

\begin{remark} 
$T_{CH}$-invariant sets can be partially irregular. 
For instance, this can happen in the case of the cocheloid, see  Example~\ref{ex:coch} below, where $Q(z)=z^2$, $P(z)=z-1$.
Indeed, we can add to the minimal invariant set $\minvset{CH}$ any 
interval $I=[-\alpha,0)$ for $\alpha>0,$  and obtain $S=\minvset{CH}\cup I$  which is also $T_{CH}$-invariant. This set is then also the minimal $T_{CH}$ invariant set for the operator $\tilde T$ with $\tilde Q(z)=Q(z)(z+\alpha), \tilde P(z)=P(z)(z+\alpha)$.
\end{remark}

\begin{example}\label{ex:coch}
Take $Q(z)=z^2$, $P(z)=z-1$. Then in  polar coordinates  the closure of the 
union of the $1$-starting separatrices is given by $\frac{\sin \theta}{\theta}$. 
Let us denote  by $S$ this union together with the region it encloses.
By the above,  the boundary of $S$ is contained in $\minvset{CH}$. 
For any point $z$ in the interior of $S$, $R(z)\neq 0$, 
so there exists $t$ such that $z+tR(z)\in \partial S$ and it
follows that $z\in \minvset{CH}$. Now, using the same argument 
as in \cref{prop:raysGiveCHinv} one can  show that no $t$-trace 
can reach the complement of $S$ through a point which is not a root of $P$.
In order to deduce that $S$ is indeed equal to $\minvset{CH}$, 
we need to show that the $t$-traces defined by the solutions to
\begin{equation}\label{eq:pgraph}
tz^2+(z-1)^2=0
\end{equation}
are in $S$ for any small (and hence for all) $t$. 
We do so by comparing the second order terms in the Taylor 
expansion of the solutions to \eqref{eq:pgraph} as $t\to 0$ as well 
as the $1$-starting separatrices as $\theta\to 0^+$.

Now, following \cref{prop:rootTrajectoryDirections} and introducing 
the change of variables $t\to t^2$, we obtain that the 
solutions to \eqref{eq:pgraph} are given by 
\[
1+i t-t^2+O\left(t^3\right), \quad 1-i t-t^2+O\left(t^3\right)
\]
and the separatrices close to $\theta=0$ are given by
\[
1+i t-\frac{2 t^2}{3}+O\left(t^3\right), \quad 1-i t-\frac{2 t^2}{3}+O\left(t^3\right).
\]
We conclude that $S=\minvset{CH}$.

\begin{figure}[!ht]
\centering
\includegraphics[width=0.27\textwidth]{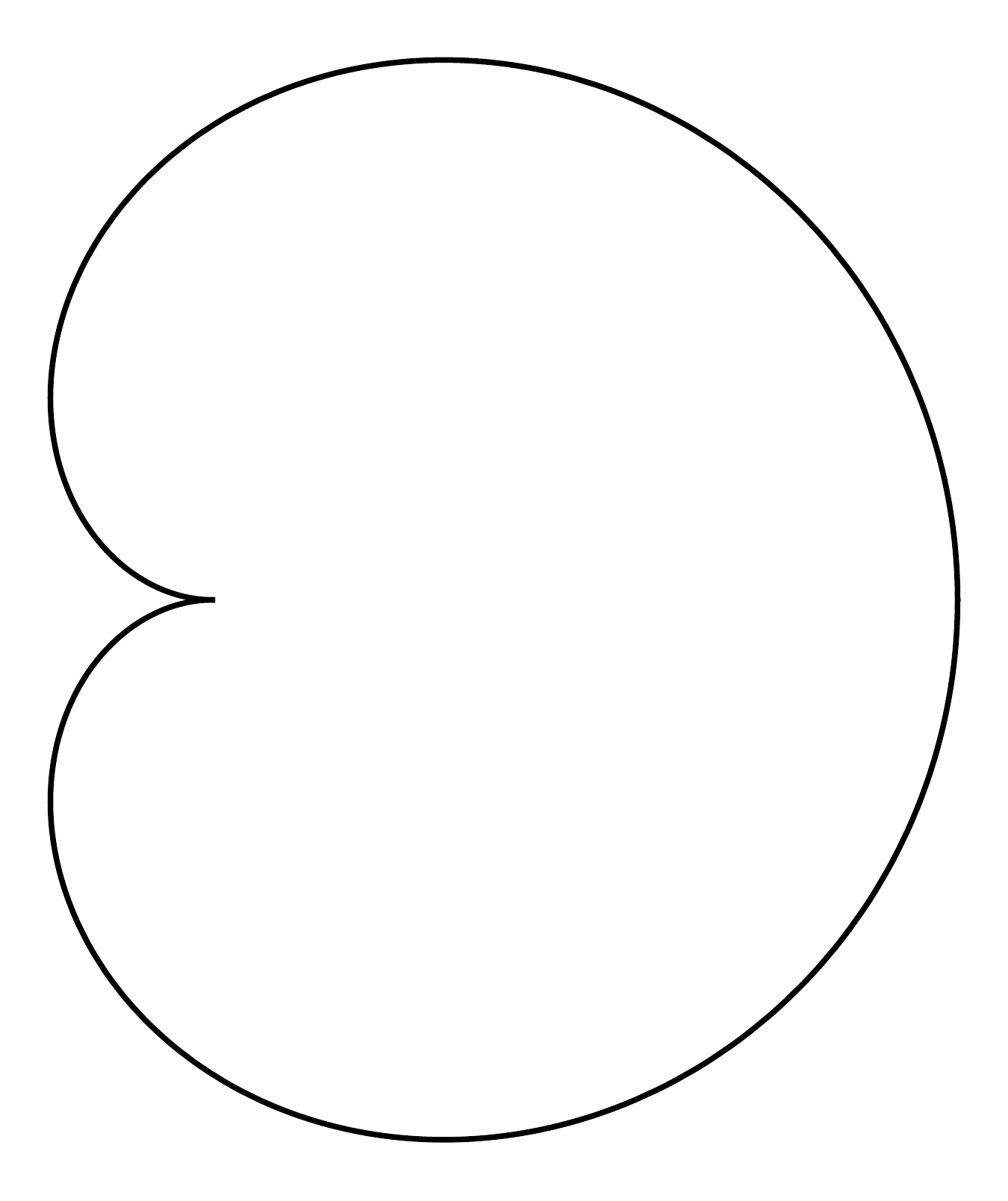} 
\caption{(The boundary of) $\minvset{CH}$ for $T=z^2\frac{d}{dz}+(z-1)$.}
\label{fig:cardiodid}
\end{figure}

\end{example}

\section{Existence of fully irregular $T_{CH}$-invariant sets}\label{sub:realirregular}

Below we will show that there exists a $T_{CH}$-invariant set contained in $\bR$ if and only if (up to a scalar factor) both $Q$ and $P$ are real and have only real and interlacing roots, see \cref{thm:classifying1d}. 
This result will explain which operators $T$ admit fully irregular $T_{CH}$-invariant sets. 

Under the assumptions of \cref{ssec:trivial} as well as $R(z)\neq \lambda(z-\alpha)$, we know from the previous section we know that any fully irregular set is contained in a finite union of lines such that $R$ is invariant under conjugation along these lines. If $\deg P=1$, $\deg Q=0$, then using \cref{prop:rootTrajectoryDirections} and \cref{th:addit}, we obtain that if there is a fully irregular set, it is contained in a line. In the same vein, but only using \cref{prop:rootTrajectoryDirections}, for any other case when $\deg P\geq 1$, we conclude that if there is a fully irregular set, it is contained in a line.
Supposing now that there is a fully irregular set, we have that $|\deg Q-\deg P|\leq 1$. 
As the roots of $Q$ and $P$ are necessarily contained in any $T_{CH}$-invariant set $S$, 
it follows that the roots of $P$ and $Q$ must 
lie on a line. Applying an affine change of variable and multiplication by a scalar, we may assume that the polynomials $P$ and $Q$  are real. 
From \cref{le:zerosofPQinIrregular} we know that $P$ and  $Q$ can only have simple roots.

%

We say that two roots $z_0,z_1 \in \zeros(PQ)$ are \defin{adjacent} 
(along the $x$-axis) if there is no root $z_2\in \zeros(PQ)$ 
such that $z_0<z_2<z_1$ or $z_1<z_2<z_0$.

It is now clear that if $T$ admits a fully irregular set $S$ then no two roots of $P$ can be adjacent. Indeed, 
if there were adjacent roots $z_1$, $z_2$ of $P$, since $S\subset \bR$ and the $P$-starting separatrices are contained in $S$, it follows that some of the $P$-starting separatrices of the vector field $-R(z){\partial_z}$ intersect, which is impossible.

The next lemma shows that the roots of $P$ and $Q$ must \defin{interlace}.
In other words, all roots of $P$ and $Q$ are simple, and roots of $P$ can only be adjacent to roots of $Q$ and vice versa.
\begin{lemma}\label{le:onlyInterlacing}
If $S\subset\mathbb{R}$ is  $T_{CH}$-invariant and $P$ is non-zero 
then $Q$ has no adjacent roots.
\end{lemma}
\begin{proof}

If there are two simple roots of $Q$, it follows that one of them is a source of $-R(z) {\partial_z}$,
and this violates \cref{prop:qRootSinks}.

\end{proof}

\begin{lemma}\label{le:singularities}
Let $\gamma_u(t)$ be a $t$-trace, i.e. solution of \eqref{eq:main}. 
Then $V(\gamma(t),t)$ has a singularity at $t_0$ only 
if there is another $t$-trace $\eta_u(t)$ such that $\gamma_u(t_0)=\eta_u(t_0)$.
\end{lemma}
\begin{proof}
This follows from the fact that simple roots of a 
polynomial are analytic in the coefficients.
\end{proof}

\begin{theorem}\label{thm:classifying1d}
Let $P$ and $Q$ have only real roots.
Then the following facts hold:  

{\rm (i)} If $\deg Q>0$ and $\deg P = 0$, then there exist fully irregular 
$T_{CH}$-invariant sets if and only if $R=\lambda(z-z_0)$ where $\lambda\notin \mathbb R_{<0}$. 

{\rm (ii)} If $\deg P>0$  and $Q$ is non-zero, then there exist fully irregular $T_{CH}$-invariant 
sets if and only if 

\quad\quad {\rm $\bullet$} the roots of $P$ and $Q$ interlace;

\quad\quad $\bullet$ there are no zeros of multiplicity greater than $1$;

\quad\quad $\bullet$ and for any root $\alpha$ of $P$, we have $r_{\alpha}<0$ where $R(z)=\frac{r_{\alpha}}{z-\alpha}+o(|z-\alpha|^{-1})$.

In case \rm{(ii)}, these invariant sets are either intervals, half-lines or lines.
\end{theorem}

\begin{proof}

We begin with item (i). If  $\lambda\notin \mathbb{R}_{<0}$, then $\{z_0\}$ is an invariant set and if $\lambda\notin \mathbb{R}_{<0}$, the only invariant set is $\bC$, see \cref{ssec:special}

Let us turn to item (ii).  
The assumption $r_{\alpha}<0$ implies that 
every $P$-starting separatrix starts with the argument equal either 0 or $\pi$, 
which is necessary for the full irregularity. 
The ``only if'' - statement follows from \cref{le:onlyInterlacing} above. 
We need to prove the ``if'' - statement. There are three cases to consider: 

\begin{enumerate}
\item \label{it:qpq}
the smallest and the largest roots of $PQ$ belong to $\zeros(Q)$;
\item \label{it:pq}
one of the smallest and the largest roots belong to $\zeros(Q)$ and the other to $\zeros(P)$;
\item \label{it:pqp}
the smallest and the largest roots are roots of $\zeros(P)$.
\end{enumerate}

We claim that in  case  (\ref{it:qpq}) the minimal invariant set  $\minvset{CH}$ 
is the interval between the smallest and the largest roots of $Q$. 
Let us denote the latter interval by $I$. Indeed, the closures of the 
$P$-starting separatrices belong to $I$. 
Furthermore, the condition $\deg Q =\deg P+1$ yields that all $t$-traces 
which do not start at $u$, 
should start in the roots of $P$. 
By \cref{prop:rootTrajectoryDirections}, 
if $u\notin \zeros(P)$ we find  that 
these $t$-traces  must have  the initial direction pointing away from $u$. 
If $u\in \zeros(P)$ we similarly find that the $t$-traces that do not start 
in $u$ have  the initial direction pointing away from $u$ and the two $t$-traces that 
start at $u$ have the directions of the $u$-starting separatrices. 
In particular, the $t$-traces belong to the 
intervals of the form 
\[
[z_*,z^*], \;  [z^*,z_*], \quad z_*\in \zeros(P), \; z^*\in \zeros(Q),
\]
having disjoint interiors.
By \cref{le:singularities}, for $t>0$, 
the vector field $V(\gamma(t),t)$ in \eqref{eq:field}
has no singularities for a $t$-trace $\gamma(t)$.
Therefore by analyticity, they remain inside  
these intervals for all $t\geq 0$.
Hence, $I$ is $T_{CH}$-invariant. 
Moreover, the closure of the $P$-starting separatrices of $-R(z) {\partial_z}$
is the convex hull of the roots of $Q$. Thus 
the interval $[a,b]$  indeed coincides with the minimal invariant set $\minvset{CH}$.
 Moreover any $T_{CH}$-invariant set containing $a_1\leq a$ contains $[a_1,b]$ and any set 
containing $b_1\geq b$ contains $[a,b_1]$ and these sets are $T_{CH}$-invariant.
Furthermore, any interval of the form $[a_1,\infty), (-\infty,b_1], (\infty,\infty)$ 
is $T_{CH}$-invariant as well.

In case (\ref{it:pq}), we suppose without loss of generality that the smallest
root is that of $P$ and the largest one is that of $Q$. Let $z^*_0$ be the largest 
root of $Q$. Then $\minvset{CH}$ coincides with $(-\infty, z^*_0]$. 
Since $\deg(Q)<\deg(P)+1$, the $t$-traces are bounded for finite $t$. 
Hence, by the same argument as above, all $t$-traces belong to the 
intervals whose endpoints are in $\zeros(PQ) \cup \{ -\infty\}$ and 
the closure of the union of the said 
intervals is the minimal invariant set $\minvset{CH}$.

For any $b>z^*_0,$ we  again find that $(-\infty, b]$ 
is $T_{CH}$-invariant and so is $(-\infty, \infty)$.

Lastly, in case (\ref{it:pqp}), the minimal invariant 
set is the real line with is settled precisely the same way as in case (2). 
Finally, by \cref{lem:IrrLine},  there are no other fully irregular 
$T_{CH}$-invariant sets in (\ref{it:qpq})--(\ref{it:pqp}).
\end{proof}
 
 \begin{corollary}\label{cor:irregularInvSets}
Suppose $P$ and $Q$ are as in the above item (ii) of \cref{thm:classifying1d}. 
  Consider the cases:  
 
{\rm (i)} both the smallest and the largest roots in $\zeros(PQ)$ are roots of $Q$;

{\rm (ii)} one of the smallest and largest roots is a root of $Q$ and the other of $P$;

{\rm (iii)} both the smallest and largest roots are roots of $P$.

The fully irregular $T_{CH}$-invariant sets are as follows. 
 
In case {\rm (i)},
any (potentially infinite) line segment containing the roots of $Q$ is $T_{CH}$-invariant.

In case {\rm (ii)}, 
if the largest root is that of $P$, then any 
line segment of the form $[a,\infty)$ containing the roots of $Q$ is 
$T_{CH}$-invariant. Otherwise, any line segment of the form $(-\infty,b]$ 
containing the roots of $Q$ is $T_{CH}$-invariant.  In both cases, $(-\infty,\infty)$ is $T_{CH}$-invariant.

In case {\rm (iii)}, 
$\mathbb{R}$ is the only fully irregular $T_{CH}$-invariant set.
\end{corollary}

The cases above handled the situation when $P(z)$ and $Q(z)$ have no common factors. We mention here what occurs if one allows this. If $R(z)=\lambda (z-\alpha)$, $\lambda \in \mathbb C^*$, and $P(z)$ and $Q(z)$ have a common root that does not equal $\alpha$, then $T$ has fully irregular invariant sets if and only if $\lambda\in \mathbb R_{>0}$. If $R(z)$ is not on the form $\lambda (z-\alpha)$ and $P(z)$ and $Q(z)$ have common factors $(z-\alpha_1)^{k_1}(z-\alpha_2)^{k_2}\cdots (z-\alpha_l)^{k_l}$. Writing $G(z)=(z-\alpha_1)^{k_1}(z-\alpha_2)^{k_2}\cdots (z-\alpha_l)^{k_l}$. Then $T$ has fully irregular sets if and only if has interlacing zeros and $Q(z)/G(z)$ and $P(z)/G(z)$ fulfill the criterion of \cref{thm:classifying1d} (taking the roles of $Q$ and $P$) and all $\alpha_j$ belong to the real line (after the change of variables so that the non-common zeros of $P$ and $Q$ lie on the real). It is easy to prove this using the methods above and we omit the details.

\section{Existence of partially irregular $T_{CH}$-invariant sets}\label{sec:partirreg}

Below we present necessary and sufficient conditions describing  for which operator $T$, there exist partially irregular $T_{CH}$-invariant sets. Observe that by definition, the boundary of a partially irregular $T_{CH}$-invariant set necessarily contains both regular and irregular points. 

Since nontrivial  $T_{CH}$-invariant sets only exist if $|\deg Q-\deg P|\le 1$ and in case $\deg Q-\deg P=1$ one has to assume additionally that $\Re \lambda \ge 0$, we only need to consider the special cases below. 

\subsubsection{Case $\deg Q-\deg P=1$} 
We have the following sufficient condition for the existence of partially irregular invariant sets  in the case $\deg Q-\deg P=1$.
\begin{proposition}\label{prop:irregularForKEqualToOne}
Take an operator $T=Q\diffz+P$  such that $\deg Q-\deg P=1$ and  such that $R(z)=\frac{Q(z)}{P(z)}=\sum_{k=-1}^\infty a_{-k} z^{-k}$ is a real rational function. Then if $\lambda=a_{1}>0$,  there exist irregular sets.  
\end{proposition}
\begin{proof}
We have that $R(z)=\sum_{k=-1}^\infty a_{-k} z^{-k}$ 
where $a_k\in \mathbb{R}$ so that $R'(z)=a_{1}+\sum_{k=1}^\infty-k a_{-k} z^{-k-1}$. 
Pick $M>0$ large enough so that
\begin{enumerate}
\item $\{z:|z|\leq M\}$ is $T_{CH}$–invariant; 
\item if $|z|>M$ then $|\sum_{k=1}^\infty-k a_ {-k} z^{-k-1}|<a_1/2$.
\end{enumerate}  
If $|z|>M$, then $\Re R'(z) >|\Im R'(z)|$ and the associated ray at $z$, $\{z+tR(z):t\geq 0\}$, does not intersect $\{z:|z|\leq M\}$.
Take $z_0=a+bi$ such that $|z_0|>M$. 
If $a\geq M$ or $a\leq -M$, comparing $R(z_0)$ with $R(a)\in \mathbb{R}$ 
and using the fact that $\Re R'(z_0)>|\Im R'(z_0)|$, 
it follows that $\text{sign}(\Im R(z_0))=\text{sign}(b)$ and 
hence that the associated ray at $z$ does not intersect the real line. 
Next, if $a\in (-M,M)$, comparing with $z_1=M$ or $z_1=-M$, 
recalling that $R(z_1)\in \mathbb{R}$ and noting 
that $|\Re(z_0-z_1)|>|\Im(z_0-z_1)|$ yields the equality  $\text{sign}(\Im R(z_0))=\text{sign}(b)$. 
Once again we obtain that the associated ray at $z$ does not 
intersect the real line. In particular, appending to $\{z:|z|\leq M\}$ 
any real interval containing $z=M$ or $z=-M$ 
yields a partially irregular set.
\end{proof}

Taking into account \cref{th:negativeResidueAtInfty}, \cref{lem:qp1LambdaInvariant} together with \cref{prop:irregularForKEqualToOne} imply the following claim.

\begin{corollary}
Given $T$ with $\deg Q - \deg P=1$, there exist irregular $T_{CH}$-invariant 
sets if and only if there is an affine change of 
variables after which $R(z){\partial_z}$ 
is such that $\lambda>0$ and if $z\in \bR$, then one has $R(z)\in \bR$ (recall the definition of $\lambda$
in \eqref{eq:expN}).
\end{corollary}

\subsubsection{Case $\deg Q-\deg P=0$} 

\begin{proposition}\label{prop:irregularK0}
Given $T=Q\frac{d}{dz}+P$, suppose that $\deg Q- \deg P =0$ and $Q,P$ are real. 
Then there exist irregular $T_{CH}$-invariant sets. 
\end{proposition}
\begin{proof}
First, after a possible affine change of variables assume that \[R(z){\partial_z}=\left(\lambda +o(1)\right) {\partial_z}\] at $\infty$ with $\lambda>0$. 
Take $M>0$ large enough so that  if $\Re z>M$, then $\Re R(z) >|\Im R(z) |$ and $R'(z)$ satisfies the conditions $|R'(z)|<\lambda/2$ and $|\Re R'(z)|>|\Im R'(z)|$ in the open disk with center $M$ and radius  $\delta<1$. Take a point $z_0=x+iy\in\{ z:|z-M|\leq \delta,\; y>0,\; x >M\}$. We find that $|\Im R(z_0)|< y \lambda/2$ and $\Re R(z_0)>\lambda/2$.
Suppose that the associated ray $z_0+tR(z_0)$ intersects the real line for some $t>0$. Then $y-t\frac{y \lambda}{2} <0$ which gives $t>\frac{2}{\lambda}$.
Further  
\[
\Re(z_0+t R(z_0))>M+\frac{2}{\lambda} \frac{\lambda}{2}=M+1>M+\delta.
\]
Similarly,  we obtain that if $y<0$, then the intersection occurs only when $z_0+tR(z_0)>M+\delta$. 
Since $\Re R(z) >|\Im R(z)|$ no associated ray of a point in the complement outside of 
the disk of radius $\delta$ can intersect $(M,M+\delta]$. 
We conclude that the set $\{z:\Re z\leq M\}\cup (M,M+\delta]$ is $T_{CH}$-invariant, 
since the associated rays of points in the complement to the latter set lie in the complement.
\end{proof}

Together with \cref{lem:IrrLine}, \cref{prop:irregularK0} implies 
the following necessary and sufficient conditions for irregularity for $\deg Q -\deg P=0$.

\begin{corollary}
Given an operator $T$ with $\deg Q -\deg P=0$, then there exist irregular $T_{CH}$-invariant 
sets if and only if $R(z)$ is real after a suitable affine change of variables.
\end{corollary}

\subsubsection{Case $\deg Q-\deg P=-1$} 

If a $T_{CH}$-invariant set is partially irregular, then case (4) of Proposition~\ref{prop:IrrTail} applies. Up to a change of variables and multiplication of $P$ and $Q$ by the same constant, we can assume that $P,Q$ are real polynomials and the condition on the argument of $R(z)$ along the tail implies that $R(z)=\frac{\lambda}{z}+o(1/z)$ where $\lambda >0$. It remains to prove that  for any linear operator $T$ satisfying these conditions we are able de construct a partially irregular $T_{CH}$-invariant set.

\begin{proposition}
Suppose that $P,Q$ are such that after an affine change of variables,
$R(z){\partial_z}=\left(\frac{\lambda}{z}+\text{ higher order terms in }\frac{1}{z}\right){\partial_z}$
with $\lambda>0$ and $R(z)\in \bR$ if $z\in \bR$. Then there exist irregular $T_{CH}$-invariant sets.
\end{proposition}

\begin{proof}
Find $M$ large enough so that $\{z:|\Re z|\leq M\}$ is $T_{CH}$-invariant and such that if $ z\geq M$, $R(z)=\frac{\lambda}{z}+f(z)$ where $f(z)$ is such that $|\Im(f(z))|\leq|\Im\left(\frac{\lambda}{2z}\right)|$ and $|\Re(f(z))|\leq|\Re\left(\frac{\lambda}{2z}\right)|$. (Note that this is possible since the rational function $R(z)=\sum_{l=1}^\infty a_l/z^l$ is real and the coefficients $a_l$ are uniformly bounded by  $C^l$ for some $C>0$). Take $z=x+iy$  with $y>0$, $x>M$. Then 
\[
\Im(R(z))\geq\Im\left(\frac{3\lambda}{2z}\right)=\frac{-3\lambda y}{2(x^2+y^2)}
\] and 
\[
\Re(R(z))\geq \frac{\lambda x}{2(x^2+y^2)}.
\]  Hence, if $\Im(z+tR(z))=0$ then $t\geq\frac{2(x^2+y^2)}{3 \lambda}$ so that $\Re(z+tR(z))\geq x+\frac{x}{3}>\frac{4}{3}M$. By symmetry, we obtain the estimate for $y<0$. In particular, $\{z:|\Re(z)|\leq M\}\cup\{(M,\frac{4}{3}M]\}$ is $T_{CH}$-invariant.
\end{proof}

We may note that even if $P$ and $Q$ have factors in common, the results of this section hold.
Summarizing, the results of this section give necessary and sufficient conditions for the existence of irregular $T_{CH}$-invariant sets in terms of operator $T$ stated below. 

\begin{theorem}
There exist irregular $T_{CH}$-invariant sets if and only if there is an affine change of variables such that we get one of the following four cases:

{\rm (1)}  $P$ and $Q$ are as in \cref{thm:classifying1d};

{\rm (2)} $\deg Q -\deg P=1$ and $R(z){\partial_z}$ is such that if $z\in \bR$ then $R(z)\in \bR$ and $\lambda>0$;

{\rm (3)} $\deg Q - \deg P=0$ and $R(z){\partial_z}$ is such that if $z\in \bR$ then $R(z)\in \bR$;

{\rm (4)} $\deg Q -\deg P=-1$ and $R(z){\partial_z}=\left(\frac{\lambda}{z}+o\left(\frac{1}{z}\right)\right){\partial_z}$ with $\lambda >0$ is such that if $z\in \bR$ then $R(z)\in \bR$;

\end{theorem}

\section{Outlook}  \label{sec:outlook} 
In this short section we formulate some of the very many open questions related to the above topic. 

\smallskip
\noindent
{\bf I.} In the next statement we reconnect our area of study to the classical complex dynamics. 

\smallskip
Let $\julia(f) \subseteq \bC$ denote the Julia set of a rational function $f$.
 In the above notation, for an operator $T=Q(z)\frac{d}{dz}+P(z)$,  introduce the family of rational functions $z\mapsto  f_t(z)$ given by 
\[
f_t(z) \coloneqq z + t R(z), 
\]
where $t$  ranges over the non-negative reals. 

\begin{proposition}\label{prop:juliaSetIsSubset}
If $\max \left(\deg Q, \deg P\right) \ge 2$, the for any $t > 0$, $\julia(f_t)$ is contained in the minimal $T_{CH}$-invariant set $\minvset{CH}$.
\end{proposition}
\begin{proof}
First, if $\deg f_t \leq 1$ then $\julia(f_t)\subset\zeros(Q)\subset \minvset{CH}$.
Next, suppose that for a given $t\geq 0$,  $\deg f_t\geq 2$ and  $u \in \minvset{CH}$. 
Then all roots of \eqref{eq:main}  also lie in $\minvset{CH}$.
Indeed,  given $u$, we search for $z$ solving
\begin{equation}
 \label{eq:juliaSolutions} 
 u=\frac{ tQ(z)+zP(z)}{P(z)} = f_t(z). 
 \end{equation}
Hence, $f^{-1}_t(u) \subseteq  \minvset{CH}$. 
Iteration of  this argument yields that
$\cup_{j=0}^\infty f_t^{-j}(u)\subseteq \minvset{CH}$.
Recall that since $f_t$ is of degree at least 2, 
by \cite[Thm. 4.2.7]{Beardon2000}, $\julia(f_t) \subseteq \overline {\cup_{j=0}^\infty f_t^{-j}(u)}$
except possibly for two  values $u\in \setC$. 
Using the fact that under our assumptions the set $\minvset{CH}$ contains a curve, we obtain that
 $\minvset{CH}$ has at least three distinct points which in turn yields that
$\julia(f_t) \subseteq \overline {\minvset{CH}} =\minvset{CH}$.
\end{proof}

\smallskip
\cref{prop:juliaSetIsSubset} proves that 
\[
\bigcup_{t\geq 0 } \julia(f_t(z))
\]
is a subset of $\minvset{CH}$. In general, we suspect that, 
for example, in case of $\deg Q-\deg P=1$ and $\Re \lambda =0$ the union of the 
above Julia sets can be strictly smaller than the minimal set $\minvset{CH}$. 
Motivated by \cite{MR1043403} we formulate the following guess. 
 
\begin{conjecture} \label{conj:boundsep} 
If the boundary of $\minvset{CH}$ consists only of $P$-starting separatrices then $\minvset{CH}$ coincides with the above union of Julia sets. 
\end{conjecture} 

\smallskip
Here is a list of further natural questions.  

\smallskip\noindent
{\bf II.}
Is it possible to extend the main results of this paper to linear differential operators $T$ of order exceeding $1$?

\smallskip\noindent
{\bf III.}
Is it possible to give a description of the boundary of a non-trivial $\minvset{CH}$  starting with the roots of $Q$ and $P$  and making finitely many steps of taking $t$-trajectories and trajectories of the vector field $R(z) {\partial_z}$?

Our best guess is that, in general, this is impossible, i.e. one needs countably many such steps. 

\smallskip\noindent
{\bf IV.}
Extend the supply of cases with ``explicit" description of $1$-point generated $T_{CH}$-invariant sets.  Two special situations which seem to be fundamental, but we do not have an answer. For example,  $T = (z^3-1)\diffz + 3z^2$, see \cref{fig:triFlower}. This figure however is not accurate. It only shows some approximation of the actual $\minvset{CH}$. 

\smallskip\noindent
{\bf V.}
Show that the boundary of $\minvset{CH}$ is always piecewise analytic. 

\smallskip\noindent
{\bf VI.}
In this paper we were mainly studying continuously Hutchinson invariant sets $\minvset{CH}$ for linear differential operators of order $1$ while our original interest was to study  minimal Hutchinson invariant sets $\minvset{H}$, see \S~\ref{sec:intro}. Numerical experiments show that their boundaries typically have a fractal structure, see e.g. \cref{fig:cochleoid}, but we do not have conclusive results yet.

\begin{figure}[!ht]
\centering
\includegraphics[width=0.35\textwidth]{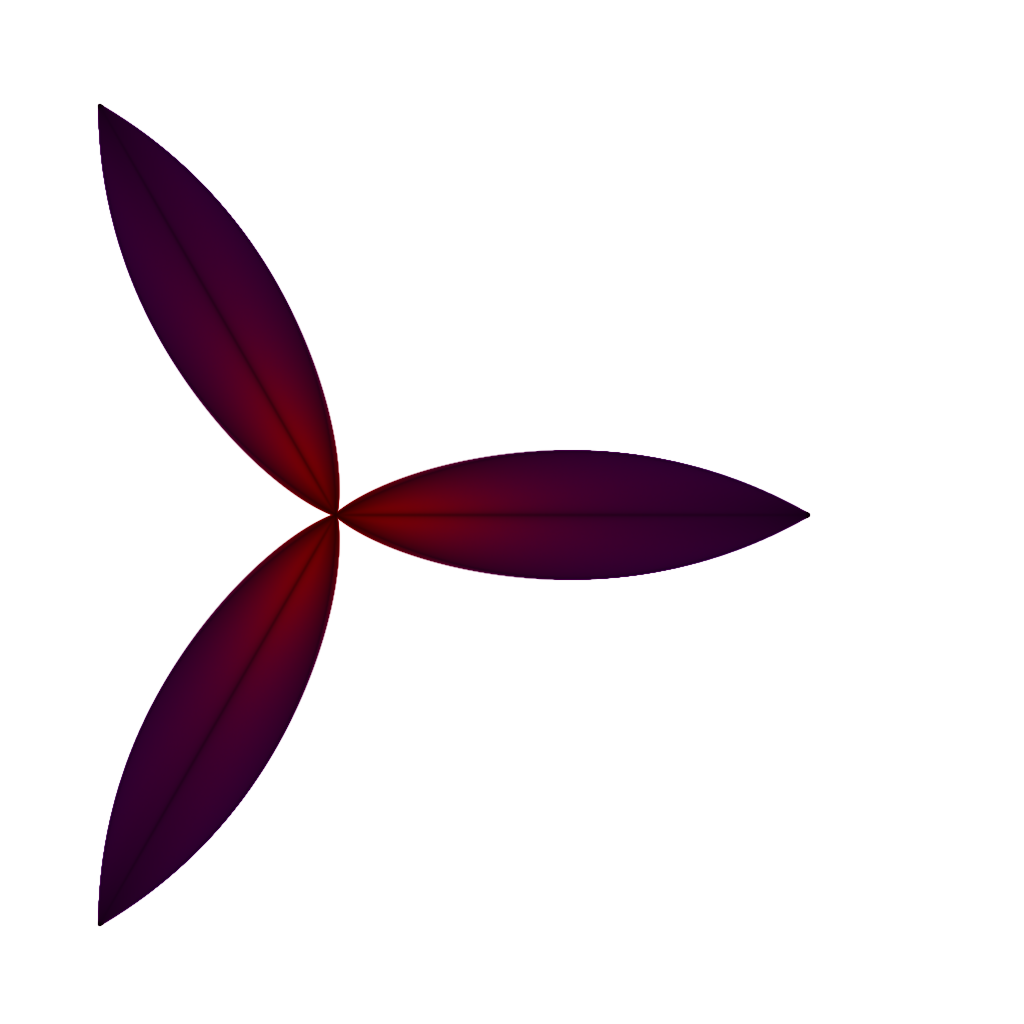} 
\caption{Monte--Carlo approximation of $\minvset{CH}$ for $T = (z^3-1)\diffz + 3z^2$. At first glance it can appear that $\minvset{CH}$ has boundary given by circle arcs. This is not the case, one can use associated rays to prove this and the actual $\minvset{CH}$ contains these circle arcs but is somewhat bigger. However, we have not been able to find an explicit description of $\minvset{CH}$ in this situation.}
\label{fig:triFlower}
\end{figure}

\begin{appendices}

\section{Appendix. Rational vector fields in $\setRS$ and their curves of inflections}\label{sec:ratfields}

For convenience of our readers, we include in this section some basic information about rational vector
fields in $\setRS$ which is frequently used in the paper.
Our exposition follows mainly the recent preprint~\cite{DiasGarijo2020x}.

\subsection{Separatrices and separatrix graphs of rational vector fields}\label{sec:septrix}

Consider the first order differential equation 
\begin{equation}\label{eq:rationaldiff}
\dot{z}(t)=\frac{dz}{dt}=-R(z),
\end{equation}
where $-R(z)=-\frac{Q(z)}{P(z)}: \setRS \to \setRS$ is a rational map. With this differential equation we associate 
the rational vector field $v_R=-R(z) {\partial_z}$. 
An important notion in this context is that of a \emph{separatrix} of $v_R$. 
Namely, separatrices are solutions of \eqref{eq:rationaldiff} whose maximal domain 
of definition is strictly smaller than $\mathbb{R}$. 

To be precise, let $\phi(t,\eta)$  be a solution of \eqref{eq:rationaldiff} with 
initial condition $\phi(0,\eta)=\eta$. 
We assume that $\phi(\cdot,\eta): (t_{min},t_{max})\to \setRS$ is defined on its maximum interval of definition. 
Then $\phi(t,\eta)$ is a \defin{separatrix} (of $v_R$) if at least one of $t_{min},t_{max}$ is finite. 

We will often use the notation $\phi(t)$ instead of $\phi(t,\eta)$ when it is 
unnecessary to specify  $\eta$. If $\lim_{t\to t_{min}}\phi(t)=z_*$, 
we say that $\phi(t)$ is \defin{$z_*$-starting} and if $\lim_{t\to t_{max}}\phi(t)=z_*$
then $\phi(t)$ is \defin{$z_*$-ending}. 
In general, if $z_*$ is not specified, we say that $\phi(t)$ is \defin{$P$-starting} resp.~\defin{$P$-ending}. 
If $t_{min}$ is finite then $\lim_{t\to t_{min}}\phi(t,\eta)$ is a pole  of $R(z)$ (equivalently, a saddle point of $v_R$). 
Similarly, if $t_{max}$ is finite then $\lim_{t\to t_{max}}\phi(t,\eta)$ is a pole of $R(z)$.
 Therefore there exist four distinct types of separatrices:

\begin{enumerate}
    \item $t_{min}$ is finite and $t_{max}$ is 
    infinite $\rightarrow$ $\phi(\eta,t)$ is called an \defin{outgoing} separatrix.
    
    \item $t_{min}$ is infinite and $t_{max}$ is 
    finite $\rightarrow$ $\phi(\eta,t)$ is called an \defin{ingoing} separatrix.
    
    \item $t_{min}$ and $t_{max}$ are both finite, but
    \[
    \lim_{t\to t_{min}}\phi(\eta,t)\neq \lim_{t\to t_{max}}\phi(\eta,t),
    \]
$\phi(\eta,t)$ is called a \defin{heteroclinic} separatrix.

    \item $t_{min}$ and $t_{max}$ are both finite and
    \[
    \lim_{t\to t_{min}}\phi(\eta,t)= \lim_{t\to t_{max}}\phi(\eta,t),
    \]
$\phi(\eta,t)$ is called a \defin{homoclinic} separatrix.
\end{enumerate}

The collection of separatrices of $v_R$ subdivides $\setRS$ into disjoint open domains
in each of which integral trajectories of $v_R$  have similar properties.
These domains can be of the  following $4$ different types: 
center zone, annular zone, parallel and elliptic zones defined as follows.

A \emph{center zone} is a simply connected region that contains a zero 
to $-R(z)$ which is a center. The integral trajectories of $v_R$ in center zones
are periodic orbits with the same period.

An \emph{annular zone} is a  doubly connected region in which integral
trajectories are periodic orbits with the same period and there exists a
positive number $L$ such that all these integral trajectories have length greater than $L$.

 Both \emph {elliptic and parallel zones} are  simply connected regions. 
The integral trajectories in elliptic zones have the property  
that all $\omega$- and $\alpha$-limits coincide with  a critical
point $z_0$ of $\phi(\eta,t)$ solving \eqref{eq:rationaldiff}. 
For parallel zones, all $\omega$-limits are equal to one critical point $z_0$ 
while all $\alpha$-limits are equal to another 
critical point $z_1$, see illustrations in \cref{fig:vectorFieldTypes}. 
\begin{figure}[!ht]
\begin{center}
\includegraphics[width=0.18\textwidth,page=1]{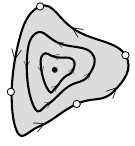}\hskip1.5cm
\includegraphics[width=0.18\textwidth,page=2]{separatrix-regions}

\includegraphics[width=0.18\textwidth,page=3]{separatrix-regions}\hskip1.5cm
\includegraphics[width=0.18\textwidth,page=4]{separatrix-regions}
\end{center}
\caption{Illustrations of different types of zones: center (top left), 
annular (top right), parallel (bottom left), and elliptic (bottom right).%
}\label{fig:vectorFieldTypes}
\end{figure}

\smallskip
As mentioned above, the poles of $-R$ are saddle points of $v_R$. 
More exactly, if $z_*$ is a pole of order $k$ of $-R$,  
then it is a saddle point of $v_R$  with $k+1$ $z_*$-starting 
and $k+1$ $z_*$-ending separatrices. Furthermore, if one goes
around $z_*$ once along the boundary of a sufficiently small,
convex neighborhood of $z_*$, one traverses segments of the
boundary where a trajectory starts and ends in $z_*$ consecutively. 
Additionally, the angles between two adjacent separatrices are  
equal to  $\frac{2\pi}{2k+2}$, see~\cref{fig:locPort}. 
More specifically, we have the following statement.

\begin{proposition}[See~\cite{needham}]\label{prop:dirSeptrix}
Suppose that $R(z)$ is a rational function having a pole of order $k$ at $z_0$ and 
\[
R(z)=\frac{c}{(z-z_0)^k}+f(z)
\]
where $\lim_{z\to z_0}f(z)(z-z_0)^k=0$. 
Then the directions of the $z_0$-starting separatrices
of $R(z) {\partial_z}$ at $z_0$ are given by the solutions to $z^{k+1}= c$
and the directions of the $z_0$-ending separatrices
are given by the solutions to $z^{k+1}= -c$.
\end{proposition}

\begin{figure}[!ht]
\begin{center}
\includegraphics[width=0.25\textwidth,page=1]{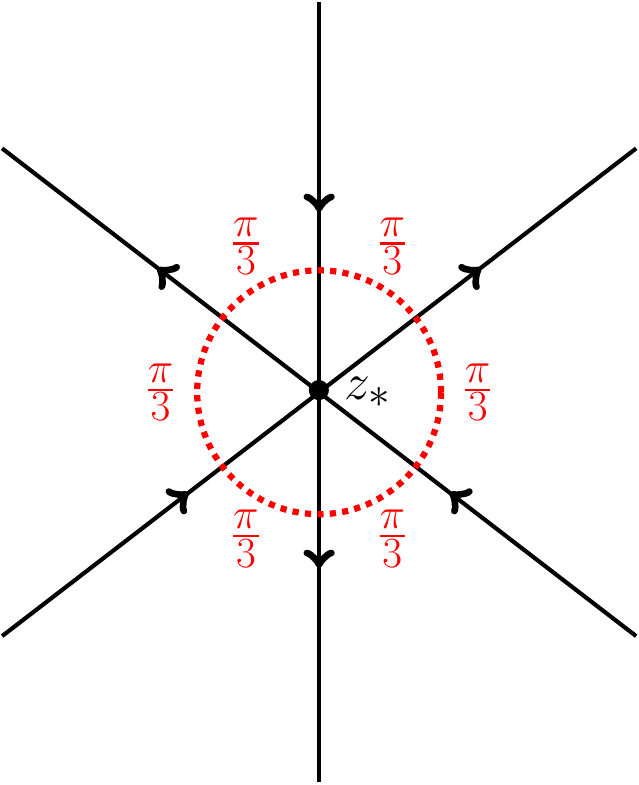}
\end{center}
\caption{Schematic description of separatrices near 
a pole $z_*$ of order $k=2$.}\label{fig:locPort}
\end{figure}

Moreover, if $z_0$ is a zero of $-R(z)$ of order $1$, then the phase portrait near $z_0$
is either a sink (if $-\Re R'(z) <0$), a source (if $-\Re R'(z)>0$), or 
a center (if $-\Re R'(z)=0$). If $z_0$ is zero of $-R(z)$ of order $k$ then 
the phase portrait near $z_0$ is the union of $2(k-1)$ elliptic sectors. 
Examples of phase portraits near singularities of analytic vector fields are shown in \cref{fig:vectorFieldSing}.

\begin{figure}[!ht]
\begin{center}
\includegraphics[width=0.2\textwidth,page=1]{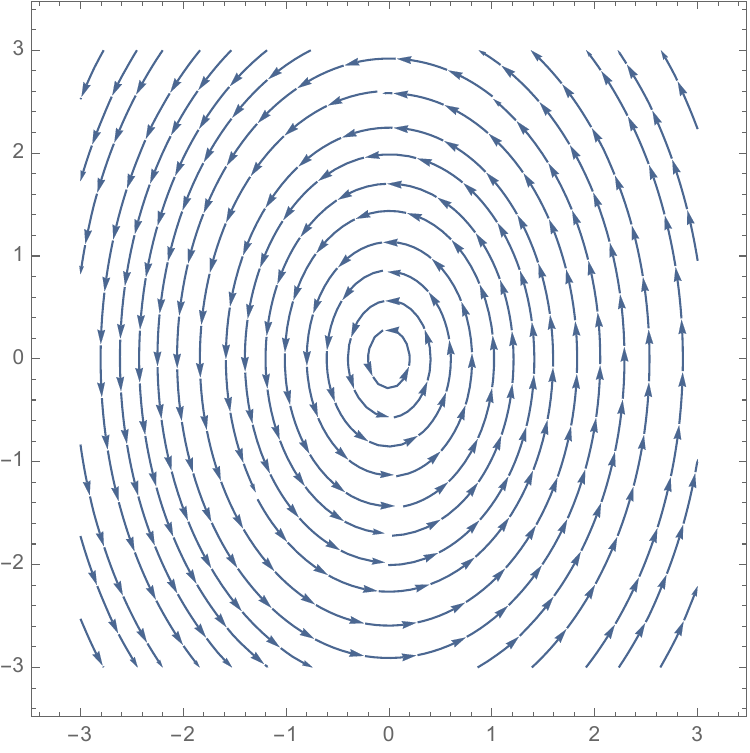}
\includegraphics[width=0.2\textwidth,page=1]{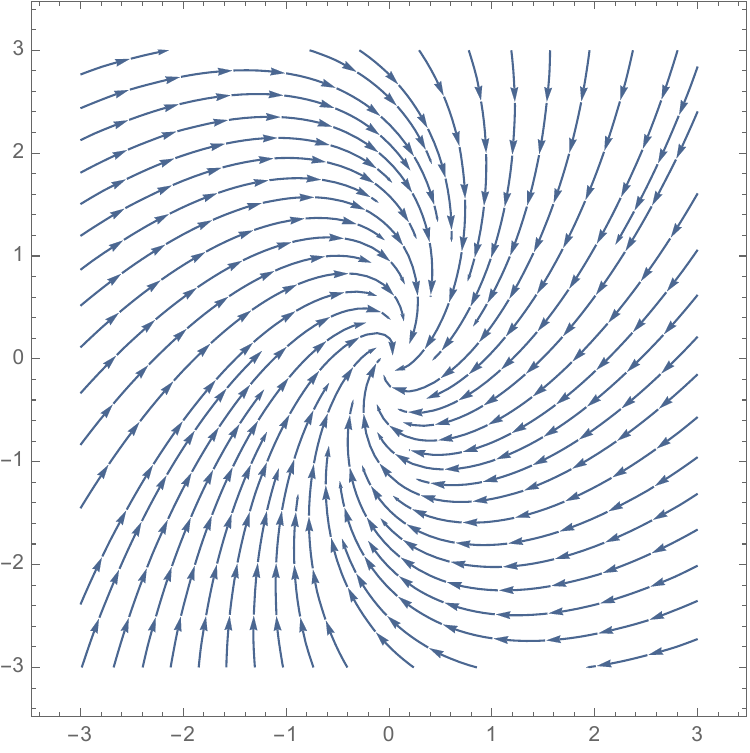}
\includegraphics[width=0.2\textwidth,page=1]{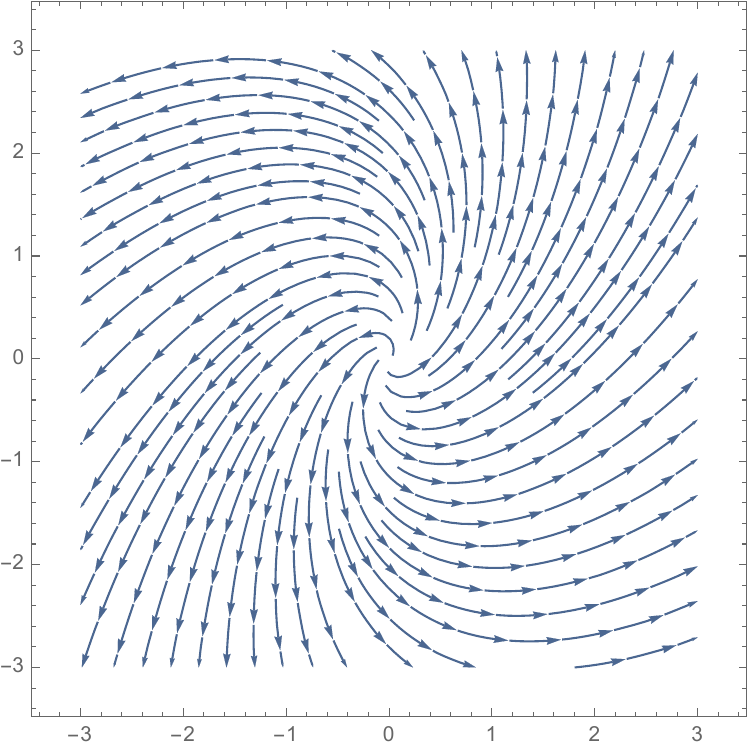}

\includegraphics[width=0.2 \textwidth,page=1]{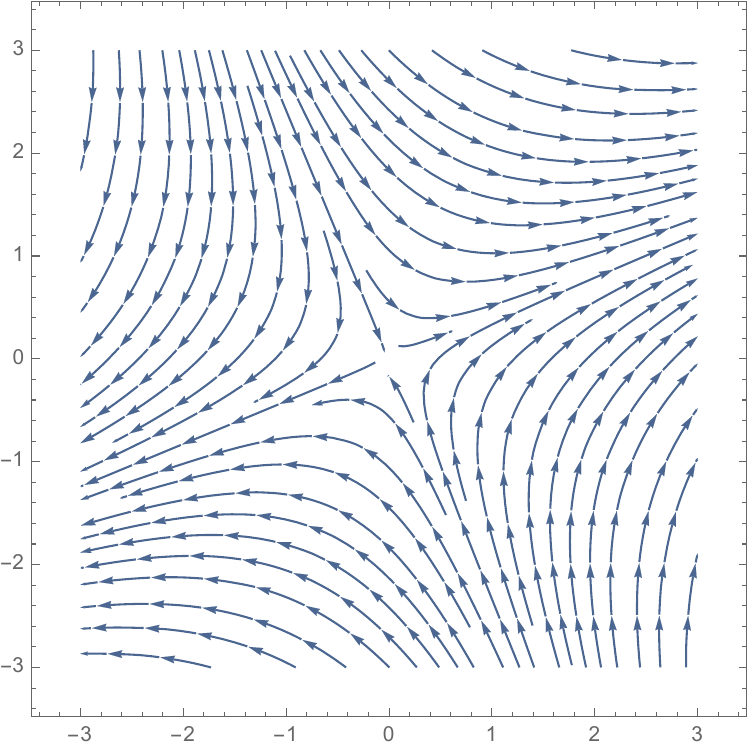}
\includegraphics[width=0.2 \textwidth,page=1]{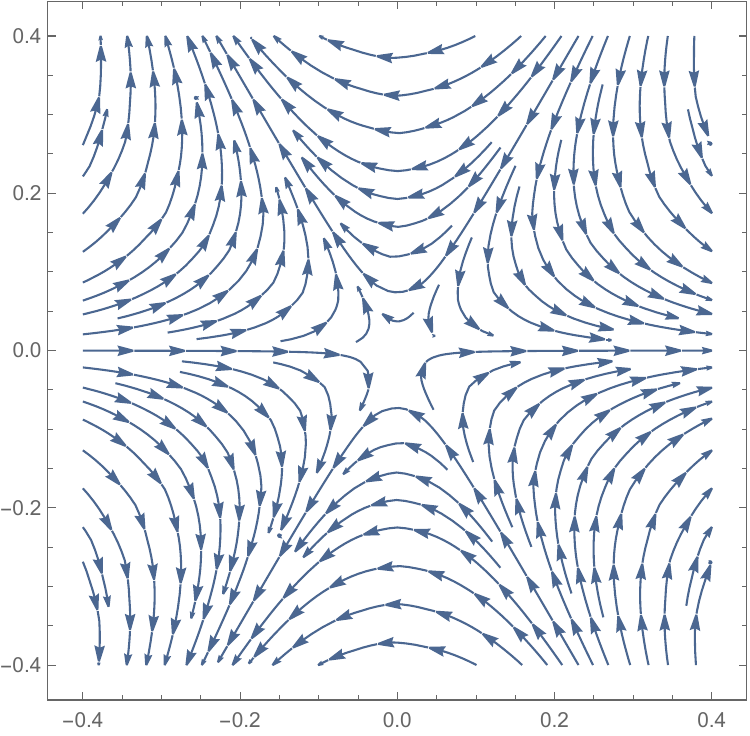}
\includegraphics[width=0.2 \textwidth,page=1]{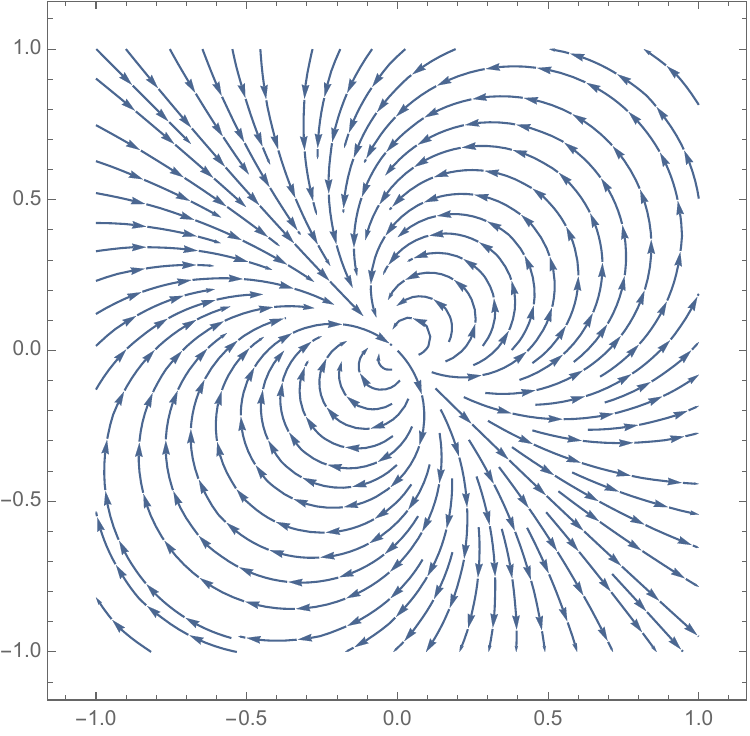}
\end{center}
\caption{Singular points of analytic vector fields (left-to-right; top-to-bottom): center, sink, source,  saddle, saddle of order $2$, and zero of order $2$.}\label{fig:vectorFieldSing}
\end{figure}

For a given rational vector field,  the union of the closures of its separatrices 
is called the \emph{separatrix graph}.  Given a rational function $R$ as above, 
let us denote by $\Gamma_R$ the separatrix graph of $v_R$. As mentioned above, there are four different types 
of zones for a rational vector field $v_R$ and its separatrix graph $\Gamma_R$ constitutes  
the boundary of these zones. Different topological  possibilities for separatrix
graphs of rational vector fields in $\setRS$ are completely characterized in \cite{DiasGarijo2020x}. 
In this paper, we will, in particular, use the property that the complement of $\setRS\setminus \Gamma_R$ is
the union of simply and doubly connected domains and that the
boundary of an annular or a  center zone consists of homo/heteroclinic separatrices.

\subsection{Separatrices and $T_{CH}$-invariant sets} 
In this subsection we discuss one general property of $T_{CH}$-invariant sets which will be especially useful when studying irregularity of these  sets. 
Namely, we will  show that the root-starting separatrices of the vector field $-R(z){\partial_z}$ (if they exist) belong to any $T_{CH}$-invariant set. \cref{prop:invSetClosedUnderForwardIntegralCurves} and its proof are  quite similar to  statements which can be found in several texts on numerical methods for differential equations, such as \cite{HairerErnst1993SODE}. (For a stronger result see \cite{Frank1981TheCO}.)  Since we were not able to find the exact statement of \cref{prop:invSetClosedUnderForwardIntegralCurves}  in the literature  and for convenience of the readers, we include it and its proof below. $||$ denotes the absolute value in the Euclidian metric.

\begin{proposition}\label{prop:invSetClosedUnderForwardIntegralCurves}
Let $z(t)$ be a solution to the initial value problem 
\[
\begin{cases}
z'(t)=-R(z)\\
 z(0)=z_0
\end{cases}
\] 
for which there exist $M>0$, $\epsilon>0$, and  $s_0>0$ such that if $|z(t)-z|<\epsilon$ for
some $t\in [0,s_0]$ then $|R(z)|<M_0$.
Then, for any $s_1\in [0,s_0]$ and $\delta>0$, there exist $t_0>0$, 
a sequence $\{z_k\}_{k=0}^\infty$ where each $z_k$ satisfies 
\[
t_0 Q(z_k)+(z_k-z_{k-1})P(z_k)=0
\]
and  a positive integer  $N$ such that $|z(s_1)-z_N|<\delta$.

\end{proposition}
\begin{proof}
Note first that the assumption that  $|R(z)|<M_0$ for all $z$ such that $|z(t)-z|<\epsilon$ for
some $t\in [0,s_0]$ implies, since $R$ is rational, that there is $M$ such that  $|R(z)|<M$, $|R'(z)|<M$, $|R''(z)|<M$ if $|z(t)-z|<\epsilon$ for
some $t\in [0,s_0]$.
Next we note that if $z(t)$ is  not a separatrix then the above criteria of boundedness of $R$ in a neighborhood of $z(t)$ is fulfilled with $s_0=\infty$. 
If it is a separatrix, but $z_0$ not a root of $P$ while $\lim_{t\to s_2}z(t)$ is a root of $P$
for some $s_2>0$ then the above criteria is fulfilled for all $s_0\in (0,s_2)$ provided that $\epsilon>0$ is chosen sufficiently small and $M$ is sufficiently large. 
Now, let $z(t)$ be as in the statement of the proposition and define the sequence $\{z_k\}$ as follows. 
Consider the map 
\[
\varphi(z)=z+t_0R(z). 
\]
(We may pick small enough $t_0$  so that $t_0M<1$.)  Then $\varphi(z)$ has a non-vanishing 
derivative in the $\epsilon$-neighborhood $U$ of the curve segment $z([0,s_0])$ and therefore by the
inverse function theorem, $\varphi$ has an  inverse $\varphi^{-1}: \varphi(U) \to U$.

We define 
\[
z_k:=\varphi^{-1}(z_{k-1})
\]
and note that $z_k$ is one of the solutions to
\[
t_0Q(z)+(z-z_{k-1})P(z)=0.
\]
Namely, we choose $z_k$ to be such a solution  that if we fix $z_{k-1}$ and take the limit $t_0\to 0$, then  $z_k\to z_{k-1}$. 
Now, we will inductively approximate $z((k+1)t_0)$.
To that end, we consider 
\begin{equation}\label{eq:approximation}
    z(kt_0+\Delta t)=z(kt_0)+\Delta tR(z(kt_0)+\Delta t))-r_{k}
\end{equation} where $r_k$ is what is usually referred to as the \emph{local residue}.
We can estimate its absolute value by using the Taylor expansion of $z(kt_0+\Delta t)$ and $z'(kt_0+\Delta t)=-R(kt_0+\Delta t)$ around $\Delta t=0$. Namely, 
\[z(kt_0+\Delta t)=z(kt_0)+\Delta t z'(kt_0)+\frac{\Delta t^2}{2} z''(kt_0)+O(\Delta t^3);\]
\[z'(kt_0+\Delta t)=z'(kt_0)-\Delta tR'(z(kt_0))+O(\Delta t^2)\]
\[=z'(kt_0)+\Delta tz''(z(kt_0))+O(\Delta t^2).\]
We substitute these two relations in \eqref{eq:approximation}, let $\Delta t=t_0$ and use the equality $z'(kt_0+\Delta t)=-R(kt_0+\Delta t)$ to obtain
\begin{equation}\label{eq:sizeoflocalres}
|r_{k}|=\frac{|t_0^2 z''(kt_0)|}{2}+O(t_0^3)=\frac{|t_0^2 M|}{2}+O(t_0^3)<|t_0^2 M|\end{equation}
for small $t_0$.

Next, as $t_0M<1$, $\varphi^{-1}$ is Lipschitz in $V$. 
Denoting the Lipschitz constant of $\varphi^{-1}$ in $V$ by $L$ note 
that $L\leq \frac{1}{1-t_0M}$. 

Now, recall that $z_0$ is in $U$ and we suppose that $z_l\in U $ for $ l=1,\dots,k-1$. Then
\[
\varphi(z_{k})=z_{k-1}, \qquad \varphi(z(kt_0))=z((k-1)t_0)-r_{k-1},
\]
and thus 
\[
|z_k-z(kt_0)|<L|z_{k-1}-z((k-1)t_0)+r_{k-1}|.
\]
Using the notation $e_k:=|z_k-z(kt_0)|$  and \eqref{eq:sizeoflocalres}, we get  
\[
e_{k}\leq \frac{|e_{k-1}|+|r_{k-1}|}{1-t_0M}=(|e_{k-1}|+|r_{k-1}|)\sum_{n=0}^\infty(t_0M)^n<|e_{k-1}|(1+2t_0M)+2|r_{k-1}|
\]
provided we pick $t_0$ small enough. 

Using \cite[Eq. 1--12]{henrici} together with $|e_{k}|\leq (1+2t_0M)e_{k-1}+2|r_{k-1}|$ and $e_0=0$ 
we get 
\[
e_k\leq 2\max |r_k|\frac{e^{2kt_0M}-1}{2t_0M}.
\]
Hence,
\begin{equation}\label{eq:finalapprox}
    |z_k-z(kt_0)|< Mt_0(e^{2kt_0M}-1). 
\end{equation}
In particular, fix $\delta>0$ and let $\delta_0=\min(\delta,\epsilon)$ and $s_1\in [0,s_0]$. For any $N\in \mathbb{N}$, let $t_{0,N}$ be such that $t_{0,N}N=s_1$. Now pick $N$ large enough so that $Mt_{0,N}( e^{N2t_{0,N}M}-1)<\delta_0$. Then we have that $z_k\in U$ for $k\leq N$. Using $t_0=t_{0.N}$ we may thereby  apply \eqref{eq:finalapprox} and by induction conclude that $|z_N-z(s_1)|<\delta$ as asserted.
\end{proof}

\begin{definition}\label{fwdTrajectories}
If the initial value problem 
\begin{equation}\label{eq:diffEq}
\begin{cases}
z'(t)=-R(z)\\
z(0)=z_0
\end{cases}
\end{equation}
has a solution $z(t)$ for $t\in [0,s_0)$ ($s_0$ may equal $\infty$), then $\overline{z([0,s_0))}$ 
is called the \emph{forward trajectory} of $z_0$.
\end{definition}
\cref{prop:invSetClosedUnderForwardIntegralCurves} states that if $S\subset \setC$ 
is a $T_{CH}$-invariant set then the forward trajectories of 
all $u\notin \zeros(P)$ are contained in $S$, provided $z(t)\in \bC$ for all $t\in [0,s_0)$. Indeed, for any $t\in [0,s_0)$ we can find an open neighborhood of $z([0,t])$ such that $|R(z)|$ is bounded. Then, if $z_0\in S$ the proposition shows that for any $\delta>0$ we can find a point in $S$ of distance less than $\delta$ from $z(t)$. Since $S$ is closed, the statement follows.

The next important definition is that of a separatrix.
\begin{definition}\label{def:septrix}
A \defin{separatrix} is a solution $\phi(t)$ of \eqref{eq:diffEq} 
whose maximal domain 
of definition $(t_{min},t_{max})$ is a proper subset of $\mathbb{R}$. 
If $t_{min}$ is finite, then $\phi(t)$ is called a 
\defin{$P$-starting separatrix}
and if $t_{min}$ is finite it is called a \defin{$P$-ending separatrix}. 
Furthermore, if $\lim_{t\to t_{min}}\phi(t)=z_0\in \zeros(P)$, 
it is called \defin{$z_0$-starting},
and if $\lim_{t\to t_{max}}\phi(t)=z_0\in \zeros(P)$, 
it is called \defin{$z_0$-ending}.
\end{definition}
When we refer to $P$-starting separatrices without 
mention of the vector field, we will always mean 
the $P$-starting separatrices of $-R(z)\partial_z$. 
The next statement shows that the $P$-starting separatrices
of $-R(z){\partial_z}$, when they exist, also lie in any $T_{CH}$-invariant set $S$.

\begin{proposition}\label{pro:separatrixIsInInv}
Let $S\subset \bC$ be a  $T_{CH}$-invariant set for an operator $T$ given by \eqref{eq:1st} 
such that  $\deg P(z)\ge 1$  and $Q(z)$ is not identically vanishing and $P$ and $Q$ have no zeros in common. Suppose that $z_0 \in \zeros(P)$ 
and let $\phi$ be a $z_0$-starting separatrix of $-R(z) {\partial_z}$ such that $\lim_{t\to 0}\phi(t)=z_0$.
Then $\overline{\phi((0,s_0))} \in S$ for all $s_0$ such that $\phi(t)\in \bC$  for $0< t<s_0$.
\end{proposition}
\begin{proof}

Near $z_0$ the vector field $-R(z) {\partial_z}$ looks
like $\frac{\alpha}{z^k} {\partial_z}$ for some $k\geq 1$ and $\alpha\in \bC$, where $k$ is the multiplicity of $z_0$.
By Proposition \ref{prop:rootTrajectoryDirections} there is a $t$-trace $\gamma_u(t)$ with $u=z_0$, $\gamma_u(0)=z_0$ which is tangent to $\phi$. We may suppose that $\gamma_u(t), t\in (0,t_0)$ is not a subset of $\phi(0,t_1)$ for some $t_1>0$ and sufficiently small $t_0$, because in this case the statement directly follows from  \cref{prop:invSetClosedUnderForwardIntegralCurves}.
Then, in a sufficiently small neighborhood of $z_0$, we can assume that $\gamma$ is confined 
in a curvilinear cone---both sides of which being separatrices for $-R(z) {\partial_z}$---
and that additionally $\gamma$ splits this cone into two domains, see Figure~\ref{fig:fig-separatrix-included}.
\begin{figure}[!ht]
 \includegraphics[width=0.4\textwidth]{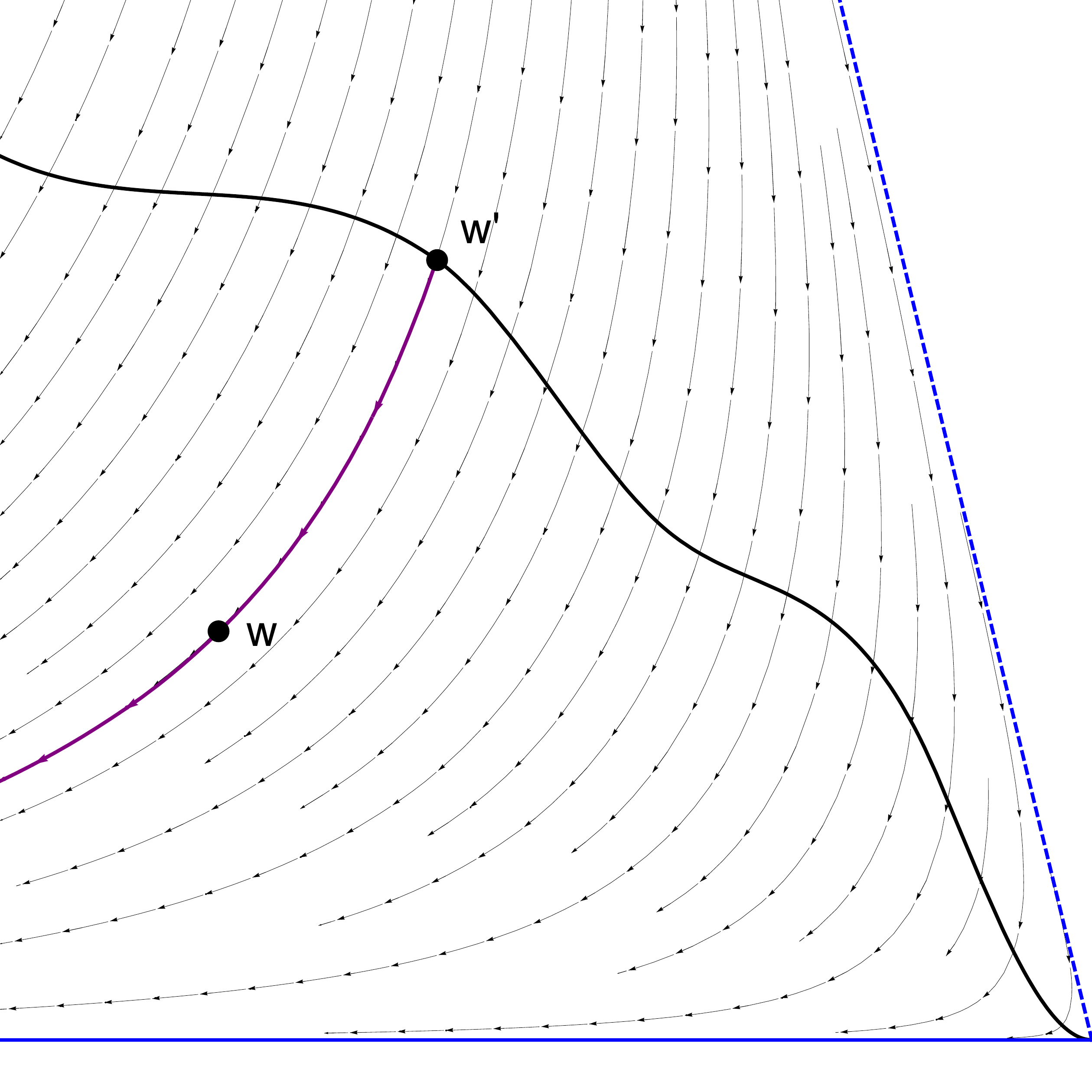}
 \caption{The separatrix $\phi$ (solid blue). In this particular example, $k=2$.
 }\label{fig:fig-separatrix-included}
\end{figure}
Now for any point $w$ sufficiently close to $z_0$ in the domain bounded by $\phi$ and $\gamma_u(t)$
there is a point $w'$ on $\gamma_u(t)$, such that the integral curve of $-R(z) {\partial_z}$
originating at $w'$ passes through $w$.
By Lemma~\ref{prop:invSetClosedUnderForwardIntegralCurves}, we obtain 
that $w \in S$. Since $S$ is closed and $w$ can be chosen arbitrary close to $\phi$, we obtain  that $\phi \subseteq S$.
\end{proof}

\smallskip
The last statement in this part is required for classification of fully irregular $T_{CH}$-invariant sets. 

\begin{proposition}\label{prop:qRootSinks}
Suppose that $T$ is not of the form
\begin{itemize} 

\item $P$ vanishes identically 

or

\item $P$ is a non-zero constant  and $Q$ has degree 1.
\end{itemize}
Then the simple roots of $Q$ 
belonging to the boundary of a  $T_{CH}$-invariant set $S$
are poles of the $1$-form  $-\frac{1}{R} dz$ with negative residues. 
\end{proposition}
\begin{proof}
Let $z^*$ be a simple zero of $Q$.

First assume that $z^*$ is a center of the vector field $-R(z) {\partial_z}$.
There is a $t$-trace going to $z^*$ that is contained in $S$.
Since the forward trajectories of points on this $t$-trace near $z^*$ are closed 
loops and they are contained in $S$, 
 we obtain $z^*\notin \partial S$.

Next, suppose that $z^*$ is a source of $-R(z) {\partial_z}$ lying on the 
boundary of $S$. Then there are points close to $z^*$ included in $S$;  
 a simple geometric argument involving associated rays
gives that there exists $z\notin S$ such 
that $z+tR(z)\in S$, contradicting \cref{th:charact}.

Finally, suppose that $z^*$ is a sink of $-R(z) {\partial_z}$, 
but the $1$-form $-\frac{1}{R} dz$ has a non-real residue at $z^*$. 
Then any integral curve approaching $z^*$ rotates about $z^*$ infinitely many times.
Now let $z_1\in S$ be a point close enough to $z^*$ and such that the forward 
trajectory of $z_1$ has $z^*$ as its limit. By picking  a neighborhood  $U$ of $z^*$ 
small enough, we have that for all $z\in U$, the ray $z+tR(z)$ intersects the 
forward trajectory of $z_1$. Since the forward trajectory of $z_1$ is contained in $S$, 
\cref{th:charact} implies that $U\subset S$.
\end{proof}

%

\subsection{Inflection points of trajectories of an analytic vector field}\label{ssec:inflections} 

In what follows we will need a description of the set of inflection points of trajectories 
of an analytic vector field in the complex plane.  
Let $W(\zvec)\partial_{\zvec}$ be a vector field where $ W(\zvec): \setC \to \setC:$
$W(\zvec) = (u(x,y), v(x,y))$. Typically, for real-analytic $W(\zvec)\partial_{\zvec}$,  there
exists a curve $\infl_W \subset \setR^2$ consisting of all points at which
trajectories of $W(\zvec)\partial_{\zvec}$ have inflections. 
Our nearest goal is to describe  $\infl_W$ when $W(\zvec)\partial_{\zvec}$ is given by
the real and imaginary parts of a complex-analytic function.  

Suppose now that $\zvec$ is an inflection point of a trajectory of $W(\zvec)\partial_{\zvec}$.
This means that the vector field at the  point $\zvec + \epsilon W(\zvec)$ for $\epsilon$ small 
is close to  $W(\zvec)$.
In other words,  $W(\zvec)$ and $W(\zvec + \epsilon W(\zvec))$
are almost parallel. 
More exactly, we require that 
\[
 \lim_{\epsilon \to 0} 
 \frac{1}{\epsilon}
 \begin{vmatrix}
  u(x,y) & v(x,y) \\
  u(x + \epsilon u(x,y),y + \epsilon v(x,y)) & v(x + \epsilon u(x,y),y + \epsilon v(x,y))
 \end{vmatrix}
=0.
\]
Expanding the determinant and applying l'Hospital's rule, we end up with
\begin{equation}\label{eq:generalFormula}
u\,  ( v'_x u + v'_y v )
-
v\, ( u'_x u + u'_y v ) = 0.
\end{equation}

\noindent In the special case when
\[
W(\zvec) = (u(x,y), v(x,y)), \quad
u(x,y) + iv(x,y) = R(\zvec), \; \quad  \zvec = x+i y \in \setC,
\]
and $R(\zvec) : \setRS \to \setRS$ is analytic/meromorphic one can say more. 
In this case we will write $\infl_R$ instead of $\infl_W$ 
and use notation $z$ instead of $\zvec$.

\begin{lemma}\label{lm:infl} 
For an analytic function $R(z)$, the curve of inflections 
of the vector field $W=(\Re R, \Im R)$ satisfies the condition $\Im R^\prime=0$. 
\end{lemma}

\begin{proof} 
Recall that the Cauchy--Riemann equations for $R(z)$ have the form
\[
 \frac{\partial u}{\partial x} = \frac{\partial v}{\partial y}, \qquad 
 \frac{\partial v}{\partial x} = -\frac{\partial u}{\partial y}.
\]
If we apply them to \eqref{eq:generalFormula} assuming that either $u$ or $v$ is non-vanishing,
we end up with the simple condition
\[
\frac{\partial u}{\partial y} = \frac{\partial v}{\partial x} = 0
\]
which is equivalent to the requirement that $R^\prime(z)$ attains a real value.  
If $u=v=0$ we get a zero of $R$ which is an uninteresting case.
\end{proof}

Set $\flex(x,y) \coloneqq - \frac{\partial u}{\partial y}$, i.e., let 
  $\flex(x,y)$ be the imaginary part of $R^\prime(z)$. 
  By \cref{lm:infl}, $\infl_R$ is the
  locus of $\flex(x,y)=0$. 
Note that for the rational function $R(z)=\frac{Q(z)}{P(z)}$, one gets 
\begin{equation}
 \flex(x,y) = \frac{1}{|P|^4}\Im\left( (PQ'- QP' )\overline{P}^2 \right).
\end{equation}

For further use let us denote by $\infl_R^+\subset \infl_R$ 
the portion where $R^\prime>0$ and by $\infl_R^-\subset \infl_R$ the portion where $R^\prime<0$.

\smallskip
The singularities on the curve $\infl_R$ correspond to the points at which both $R^\prime(z)$ and 
$R^{\prime\prime}(z)$ attain real values.  In particular, since 
\[
R^\prime(z)=\frac{Q^\prime P-QP^\prime}{P^2}\quad \text{and} \quad 
R^{\prime\prime}(z)=\frac{Q^{\prime\prime}P^2-QPP^{\prime\prime} - 2PP^\prime Q^\prime-2Q(P^\prime)^2}{P^3}
\]
then zeros of $P$ are the singular points of $\infl_R$ which can be seen in \cref{fig:bluecurve}.
Moreover, for generic rational $R(z)$,  its poles are the only singular points of $\infl_R$.

\begin{lemma}
Any bounded connected component of the solutions to $\flex(x,y)=0$
contains at least one pole of $\RR(z)$.
\end{lemma}
\begin{proof}
Let $\Gamma$ be such a bounded connected component, and assume that $\Gamma$
does not contain any pole of $\RR$.

We can then let $D_\epsilon \supset \Gamma$ be the domain consisting of all points 
of distance at most $\epsilon$ from $\Gamma$. By choosing $\epsilon$ sufficiently small,
we can guarantee that $D_\epsilon$ is  a bounded domain which does not contain poles of $\RR$.
Moreover, we can ensure that $D_\epsilon$ does not intersect any other connected component of $\flex$.

But since $\flex$ is the imaginary part of the analytic function $R'(z)$ 
it is harmonic in $D_\epsilon$, and so $|\flex(x,y)|$ must be $0$ on the boundary of $D_\epsilon$.
This violates the choice of $\epsilon$, and our assumption must have been false.
\end{proof}

One can also ask under what conditions on $R$ the curve $\infl_R$ is compact.   
In our main application, we have that $\deg Q=\deg P+1$. As before, expand $R(z)\partial_z$ at $\infty$: 
\[
\frac{Q(z)}{P(z)}\partial_z =\left( \lambda z + O\left( 1\right)\right)\partial_z.
\]
The following claim holds. 
\begin{lemma}\label{lem:blueCurveCompact}
The set  $\infl_R$ is compact whenever $\Im \lambda \neq 0$.
\end{lemma}

\begin{proof}
Since
\[
\infl_R = \{ z\in \setC : \Im R^\prime(z)=0 \}
= \{ z\in \setC : \Im \lambda  +  O\left(\frac{1}{z}\right)=0 \},
\]
we see that $\Im R^\prime(z) \approx \Im \lambda \neq 0$ whenever $|z|$ is large.
Hence, $\infl_R$ is bounded.
\end{proof}

\begin{example}\label{ex:5}
 In \cref{fig:bluecurve}, we show  $\infl_R$ for the vector field $R(z) {\partial_z}$ with 
 \[
 \RR(z) \coloneqq  (1+i)z+\frac{2}{z}+\frac{1}{z-i}+\frac{4}{z-(1+i)}.
 \]
 \begin{figure}[H]
 \begin{center}
 \includegraphics[width=0.5\textwidth]{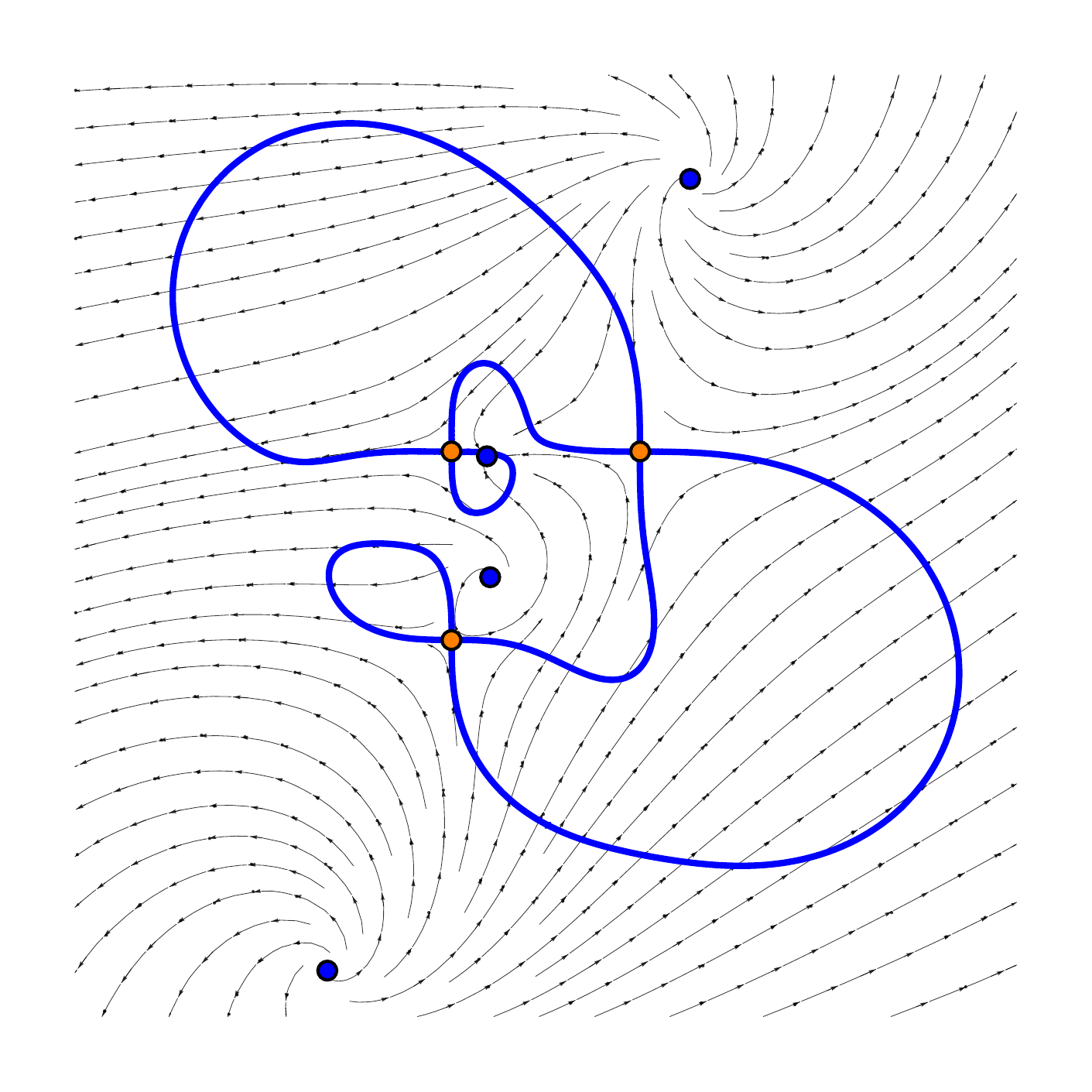}
 \end{center}\caption{ Example of a curve of inflections. }\label{fig:bluecurve}
\end{figure}
Here 
\[
 \flex(x,y) = -1 -\frac{4 x y}{\left(x^2+y^2\right)^2}-\frac{x (2
   y-2)}{\left(x^2+y^2-2 y+1\right)^2}-\frac{(4 x-4)
   (2 y-2)}{\left(x^2-2 x+y^2-2 y+2\right)^2}.
\]
\end{example}

\begin{proposition} 
For generic $R(z)$, the direction of the curve of inflections at a pole $z_0$ of $R$ 
coincides with the direction of separatrices emanating from $z_0$. 
\end{proposition}
\begin{proof}
Given an arbitrary rational $R(z)$, suppose that it has
a pole $z_0$ of order $k$. 
Without loss of generality, let us assume that $z_0=0$.
We can then write 
\[
R(z)=\frac{c}{z^k}+ f(z)
\]
where $f(z)$ is rational and $\lim_{z\to0}f(z)z^k=0$. 
Hence, the derivative of $R(z)$ in a punctured neighborhood of $0$ is
\[
R'(z)=\frac{-k c}{z^{k+1}}+f'(z)
\]
where
$\lim_{z\to 0}f'(z)z^{k+1}=0$. 
Now, expanding $c=a+bi$ we get  
\begin{align*}
\Im R'(z) &=\Im\left(\frac{-k c}{z^{k+1}}+f'(z)\right) \\
&=\frac{ka}{z^{k+1}}\sin((k+1)\arg z)-\frac{kb}{z^{k+1}}\cos((k+1)\arg z)+\Im f'(z).
\end{align*}
Setting this expression equal to 0 and multiplying both sides by $z^{k+1}$ we get 
\[
ka\sin((k+1)\arg z)-kb\cos((k+1)\arg z)+\Im f(z)z^{k+1}=0.
\]
Letting $|z|\to 0$ we obtain
\[
\lim_{z\to 0}(a\sin((k+1)\arg z)-b\cos((k+1)\arg z)=0.
\]
Comparing the above with \cref{prop:dirSeptrix}, we find that
this equation is solved precisely when $\lim_{z\to0}\arg z$
is equal to the argument of the separatrices emanating from $0$.
\end{proof}

\end{appendices}

\bibliographystyle{alphaurl}
\bibliography{firstOrderBib}

\end{document}